\documentclass{amsart}

\usepackage{enumerate, amsmath, amsfonts, amssymb, amsthm, wasysym, graphics, graphicx, xcolor, url, hyperref, hypcap,  shuffle, xargs, multicol, overpic, pdflscape, multirow, hvfloat, minibox, accents, array, xifthen, xparse}
\usepackage{marginnote}
\hypersetup{colorlinks=true, citecolor=darkblue, linkcolor=darkblue}
\usepackage[all]{xy}
\usepackage[bottom]{footmisc}
\usepackage{tikz}
\usepackage{tkz-graph}
\usetikzlibrary{trees, decorations, decorations.markings, shapes, arrows, matrix, calc, fit}
\graphicspath{{figures/}}

\title[The facial weak order and its lattice quotients]{The facial weak order and its lattice quotients}

\thanks{AD was partially supported by the French ANR grant EGOS~(12\,JS02\,002\,01) and a ISM Graduate Scholarship. CH was supported by NSERC Discovery grant {\em Coxeter groups and related structures}. VP~was partially supported by the French ANR grants EGOS~(12\,JS02\,002\,01) and SC3A~(15\,CE40\,0004\,01).}

\author{Aram Dermenjian}
\address[Aram Dermenjian]{LIX, \'Ecole Polytechnique, Palaiseau \& LaCIM, Universit\'e du Qu\'ebec \`A Montr\'eal (UQAM)}
\email{aram.dermenjian@gmail.com}

\author{Christophe Hohlweg}
\address[Christophe Hohlweg]{LaCIM, Universit\'e du Qu\'ebec \`A Montr\'eal (UQAM)}
\email{hohlweg.christophe@uqam.ca}
\urladdr{http://hohlweg.math.uqam.ca/}

\author{Vincent Pilaud}
\address[Vincent Pilaud]{CNRS \& LIX, \'Ecole Polytechnique, Palaiseau}
\email{vincent.pilaud@lix.polytechnique.fr}
\urladdr{http://www.lix.polytechnique.fr/~pilaud/}

\newtheorem{theorem}{Theorem}[section]
\newtheorem{corollary}[theorem]{Corollary}
\newtheorem{proposition}[theorem]{Proposition}
\newtheorem{lemma}[theorem]{Lemma}

\theoremstyle{definition}
\newtheorem{definition}[theorem]{Definition}
\newtheorem{example}[theorem]{Example}
\newtheorem{remark}[theorem]{Remark}

\newcommand{\R}{\mathbb{R}} 
\newcommand{\N}{\mathbb{N}} 
\newcommand{\Z}{\mathbb{Z}} 
\newcommand{\fS}{\mathfrak{S}} 
\newcommand{\cC}{\mathcal{C}} 
\newcommand{\precdot}{{\;<\!\!\!\!\cdot\;\,}} 

\newcommand{\set}[2]{\left\{ #1 \;\middle|\; #2 \right\}} 
\newcommand{\setangle}[2]{\left\langle #1 \;\middle|\; #2 \right\rangle} 
\newcommand{\bigset}[2]{\big\{ #1 \;|\; #2 \big\}} 
\newcommand{\ssm}{\smallsetminus} 
\newcommand{\dotprod}[2]{\langle \, #1 \; | \; #2 \, \rangle} 
\newcommand{\symdif}{\,\triangle\,} 
\newcommand{\eqdef}{\mbox{\,\raisebox{0.2ex}{\scriptsize\ensuremath{\mathrm:}}\ensuremath{=}\,}} 
\newcommand{\polar}{^\diamond} 
\renewcommand{\implies}{\Rightarrow} 

\DeclareMathOperator{\inversionSet}{\mathbf{N}} 
\DeclareMathOperator{\rootSet}{\mathbf{R}} 
\DeclareMathOperator{\weightSet}{\mathbf{W}} 
\DeclareMathOperator{\rootDescentSet}{\mathbf{D}} 
\newcommand{\LdescentSet}{D_L} 
\newcommand{\RdescentSet}{D_R} 
\DeclareMathOperator{\length}{\ell} 
\newcommand{\fundamentalChamber}{\mathcal{C}} 
\newcommand{\wo}[1]{w_{\circ,#1}} 
\newcommand{\woo}{w_{\circ}} 
\newcommand{\tinymath}[1]{\scalebox{.5}{$#1$}} 
\newcommand{\meet}{\wedge} 
\newcommand{\tinymeet}{{\tinymath{\meet}}} 
\newcommand{\zm}{z_\tinymeet} 
\newcommand{\Km}{{K_\tinymeet}} 
\newcommand{\join}{\vee} 
\newcommand{\tinyjoin}{{\tinymath{\join}}} 
\newcommand{\zj}{z_\tinyjoin} 
\newcommand{\Kj}{{K_\tinyjoin}} 
\newcommand{\CoxeterComplex}[1]{\mathcal{P}_{#1}} 
\newcommand{\DavisComplex}[1]{\mathcal{D}_{#1}} 
\newcommand{\projDown}{\pi_{\!\downarrow\!}} 
\newcommand{\projUp}{\pi^{\!\uparrow\!}} 
\newcommand{\ProjDown}{\Pi_{\downarrow\!}} 
\newcommand{\ProjUp}{\Pi^{\uparrow\!}} 

\newcommand{\sUp}{\sigma^{\!\uparrow\!}}
\newcommand{\SDown}{\Sigma_{\downarrow\!}}
\newcommand{\SUp}{\Sigma^{\uparrow\!}}
\newcommand{\plusTop}[1]{\big[ #1 \big]^\oplus} 
\newcommand{\plusBottom}[1]{\big[ #1 \big]_\oplus} 
\newcommand{\minusTop}[1]{\big[ #1 \big]^\ominus} 
\newcommand{\minusBottom}[1]{\big[ #1 \big]_\ominus} 

\newcommand{\Perm}{\mathsf{Perm}} 
\newcommand{\Para}{\mathsf{Para}} 
\DeclareMathOperator{\face}{\mathbf{F}} 
\newcommand{\Fan}{\mathcal{F}} 

\DeclareMathOperator{\conv}{conv} 
\DeclareMathOperator{\vect}{vect} 
\DeclareMathOperator{\cone}{cone} 
\DeclareMathOperator{\inv}{inv} 
\DeclareMathOperator{\des}{des} 

\newcommand{\fref}[1]{Figure~\ref{#1}} 
\newcommand{\ie}{\textit{i.e.,}~} 
\newcommand{\eg}{\textit{e.g.,}~} 
\definecolor{darkblue}{rgb}{0,0,0.7} 
\newcommand{\darkblue}{\color{darkblue}} 
\newcommand{\defn}[1]{\emph{\darkblue #1}} 
\newcommand{\para}[1]{\medskip\noindent\textbf{#1.}} 
\newcommand{\leqequiv}{
 \mathrel{
  \begin{tikzpicture}[x=.55em, y=0.18ex]
   \foreach \i in {-3,-1,1} {
    \draw[cap=round] (0,\i) -- (1,\i);
   }
   \draw[cap=round] (0,1) -- (1,5);
  \end{tikzpicture}
 }
} 
\newcommand{\Equiv}{
 \mathrel{
  \begin{tikzpicture}[x=.65em, y=0.22ex]
   \foreach \i in {-3,-1,1,3} {
    \draw[cap=round] (0,\i) -- (1,\i);
   }
  \end{tikzpicture}
 }
}
\newcommand{\equivdes}{\equiv^{\mathsf{des}}} 
\newcommand{\Equivdes}{\Equiv^{\mathsf{des}}} 
\newcommand{\equivc}{\equiv^{c}} 
\newcommand{\Equivc}{\Equiv^{c}} 

\usepackage{todonotes}

\begin{document}

\begin{abstract}
We investigate a poset structure that extends the weak order on a finite Coxeter group~$W$ to the set of all faces of the permutahedron of~$W$. We call this order the {\em facial weak order}. We first provide two alternative characterizations of this poset: a first one,  geometric, that generalizes the notion of inversion sets of roots, and a second one, combinatorial, that uses comparisons of the minimal and maximal length representatives of the cosets. These characterizations are then used to show that the facial weak order is in fact a lattice, generalizing a well-known result of A.~Bj\"orner for the classical weak order. Finally, we show that any lattice congruence of the classical weak order induces a lattice congruence of the facial weak order, and we give a geometric interpretation of their classes. As application, we describe the {\em facial boolean lattice} on the faces of the cube and the {\em facial Cambrian lattice} on the faces of the corresponding generalized associahedron. 

\medskip
\noindent
\textsc{keywords.} Permutahedra $\cdot$ weak order $\cdot$ Coxeter complex $\cdot$ lattice quotients $\cdot$ associahedra $\cdot$ Cambrian lattices.
\end{abstract}

\vspace*{.3cm}
\maketitle

\vspace*{-.4cm}
\tableofcontents


\section{Introduction}
\label{sec:intro}

\enlargethispage{.5cm}
The (right) Cayley graph of a Coxeter system $(W,S)$ is naturally oriented by the (right) weak order on~$W$: an edge is oriented from $w$ to $ws$ if $s\in S$ is such that ${\length(w) < \length(ws)}$, see~\cite[Chapter~3]{BrentiBjorner} for details. A celebrated result of A.~Bj\"orner~\cite{Bjorner} states that the weak order is a complete meet-semilattice and even a complete ortholattice in the case of a finite Coxeter system. The weak order is a very useful tool to study Coxeter groups as it encodes the combinatorics of reduced words associated to $(W,S)$, and it underlines the connection between the words and the root system via the notion of inversion sets, see for instance~\cite{Dyer-WeakOrder, HohlwegLabbe} and the references therein. 

In the case of a finite Coxeter system, the Cayley graph of~$W$ is isomorphic to the $1$-skeleton of the $W$-permutahedron. Then the weak order is given by an orientation of the $1$-skeleton of the $W$-permu\-tahedron associated to the choice of a linear form of the ambiant Euclidean space. This point of view was very useful in order to build generalized associahedra out of a $W$-permutahedron using N.~Reading's Cambrian lattices, see~\cite{Reading-survey, HohlwegLangeThomas, Hohlweg}.

In this paper, we study a poset structure on all faces of the $W$-permutahedron that we call the (right) \defn{facial weak order}. This order was introduced by D.~Krob, M.~Latapy, J.-C.~Novelli, H.-D.~Phan and S.~Schwer in~\cite{KrobLatapyNovelliPhanSchwer} for the symmetric group then extended by P.~Palacios and M.~Ronco in~\cite{PalaciosRonco} for arbitrary finite Coxeter groups. Recall that the faces of the $W$-permutahedron are naturally parameterized by the Coxeter complex~$\CoxeterComplex{W}$ which consists of all standard parabolic cosets~$W/W_I$ for~$I\subseteq S$. The aims of this article are:
\begin{enumerate}
\item To give two alternative characterizations of the facial weak order (see Theorem~\ref{thm:facialWeakOrderCharacterizations}): one in terms of root inversion sets of parabolic cosets which extend the notion of inversion sets of elements of~$W$, and the other using weak order comparisons between the minimal and maximal representatives of the parabolic cosets. The advantage of these two definitions is that they give immediate global comparison, while the original definition of~\cite{PalaciosRonco} uses cover relations.
\item To show that the facial weak order is a lattice  (see Theorem~\ref{thm:lattice}), whose restriction to the vertices of the permutahedron produces the weak order as a sublattice. This result was motivated by the special case of type~$A$ proved in~\cite{KrobLatapyNovelliPhanSchwer}.
\item To discuss generalizations of these statements to infinite Coxeter groups via the Davis complex (see Theorem~\ref{thm:latticeInfinite}).
\item To show that any lattice congruence~$\equiv$ of the weak order naturally extends to a lattice congruence~$\Equiv$ of the facial weak order (see Theorem~\ref{theo:latticeCongruenceFacialWeakOrder}). This provides a complete description (see Theorem~\ref{theo:allFacesCongruenceFan}) of the simplicial fan~$\Fan_\equiv$ associated to the weak order congruence~$\equiv$ in N.~Reading's work~\cite{Reading-HopfAlgebras}: while the classes of~$\equiv$ correspond to maximal cones in~$\Fan_\equiv$, the classes of~$\Equiv$ correspond to all cones in~$\Fan_\equiv$. Relevant illustrations are given for Cambrian lattices and fans~\cite{Reading-CambrianLattices, ReadingSpeyer-CambrianFans}, which extend to facial Cambrian lattices on the faces of generalized associahedra (see Theorem~\ref{thm:FacialCambrian}).
\end{enumerate}

The results of this paper are based on combinatorial properties of Coxeter groups, parabolic cosets, and reduced words. However, their motivation and intuition come from the geometry of the Coxeter arrangement and of the $W$-permutahedron. So we made a point to introduce enough of the geometrical material to make the geometric intuition clear.


\section{Preliminaries}
\label{sec:preliminaries}

We start by fixing notations and classical definitions on finite Coxeter groups. Details can be found in textbooks by J.~Humphreys~\cite{Humphreys} and A.~Bj\"orner and F.~Brenti~\cite{BrentiBjorner}. The reader familiar with finite Coxeter groups and root systems is invited to proceed directly to Section~\ref{sec:weakOrder}.


\subsection{Finite reflection groups and Coxeter systems}
\label{subsec:CoxeterSystems}

Let~$(V,\dotprod{\cdot}{\cdot})$ be an $n$-dimensional Euclidean vector space. For any vector~${v \in V \ssm \{0\}}$, we denote by~$s_v$ the reflection interchanging $v$ and $-v$ while fixing the orthogonal hyperplane pointwise. Remember that~$ws_v = s_{w(v)}w$ for any vector~$v \in V \ssm \{0\}$ and any orthogonal transformation~$w$ of~$V$.

We consider a \defn{finite reflection group}~$W$ acting on~$V$, that is, a finite group generated by reflections in the orthogonal group $O(V)$. The \defn{Coxeter arrangement} of~$W$ is the collection of all reflecting hyperplanes. Its complement in~$V$ is a union of open polyhedral cones. Their closures are called \defn{chambers}. The \defn{Coxeter fan} is the polyhedral fan formed by the chambers together with all their faces. This fan is \defn{complete} (its cones cover~$V$) and \defn{simplicial} (all cones are simplicial), and we can assume without loss of generality that it is \defn{essential} (the intersection of all chambers is reduced to the origin). We fix an arbitrary chamber $\fundamentalChamber$ which we call the \defn{fundamental chamber}. The $n$ reflections orthogonal to the facet defining hyperplanes of~$\fundamentalChamber$ are called \defn{simple reflections}. The set~$S$ of simple reflections generates~$W$. The pair~$(W,S)$ forms a \defn{Coxeter system}. See \fref{fig:CoxeterArrangements} for an illustration of the Coxeter arrangements of types~$A_3$, $B_3$, and~$H_3$.

\begin{figure}[b]
	\centerline{\includegraphics[width=1.3\textwidth]{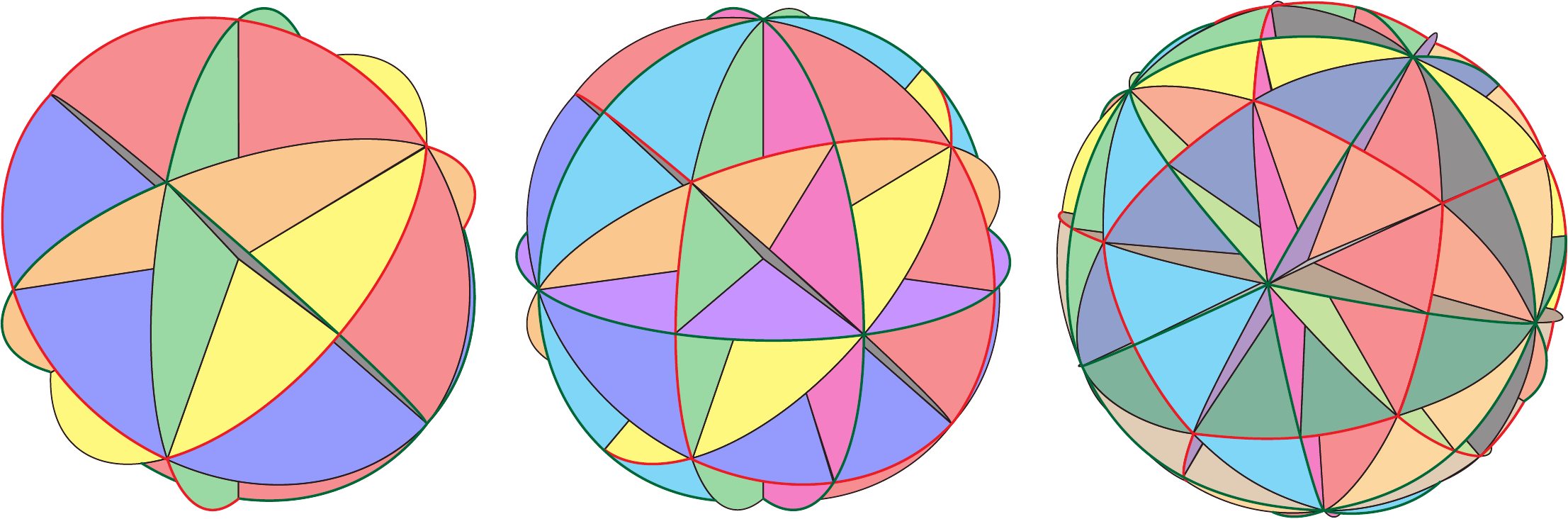}}
	\caption{The type~$A_3$, $B_3$, and~$H_3$ Coxeter arrangements.}
	\label{fig:CoxeterArrangements}
\end{figure}


\subsection{Roots and weights}
\label{subsec:rootsWeights}

We consider a \defn{root system}~$\Phi$ for~$W$, \ie a set of vectors invariant under the action of~$W$ and containing precisely two opposite roots orthogonal to each reflecting hyperplane of~$W$. The \defn{simple roots}~$\Delta$ are the roots orthogonal to the defining hyperplanes of~$\fundamentalChamber$ and pointing towards~$\fundamentalChamber$. They form a linear basis of~$V$. The root system~$\Phi$ splits into the \defn{positive roots}~${\Phi^+ \eqdef \Phi \cap \cone(\Delta)}$ and the \defn{negative roots} $\Phi^- \eqdef \Phi \cap \cone(-\Delta) = -\Phi^+$, where $\cone(X)$ denotes the set of nonnegative linear combinations of vectors in~$X\subseteq V$. In other words, the positive roots are the roots whose scalar product with any vector of the interior of the fundamental chamber~$\fundamentalChamber$ is positive, and the simple roots form the basis of the cone generated by~$\Phi^+$. Each reflection hyperplane is orthogonal to one positive and one negative root. For a reflection $s \in R$, we set $\alpha_s$ to be the unique positive root orthogonal to the reflection hyperplane of~$s$, \ie such that $s = s_{\alpha_s}$.

We denote by~$\alpha_s^\vee \eqdef 2 \alpha_s / \dotprod{\alpha_s}{\alpha_s}$ the \defn{coroot} corresponding to~$\alpha_s \in \Delta$, and by $\Delta^\vee \eqdef \set{\alpha^\vee_s}{s \in S}$ the coroot basis. The vectors of its dual basis~$\nabla \eqdef \set{\omega_s}{s \in S}$ are called \defn{fundamental weights}. In other words, the fundamental weights of~$W$ are defined by~${\dotprod{\alpha_s^\vee}{\omega_t} = \delta_{s=t}}$ for all~$s,t \in S$. Geometrically, the fundamental weight~$\omega_s$ gives the direction of the ray of the fundamental chamber~$\fundamentalChamber$ not contained in the reflecting hyperplane~of~$s$. We let~$\Omega \eqdef W(\nabla) = \set{w(\omega_s)}{w \in W, s \in S}$ denote the set of all weights of~$W$, obtained as the orbit of the fundamental weights under~$W$.


\subsection{Length, reduced words and weak order}
\label{subsec:lengthReducedWordsAndWeakOrder}

The \defn{length}~$\length(w)$ of an element~$w \in W$ is the length of the smallest word for~$w$ as a product of generators in~$S$. A word~${w = s_1 \cdots s_k}$ with~$s_1,\dots,s_k \in S$ is called \defn{reduced} if~$k = \length(w)$. For $u,v\in W$, the product $uv$ is said to be \defn{reduced} if the concatenation of a reduced word for $u$ and of a reduced word for $v$ is a reduced word for $uv$, \ie if $\length(uv)=\length(u)+\length(v)$. We say that $u\in W$ is a \defn{prefix} of $v \in W$ if there is a reduced word for $u$ that is the prefix of a reduced word for $v$, \ie if~$\length(u^{-1}v)=\length(v)-\length(u)$.

The (right) \defn{weak order} is the order on~$W$ defined equivalently by
\[
u \le v \iff \length(u) + \length(u^{-1}v) = \length(v) \iff u \text{ is a prefix of } v.
\]
A.~Bj\"orner shows in~\cite{Bjorner} that the weak order defines a lattice structure on~$W$ (finite Coxeter group), with minimal element~$e$ and maximal element~$\woo$ (which sends all positive roots to negative ones and all positive simple roots to negative simple ones). The conjugation~$w \mapsto \woo w \woo$ defines a weak order automorphism while the left and right multiplications~$w \mapsto \woo w$ and~$w \mapsto w \woo$ define weak order anti-automorphisms. We refer the reader to \cite[Chapter~3]{BrentiBjorner} for more details. 

The weak order encodes the combinatorics of reduced words and enjoys a useful geometric characterization within the root system, which we explain now. The (left) \defn{inversion set} of~$w$ is the set $\inversionSet(w) \eqdef \Phi^+ \, \cap \, w(\Phi^-)$ of positive roots sent to negative ones by~$w^{-1}$. If $w=uv$ is reduced then ${\inversionSet(w)=\inversionSet(u)\sqcup u \big( \inversionSet(v) \big)}$. In particular, we have ${\inversionSet(w) = \big\{\alpha_{s_1}, s_1(\alpha_{s_2}), \dots, s_1s_2 \cdots s_{p-1}(\alpha_{s_k})\big\}}$ for any reduced word~$w = s_1 \cdots s_k$, and therefore~$\length(w) = |\inversionSet(w)|$. Moreover, the weak order is characterized in term of inversion sets by:
\[ 
u \leq v \iff \inversionSet(u) \subseteq \inversionSet(v),
\]
for any $u,v\in W$. We refer the reader to \cite[Section~2]{HohlwegLabbe} and the references therein for more details on inversion sets and the weak order.

We say that~$s \in S$ is a \defn{left ascent} of~$w \in W$ if~${\length(sw) = \length(w) + 1}$ and a \defn{left descent} of~$w$ if~${\length(sw) = \length(w) - 1}$. We denote by~$\LdescentSet(w)$ the set of left descents of~$w$. Note that for~$s \in S$ and~$w \in W$, we have~$s \in \LdescentSet(w) \iff \alpha_s \in \inversionSet(w) \iff s \leq w$. Similarly, $s \in S$ is a \defn{right descent} of~$w \in W$ if~${\length(ws) = \length(w) - 1}$, and we denote by~$\RdescentSet(w)$ the set of right descents of~$w$.


\subsection{Parabolic subgroups and cosets}
\label{subsec:parabolicSubgroupsCosets}

Consider a subset~$I \subseteq S$. The \defn{standard parabolic subgroup}~$W_I$ is the subgroup of~$W$ generated by~$I$. It is also a Coxeter group with simple generators~$I$, simple roots~$\Delta_I \eqdef \set{\alpha_s}{s \in I}$, root system~${\Phi_I = W_I(\Delta_I) = \Phi \cap \vect(\Delta_I)}$, length function~$\length_I = \length|_{W_I}$, longest element~$\wo{I}$, etc. For example, $W_\varnothing = \{e\}$ while~$W_S = W$. 

We denote by~$W^I \eqdef \set{w \in W}{\length(ws) > \length(w) \text{ for all } s \in I}$ the set of elements of~$W$ with no right descents in~$I$. For example, $W^\varnothing = W$ while~$W^S = \{e\}$. Observe that for any~$x \in W^I$, we have~$x(\Delta_I) \subseteq \Phi^+$ and thus $x(\Phi_I^+) \subseteq \Phi^+$. We will use this property repeatedly in this paper.

Any element~$w \in W$ admits a unique factorization~$w = w^I \cdot w_I$ with~$w^I \in W^I$ and~$w_I \in W_I$, and moreover, $\length(w) = \length(w^I) + \length(w_I)$ (see \cite[Proposition~2.4.4]{BrentiBjorner}). Therefore, $W^I$ is the set of \defn{minimal length coset representatives} of the cosets~$W/W_I$. Throughout the paper, we will always implicitly assume that~$x \in W^I$ when writing that~$xW_I$ is a \defn{standard parabolic coset}. Note that any standard parabolic coset~$xW_I = [x, x\wo{I}]$ is an interval in the weak order. The \defn{Coxeter complex}~$\CoxeterComplex{W}$ is the abstract simplicial complex whose faces are all standard parabolic cosets~of~$W$: 
\[
\CoxeterComplex{W} = \bigcup_{I\subseteq S} W/W_I  = \bigset{xW_I}{I \subseteq S, \ x \in W}= \bigset{xW_I}{I \subseteq S, \ x \in W^I}.
\]

We will also need {\em Deodhar's Lemma}: for $s\in S$, $I\subseteq S$ and $x \in W^I$, either $sx \in W^I$ or $sx = xr$ for some $r\in I$. See for instance~\cite[Lemma~2.1.2]{GeckPfeiffer} where it is stated for the cosets $W_I\backslash W$ instead of $W/W_I$.


\subsection{Permutahedron}
\label{subsec:permutahedron}

Remember that a \defn{polytope}~$P$ is the convex hull of finitely many points of~$V$, or equivalently a bounded intersection of finitely many affine halfspaces of~$V$. The \defn{faces} of~$P$ are the intersections of~$P$ with its supporting hyperplanes and the \defn{face lattice} of~$P$ is the lattice of its faces ordered by inclusion. The \defn{inner primal cone} of a face~$F$ of~$P$ is the cone generated by~$\set{u-v}{u \in P, v \in F}$. The \defn{outer normal cone} of a face~$F$ of~$P$ is the cone generated by the outer normal vectors of the facets of~$P$ containing~$F$. Note that these two cones are polar to each other. The \defn{normal fan} is the complete polyhedral fan formed by the outer normal cones of all faces of~$P$. We refer to~\cite{Ziegler} for details on polytopes, cones, and fans.

The \defn{$W$-permutahedron}~$\Perm^p(W)$ is the convex hull of the orbit under~$W$ of a generic point~${p \in V}$ (not located on any reflection hyperplane of~$W$). Its vertex and facet descriptions are given by
\[
\Perm^p(W) = \conv \bigset{w(p)}{w \in W} = \bigcap_{\substack{s \in S \\ w \in W}} \bigset{v \in V}{\dotprod{w(\omega_s)}{v} \le \dotprod{\omega_s}{p}}.
\]
Examples in types~$A_2$ and $B_2$ are represented in Figure \ref{fig:A2B2faces}.
Examples in types~$A_3$, $B_3$, and~$H_3$ are represented in \fref{fig:CoxeterPermutahedra}.

We often write~$\Perm(W)$ instead of~$\Perm^p(W)$ as the combinatorics of the $W$-permutahedron is independent of the choice of the point~$p$ and is encoded by the Coxeter complex $\CoxeterComplex{W}$. More precisely, each standard parabolic coset~$xW_I$ corresponds to a face~$\face(xW_I)$ of~$\Perm^p(W)$ given by
\[
\face(xW_I) = x \big( \Perm^p(W_I) \big) = \Perm^{x(p)} \big( xW_Ix^{-1} \big).
\] 
\begin{figure}
	\centerline{
	\raisebox{.3cm}{
	\begin{tikzpicture}
		[scale=2.5,
		edge/.style={color=blue!95!black},
		face/.style={fill=red!75!blue, fill opacity=0.100000},
		vertex/.style={inner sep=1pt, circle, draw=black, fill=black, thick, color=black}
		]
		%
		\coordinate (e) at (0,0.42) {};
		\coordinate (s) at (-1,1) {};
		\coordinate (t) at (1,1) {};
		\coordinate (st) at (-1,2) {};
		\coordinate (ts) at (1,2) {};
		\coordinate (sts) at (0,2.58) {};
		%
		\draw[edge] (e) -- (s) node (Ws) [midway, below left] {$\face(W_s)$};
		\draw[edge] (e) -- (t) node (Wt) [midway, below right] {$\face(W_t)$};
		\draw[edge] (s) -- (st) node (sWt) [midway, left] {$\face(sW_t)$};
		\draw[edge] (t) -- (ts) node (tWs) [midway, right] {$\face(tW_s)$};
		\draw[edge] (st) -- (sts) node (stWs) [midway, above left] {$\face(stW_s)$};
		\draw[edge] (ts) -- (sts) node (tsWt) [midway, above right] {$\face(tsW_t)$};
		%
		\node[vertex] at (e) {};
		\node[vertex] at (s) {};
		\node[vertex] at (t) {};
		\node[vertex] at (st) {};
		\node[vertex] at (ts) {};
		\node[vertex] at (sts) {};
		%
		\node[below, color=black] at (e) {$\face(e)$};
		\node[left, color=black] at (s) {$\face(s)$};
		\node[right, color=black] at (t) {$\face(t)$};
		\node[left, color=black] at (st) {$\face(st)$};
		\node[right, color=black] at (ts) {$\face(ts)$};
		\node[above, color=black] at (sts) {$\face(sts)$};
		%
		\fill[face] (e) -- (t) -- (ts) -- (sts) -- (st) -- (s) -- cycle {};
		\node[color=red!75!blue] at (0,1.5) {$\face(W)$};		
	\end{tikzpicture}}
	\begin{tikzpicture}
		[scale=1.5,
		edge/.style={color=blue!95!black},
		face/.style={fill=red!75!blue, fill opacity=0.100000},
		vertex/.style={inner sep=1pt, circle, draw=black, fill=black, thick, color=black}
		]
		%
		\coordinate (e) at (0,0);
		\coordinate (s) at (-1.4,0.6);
		\coordinate (t) at (1.4,0.6);
		\coordinate (st) at (-2,2);
		\coordinate (ts) at (2,2);
		\coordinate (sts) at (-1.4,3.4);
		\coordinate (tst) at (1.4,3.4);
		\coordinate (stst) at (0,4);
		%
		\draw[edge] (e) -- (s) node (Ws) [midway, below left] {$\face(W_s)$};
		\draw[edge] (e) -- (t) node (Wt) [midway, below right] {$\face(W_t)$};
		\draw[edge] (s) -- (st) node (sWt) [midway, left] {$\face(sW_t)$};
		\draw[edge] (t) -- (ts) node (tWs) [midway, right] {$\face(tW_s)$};
		\draw[edge] (st) -- (sts) node (stWs) [midway, left] {$\face(stW_s)$};
		\draw[edge] (ts) -- (tst) node (tsWt) [midway, right] {$\face(tsW_t)$};
		\draw[edge] (sts) -- (stst) node (stsWt) [midway, above left] {$\face(stsW_t)$};
		\draw[edge] (tst) -- (stst) node (tstWs) [midway, above right] {$\face(tstW_s)$};
		%
		\node[vertex] at (e) {};
		\node[vertex] at (s) {};
		\node[vertex] at (t) {};
		\node[vertex] at (st) {};
		\node[vertex] at (ts) {};
		\node[vertex] at (tst) {};
		\node[vertex] at (sts) {};
		\node[vertex] at (stst) {};
		%
		\node[below, color=black] at (e) {$\face(e)$};
		\node[left, color=black] at (s) {$\face(s)$};
		\node[right, color=black] at (t) {$\face(t)$};
		\node[left, color=black] at (st) {$\face(st)$};
		\node[right, color=black] at (ts) {$\face(ts)$};
		\node[left, color=black] at (sts) {$\face(sts)$};
		\node[right, color=black] at (tst) {$\face(tst)$};
		\node[above, color=black] at (stst) {$\face(stst)$};
		%
		\fill[face] (e) -- (t) -- (ts) -- (tst) -- (stst) -- (sts) -- (st) -- (s) -- cycle {};
		\node[color=red!75!blue] at (0,2) {$\face(W)$};
	\end{tikzpicture}
}
	\caption{Standard parabolic cosets of the type $A_2$ and $B_2$ Coxeter groups and the corresponding faces on their permutahedra.}
	\label{fig:A2B2faces}
\end{figure}
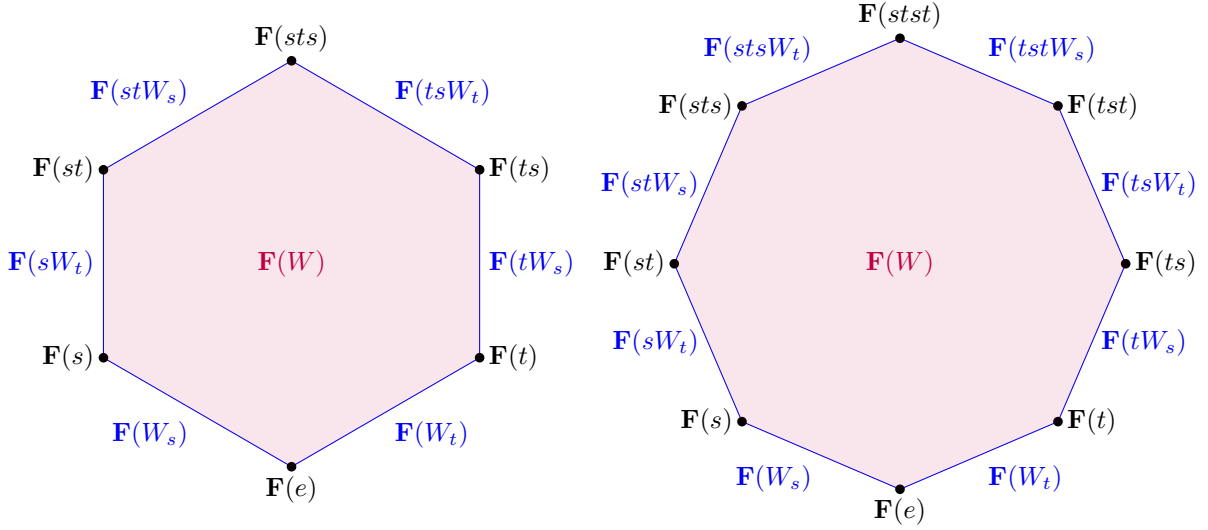
\begin{figure}
	\centerline{\includegraphics[width=1.3\textwidth]{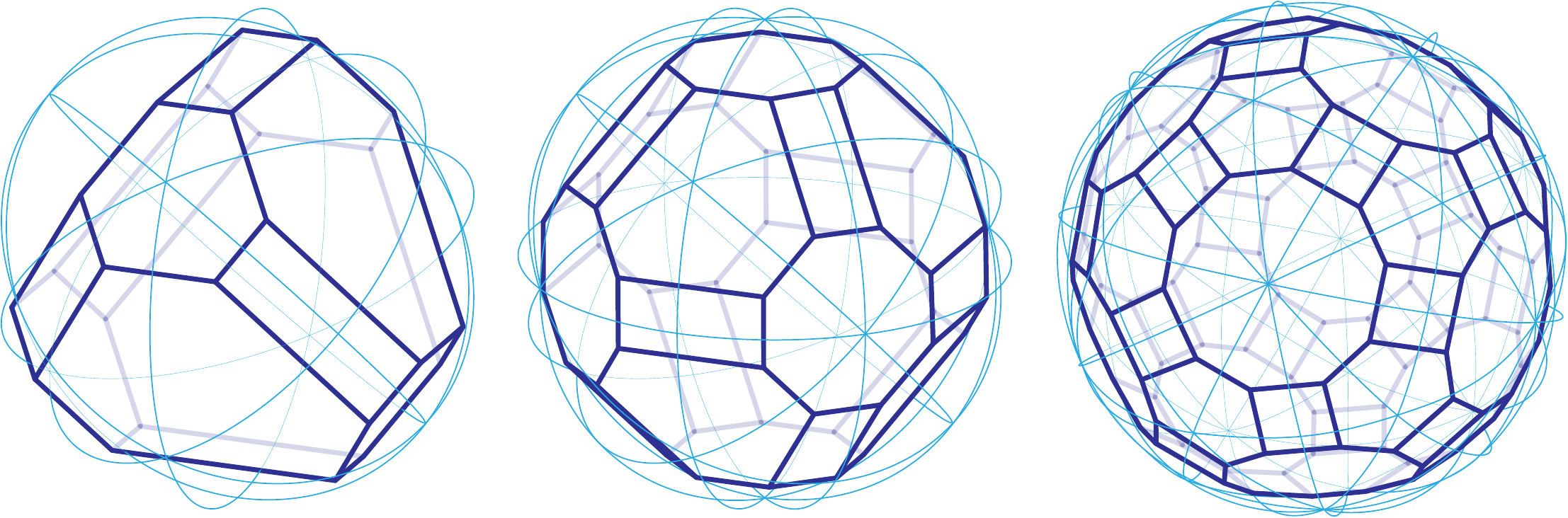}}
	\caption{The type~$A_3$, $B_3$, and~$H_3$ permutahedra.}
	\label{fig:CoxeterPermutahedra}
\end{figure}
Therefore, the $k$-dimensional faces of~$\Perm^p(W)$ correspond to the cosets~$xW_I$ with~$|I| = k$ and the face lattice of~$\Perm^p(W)$ is isomorphic to the inclusion poset $(\CoxeterComplex{W},\subseteq )$. The normal fan of~$\Perm^p(W)$ is the Coxeter fan. The graph of the permutahedron~$\Perm^p(W)$ is isomorphic to the Cayley graph of the Coxeter system~$(W,S)$. Moreover, when oriented in the linear direction~$\woo(p) - p$, it coincides with the Hasse diagram of the (right) weak order on~$W$. We refer the reader to~\cite{Hohlweg} for more details on the $W$-permutahedron.

\begin{example}
\label{exm:typeA}
The Coxeter group of type~$A_{n-1}$ is the symmetric group~$\fS_n$. Its simple generators are the simple transpositions~$\tau_i = (i\;i+1)$ for~$i \in [n-1]$ with relations~$\tau_i^2 = 1$ and~$\tau_i\tau_{i+1}\tau_i = \tau_{i+1}\tau_i\tau_{i+1}$. Its elements are permutations of~$[n]$ and its standard parabolic cosets are ordered partitions of~$[n]$. A root system for $\fS_n$ consists in the set of vectors~$\set{e_i-e_j}{i \ne j \in [n]}$ where $(e_1,\dots,e_n)$ is the canonical basis of~$\R^n$. The type~$A_3$ Coxeter arrangement is represented in \fref{fig:CoxeterArrangements}\,(left), and the type~$A_2$ and~$A_3$ permutahedra are represented in Figures~\ref{fig:A2B2faces}\,(left) and~\ref{fig:CoxeterPermutahedra}\,(left).
\end{example}


\section{Facial weak order on the Coxeter complex}
\label{sec:weakOrder}

In this section we study an analogue of the weak order on standard parabolic cosets, which we call the \defn{facial weak order}. It was defined  for the symmetric group by D.~Krob, M.~Latapy, J.-C.~Novelli, H.-D.~Phan and S.~Schwer in~\cite{KrobLatapyNovelliPhanSchwer}, then extended for arbitrary finite Coxeter groups  by P.~Palacios and M.~Ronco in~\cite{PalaciosRonco}.

\begin{definition}[\cite{KrobLatapyNovelliPhanSchwer, PalaciosRonco}]
\label{def:facialWeakOrder}
The \defn{(right) facial weak order} is the order $\leq$ on the Coxeter complex $\CoxeterComplex{W}$ defined by cover relations of two types: 
\begin{align*}
\text{(1)}\qquad xW_I & \precdot xW_{I \cup \{s\}} \qquad\text{if } s \notin I \text{ and } x \in W^{I \cup \{s\}}, \\
\text{(2)}\qquad xW_I & \precdot x\wo{I}\wo{I \ssm \{s\}}W_{I \ssm \{s\}} \qquad\text{if } s \in I,
\end{align*}
where $I\subseteq S$ and $x\in W^I$.
\end{definition}

We have illustrated the facial weak order on the faces of the permutahedron in types~$A_2$ and $B_2$ in~\fref{fig:A2B2FacialWeakOrder} and in type~$A_3$ in~\fref{fig:A3Words}.

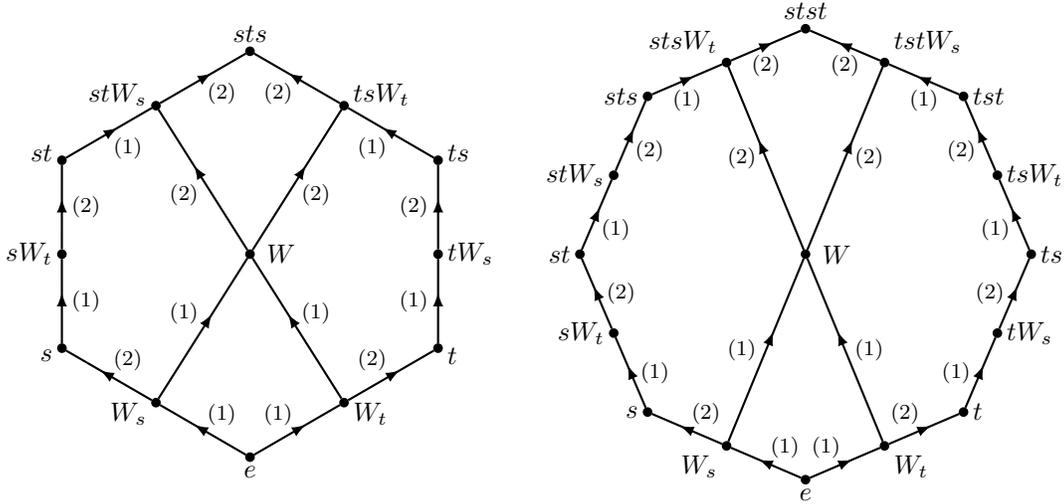
\begin{figure}[h]
	\centerline{
	\raisebox{.3cm}{
	\begin{tikzpicture}
		[scale=2.5,
		aface/.style={color=blue},
		bface/.style={color=red},
		face/.style={color=red!50!blue},
		vertex/.style={inner sep=1pt, circle, draw=black, fill=black, thick, color=black},
		middlearrow/.style={decoration={markings, mark=at position 0.6 with {\arrow{#1}}}, postaction={decorate}}
		]
		%
		\node[vertex] (e) at (0,0.42) {};
		\node[vertex] (s) at (-1,1) {};
		\node[vertex] (t) at (1,1) {};
		\node[vertex] (st) at (-1,2) {};
		\node[vertex] (ts) at (1,2) {};
		\node[vertex] (sts) at (0,2.58) {};
		%
		\node[vertex] (Ws) at ($(e)!0.5!(s)$) {};
		\node[vertex] (Wt) at ($(e)!0.5!(t)$) {};
		\node[vertex] (tWs) at ($(t)!0.5!(ts)$) {};
		\node[vertex] (sWt) at ($(s)!0.5!(st)$) {};
		\node[vertex] (stWs) at ($(st)!0.5!(sts)$) {};
		\node[vertex] (tsWt) at ($(ts)!0.5!(sts)$) {};
		%
		\node[vertex] (W) at ($(e)!0.5!(sts)$) {};
		%
		\node[below] at (e) {$e$};
		\node[below=1mm, left] at (s) {$s$};
		\node[below=1mm, right] at (t) {$t$};
		\node[above=1mm, left] at (st) {$st$};
		\node[above=1mm, right] at (ts) {$ts$};
		\node[above] at (sts) {$sts$};
		\node[below=1.5mm, left] at (Ws) {$W_{s}$};
		\node[below=1.5mm, right] at (Wt) {$W_{t}$};
		\node[right] at (tWs) {$tW_{s}$};
		\node[left] at (sWt) {$sW_{t}$};
		\node[above=1.5mm, left] at (stWs) {$stW_{s}$};
		\node[above=1.5mm, right] at (tsWt) {$tsW_{t}$};
		\node[right=1mm] at (W) {$W$};
		%
		\draw[middlearrow={latex}, thick] (e) -- (Ws) node [midway, above right=-1mm] {\footnotesize(1)};
		\draw[middlearrow={latex}, thick] (e) -- (Wt) node [midway, above left=-1mm] {\footnotesize(1)};
		\draw[middlearrow={latex}, thick] (s) -- (sWt) node [midway, right] {\footnotesize(1)};
		\draw[middlearrow={latex}, thick] (t) -- (tWs) node [midway, left] {\footnotesize(1)};
		\draw[middlearrow={latex}, thick] (st) -- (stWs) node [midway, below right=-1mm] {\footnotesize(1)};
		\draw[middlearrow={latex}, thick] (ts) -- (tsWt) node [midway, below left=-1mm] {\footnotesize(1)};
		\draw[middlearrow={latex}, thick] (Ws) -- (s) node [midway, above right=-1mm] {\footnotesize(2)};
		\draw[middlearrow={latex}, thick] (Wt) -- (t) node [midway, above left=-1mm] {\footnotesize(2)};
		\draw[middlearrow={latex}, thick] (sWt) -- (st) node [midway, right] {\footnotesize(2)};
		\draw[middlearrow={latex}, thick] (tWs) -- (ts) node [midway, left] {\footnotesize(2)};
		\draw[middlearrow={latex}, thick] (stWs) -- (sts) node [midway, below right=-1mm] {\footnotesize(2)};
		\draw[middlearrow={latex}, thick] (tsWt) -- (sts) node [midway, below left=-1mm] {\footnotesize(2)};
		%
		\draw[middlearrow={latex}, thick] (Ws) -- (W) node [midway, above left=-1mm] {\footnotesize(1)};
		\draw[middlearrow={latex}, thick] (Wt) -- (W) node [midway, above right=-1mm] {\footnotesize(1)};
		\draw[middlearrow={latex}, thick] (W) -- (stWs) node [midway, below left=-1mm] {\footnotesize(2)};
		\draw[middlearrow={latex}, thick] (W) -- (tsWt) node [midway, below right=-1mm] {\footnotesize(2)};
	\end{tikzpicture}}
\quad
	\begin{tikzpicture}
		[scale=1.5,
		edge/.style={color=blue!95!black},
		face/.style={fill=red!75!blue, fill opacity=0.100000},
		vertex/.style={inner sep=1pt, circle, draw=black, fill=black, thick, color=black},
		middlearrow/.style={decoration={markings, mark=at position 0.6 with {\arrow{#1}}}, postaction={decorate}}
		]
		%
		\node[vertex] (e) at (0,0) {};
		\node[vertex] (s) at (-1.4,0.6) {};
		\node[vertex] (t) at (1.4,0.6) {};
		\node[vertex] (st) at (-2,2) {};
		\node[vertex] (ts) at (2,2) {};
		\node[vertex] (sts) at (-1.4,3.4) {};
		\node[vertex] (tst) at (1.4,3.4) {};
		\node[vertex] (stst) at (0,4) {};
		%
		\node[vertex] (Ws) at ($(e)!0.5!(s)$) {};
		\node[vertex] (Wt) at ($(e)!0.5!(t)$) {};
		\node[vertex] (tWs) at ($(t)!0.5!(ts)$) {};
		\node[vertex] (sWt) at ($(s)!0.5!(st)$) {};
		\node[vertex] (stWs) at ($(st)!0.5!(sts)$) {};
		\node[vertex] (tsWt) at ($(ts)!0.5!(tst)$) {};
		\node[vertex] (stsWt) at ($(sts)!0.5!(stst)$) {};
		\node[vertex] (tstWs) at ($(tst)!0.5!(stst)$) {};
		%
		\node[vertex] (W) at ($(e)!0.5!(stst)$) {};
		%
		\node[below, color=black] at (e) {$e$};
		\node[left, color=black] at (s) {$s$};
		\node[right, color=black] at (t) {$t$};
		\node[left, color=black] at (st) {$st$};
		\node[right, color=black] at (ts) {$ts$};
		\node[left, color=black] at (sts) {$sts$};
		\node[right, color=black] at (tst) {$tst$};
		\node[above, color=black] at (stst) {$stst$};
		\node[below left] at (Ws) {$W_{s}$};
		\node[below right] at (Wt) {$W_{t}$};
		\node[left] at (sWt) {$sW_{t}$};
		\node[right] at (tWs) {$tW_{s}$};
		\node[left] at (stWs) {$stW_{s}$};
		\node[right] at (tsWt) {$tsW_{t}$};
		\node[above left] at (stsWt) {$stsW_{t}$};
		\node[above right] at (tstWs) {$tstW_{s}$};
		\node[right=1mm] at (W) {$W$};
		\draw[middlearrow={latex}, thick] (e) -- (Ws) node [midway, above right=-1mm] {\footnotesize(1)};
		\draw[middlearrow={latex}, thick] (e) -- (Wt) node [midway, above left=-1mm] {\footnotesize(1)};
		\draw[middlearrow={latex}, thick] (s) -- (sWt) node [midway, right] {\footnotesize(1)};
		\draw[middlearrow={latex}, thick] (t) -- (tWs) node [midway, left] {\footnotesize(1)};
		\draw[middlearrow={latex}, thick] (st) -- (stWs) node [midway, below right=-1mm] {\footnotesize(1)};
		\draw[middlearrow={latex}, thick] (ts) -- (tsWt) node [midway, below left=-1mm] {\footnotesize(1)};
		\draw[middlearrow={latex}, thick] (sts) -- (stsWt) node [midway, below] {\footnotesize(1)};
		\draw[middlearrow={latex}, thick] (tst) -- (tstWs) node [midway, below] {\footnotesize(1)};
		\draw[middlearrow={latex}, thick] (Ws) -- (s) node [midway, above right=-1mm] {\footnotesize(2)};
		\draw[middlearrow={latex}, thick] (Wt) -- (t) node [midway, above left=-1mm] {\footnotesize(2)};
		\draw[middlearrow={latex}, thick] (sWt) -- (st) node [midway, right] {\footnotesize(2)};
		\draw[middlearrow={latex}, thick] (tWs) -- (ts) node [midway, left] {\footnotesize(2)};
		\draw[middlearrow={latex}, thick] (stWs) -- (sts) node [midway, below right=-1mm] {\footnotesize(2)};
		\draw[middlearrow={latex}, thick] (tsWt) -- (tst) node [midway, below left=-1mm] {\footnotesize(2)};
		\draw[middlearrow={latex}, thick] (stsWt) -- (stst) node [midway, below] {\footnotesize(2)};
		\draw[middlearrow={latex}, thick] (tstWs) -- (stst) node [midway, below] {\footnotesize(2)};
		%
		\draw[middlearrow={latex}, thick] (Ws) -- (W) node [midway, left] {\footnotesize(1)};
		\draw[middlearrow={latex}, thick] (Wt) -- (W) node [midway, right] {\footnotesize(1)};
		\draw[middlearrow={latex}, thick] (W) -- (stsWt) node [midway, left] {\footnotesize(2)};
		\draw[middlearrow={latex}, thick] (W) -- (tstWs) node [midway, right] {\footnotesize(2)};
	\end{tikzpicture}
}
	\caption{The facial weak order on the standard parabolic cosets of the Coxeter group of types~$A_2$ and~$B_2$. Edges are labelled with the cover relations of type (1) or (2) as in Definition~\ref{def:facialWeakOrder}.}
	\label{fig:A2B2FacialWeakOrder}
\end{figure}

\begin{figure}
	\centerline{
	\begin{tikzpicture}%
		[x={(0.767968cm, 0.559570cm)},
		y={(-0.407418cm, 0.802202cm)},
		z={(0.494203cm, -0.208215cm)},
		scale=3.80000,
		vertex/.style={draw=blue!95!black, anchor=base, rectangle, fill=blue!15!white, minimum width=1cm, minimum height=0.5cm},
		bvertex/.style={draw=red!95!black, anchor=base, rectangle, fill=red!15!white, minimum width=1cm, minimum height=0.5cm, opacity=0.6},
		edge/.style={color=blue!95!black, ultra thick, decoration={markings, mark=at position 0.3 with {\arrow{latex}}, mark=at position 0.75 with {\arrow{latex}}}, postaction={decorate}},
		bedge/.style={color=red!95!black, fill=red!15!white, opacity=0.1, ultra thick, decoration={markings, mark=at position 0.3 with {\arrow{latex}}, mark=at position 0.75 with {\arrow{latex}}}, postaction={decorate}},
		fedge/.style={color=blue!95!black, decoration={markings, mark=at position 0.3 with {\arrow[ultra thick]{latex}}, mark=at position 0.75 with {\arrow[ultra thick]{latex}}}, postaction={decorate}},
		bfedge/.style={color=red!95!black, opacity=0.1, thin, decoration={markings, mark=at position 0.3 with {\arrow[ultra thick]{latex}}, mark=at position 0.75 with {\arrow[ultra thick]{latex}}}, postaction={decorate}},
		cedge/.style={color=green!80!black, opacity=.4, thin, decoration={markings, mark=at position 0.3 with {\arrow[ultra thick]{latex}}, mark=at position 0.75 with {\arrow[ultra thick]{latex}}}, postaction={decorate}},
		facet/.style={fill=blue, fill opacity=0.1}]
		%
		\coordinate (r) at (-2.12, -0.408, 0.577);
		\coordinate (rs) at (-2.12, 0.408, -0.577);
		\coordinate (e) at (-1.41, -1.63, 0.577);
		\coordinate (srs) at (-1.41, 0.0000200, -1.73);
		\coordinate (rt) at (-1.41, 0.0000300, 1.73);
		\coordinate (rst) at (-1.41, 1.63, -0.577);
		\coordinate (s) at (-0.707, -2.04, -0.577);
		\coordinate (sr) at (-0.707, -1.22, -1.73);
		\coordinate (t) at (-0.707, -1.22, 1.73);
		\coordinate (rts) at (-0.707, 1.22, 1.73);
		\coordinate (rtst) at (-0.707, 2.04, 0.577);
		\coordinate (srst) at (-0.707, 1.22, -1.73);
		\coordinate (srt) at (0.707, -1.22, -1.73);
		\coordinate (rtsrt) at (0.707, 2.04, 0.577);
		\coordinate (st) at (0.707, -2.04, -0.577);
		\coordinate (ts) at (0.707, -1.22, 1.73);
		\coordinate (srtst) at (0.707, 1.22, -1.73);
		\coordinate (rtsr) at (0.707, 1.22, 1.73);
		\coordinate (stsr) at (2.12, -0.408, 0.577);
		\coordinate (sts) at (1.41, -1.63, 0.577);
		\coordinate (srts) at (1.41, 0.0000200, -1.73);
		\coordinate (tsr) at (1.41, 0.0000300, 1.73);
		\coordinate (srtsrt) at (1.41, 1.63, -0.577);
		\coordinate (srtsr) at (2.12, 0.408, -0.577);
		\coordinate (W) at (0, 0, 0);
		%
		\draw[bedge] (rs) -- (srs) node (rsWr) [midway, opacity=.6] {$rsW_r$};
		\draw[bedge] (e) -- (s) node (Ws) [midway, opacity=.6] {$W_s$};
		\draw[bedge] (sr) -- (srs) node (srWs) [midway, opacity=.6] {$srW_s$};
		\draw[bedge] (srs) -- (srst) node (srsWt) [midway, opacity=.6] {$srsW_t$};
		\draw[bedge] (s) -- (sr) node (sWr) [midway, opacity=.6] {$sW_r$};
		\draw[bedge] (s) -- (st) node (sWt) [midway, opacity=.6] {$sW_t$};
		\draw[bedge] (sr) -- (srt) node (srWt) [midway, opacity=.6] {$srW_t$};
		\draw[bedge] (st) -- (srt) node (stWr) [midway, opacity=.6] {$stW_r$};
		\draw[bedge] (srt) -- (srts) node (srtWs) [midway, opacity=.6] {$srtW_s$};
		\draw[bedge] (st) -- (sts) node (stWs) [midway, opacity=.6] {$stW_s$};
		\draw[bedge] (srts) -- (srtst) node (srtsWt) [midway, opacity=.6] {$srtsW_t$};
		\draw[bedge] (srts) -- (srtsr) node (srtsWr) [midway, opacity=.6] {$srtsW_r$};
		\draw[bedge] (e) -- (r) node (Wr) [midway] {};
		\draw[bedge] (r) -- (rs) node (rWs) [midway] {};
		\draw[bedge] (rs) -- (rst) node (rsWt) [midway] {};
		\draw[bedge] (rst) -- (srst) node (rstWr) [midway] {};
		\draw[bedge] (srst) -- (srtst) node (srstWs) [midway] {};
		\draw[bedge] (srtst) -- (srtsrt) node (srtstWr) [midway] {};
		\draw[bedge] (e) -- (t) node (Wt) [midway] {};
		\draw[bedge] (t) -- (ts) node (tWs) [midway] {};
		\draw[bedge] (ts) -- (sts) node (tsWt) [midway] {};
		\draw[bedge] (sts) -- (stsr) node (stsWr) [midway] {};
		\draw[bedge] (stsr) -- (srtsr) node (stsrWs) [midway] {};
		\draw[bedge] (srtsr) -- (srtsrt) node (srtsrWt) [midway] {};
		\draw[bfedge] (Wr.center) -- (srWs.center);
		\draw[bfedge] (Ws.center) -- (rsWr.center) node (Wrs) [midway, opacity=.6] {$W_{rs}$};
		\draw[bfedge] (Ws.center) -- (tsWt.center);
		\draw[bfedge] (Wt.center) -- (stWs.center) node (Wst) [midway, opacity=.6] {$W_{st}$};
		\draw[bfedge] (sWr.center) -- (stWr.center);
		\draw[bfedge] (sWt.center) -- (srWt.center) node (sWrt) [midway, opacity=.6] {$sW_{rt}$};
		\draw[bfedge] (stWr.center) -- (stsrWs.center);
		\draw[bfedge] (stWs.center) -- (srtsWr.center) node (stWrs) [midway, opacity=.6] {$stW_{rs}$};
		\draw[bfedge] (srWs.center) -- (srtsWt.center);
		\draw[bfedge] (srWt.center) -- (srstWs.center) node (srWst) [midway, opacity=.6] {$srW_{st}$};
		\draw[bfedge] (rsWr.center) -- (rstWr.center);
		\draw[bfedge] (rsWt.center) -- (srsWt.center) node (rsWrt) [midway, opacity=.6] {$\!\!\!\!\!\!\!\!\!\!\!\!\!\!\!\!\!\!\!\!\!\!\!\!\!\!\!\!rsW_{rt}$};
		\draw[bfedge] (srtsWt.center) -- (srtsrWt.center);
		\draw[bfedge] (srtsWr.center) -- (srtstWr.center) node (srtsWrt) [midway, opacity=.6] {$srtsW_{rt}$};
		\node[bvertex] at (srt) {$srt$};
		\node[bvertex] at (st) {$st$};
		\node[bvertex] at (srts) {$srts$};
		\node[bvertex] at (srs) {$srs$};
		\node[bvertex] at (s) {$s$};
		\node[bvertex] at (sr) {$sr$};
		%
		\draw[edge] (e) -- (r) node (Wr) [midway] {$W_r$};
		\draw[edge] (r) -- (rs) node (rWs) [midway] {$rW_s$};
		\draw[edge] (r) -- (rt) node (rWt) [midway] {$rW_t$};
		\draw[edge] (rs) -- (rst) node (rsWt) [midway] {$rsW_t$};
		\draw[edge] (e) -- (t) node (Wt) [midway] {$W_t$};
		\draw[edge] (t) -- (rt) node (tWr) [midway] {$tW_r$};
		\draw[edge] (rt) -- (rts) node (rtWs) [midway] {$rtW_s$};
		\draw[edge] (rst) -- (rtst) node (rstWs) [midway] {$rstW_s$};
		\draw[edge] (rst) -- (srst) node (rstWr) [midway] {$rstW_r$};
		\draw[edge] (t) -- (ts) node (tWs) [midway] {$tW_s$};
		\draw[edge] (rts) -- (rtst) node (rtsWt) [midway] {$rtsW_t$};
		\draw[edge] (rts) -- (rtsr) node (rtsWr) [midway] {$rtsW_r$};
		\draw[edge] (rtst) -- (rtsrt) node (rtstWr) [midway] {$rtstW_r$};
		\draw[edge] (srst) -- (srtst) node (srstWs) [midway] {$srstW_s$};
		\draw[edge] (rtsr) -- (rtsrt) node (rtsrWt) [midway] {$rtsrW_t$};
		\draw[edge] (rtsrt) -- (srtsrt) node (rtsrtWs) [midway] {$rtsrtW_s$};
		\draw[edge] (ts) -- (sts) node (tsWt) [midway] {$tsW_t$};
		\draw[edge] (ts) -- (tsr) node (tsWr) [midway] {$tsW_r$};
		\draw[edge] (srtst) -- (srtsrt) node (srtstWr) [midway] {$srtstW_r$};
		\draw[edge] (tsr) -- (rtsr) node (tsrWs) [midway] {$tsrW_s$};
		\draw[edge] (sts) -- (stsr) node (stsWr) [midway] {$stsW_r$};
		\draw[edge] (tsr) -- (stsr) node (tsrWt) [midway] {$tsrW_t$};
		\draw[edge] (stsr) -- (srtsr) node (stsrWs) [midway] {$stsrW_s$};
		\draw[edge] (srtsr) -- (srtsrt) node (srtsrWt) [midway] {$srtsrW_t$};
		\draw[fedge] (Wr.center) -- (tWr.center);
		\draw[fedge] (Wt.center) -- (rWt.center) node (Wrt) [midway] {$W_{rt}$};
		\draw[fedge] (rWt.center) -- (rstWs.center);
		\draw[fedge] (rWs.center) -- (rtsWt.center) node (rWst) [midway] {$rW_{st}$};
		\draw[fedge] (tWr.center) -- (tsrWs.center);
		\draw[fedge] (tWs.center) -- (rtsWr.center) node (tWrs) [midway] {$tW_{rs}$};
		\draw[fedge] (rtsWt.center) -- (rtsrWt.center);
		\draw[fedge] (rtsWr.center) -- (rtstWr.center) node (rtsWrt) [midway] {$rtsW_{rt}$};
		\draw[fedge] (rstWr.center) -- (rtsrtWs.center);
		\draw[fedge] (rstWs.center) -- (srtstWr.center) node (rstWrs) [midway] {$rstW_{rs}$};
		\draw[fedge] (tsrWt.center) -- (rtsrtWs.center);
		\draw[fedge] (tsrWs.center) -- (srtsrWt.center) node (tsrWst) [midway] {$tsrW_{st}$};
		\draw[fedge] (tsWt.center) -- (tsrWt.center);
		\draw[fedge] (tsWr.center) -- (stsWr.center) node (tsWrt) [midway] {$\qquad\quad tsW_{rt}$};
		\draw[cedge] (Wrs.center) -- (tsrWst.center);
		\draw[cedge] (Wrt.center) -- (srtsWrt.center);
		\draw[cedge] (Wst.center) -- (rstWrs.center);
		\fill[facet] (srtsr) -- (stsr) -- (tsr) -- (rtsr) -- (rtsrt) -- (srtsrt) -- cycle {};
		\fill[facet] (t) -- (e) -- (r) -- (rt) -- cycle {};
		\fill[facet] (rtst) -- (rst) -- (rs) -- (r) -- (rt) -- (rts) -- cycle {};
		\fill[facet] (stsr) -- (tsr) -- (ts) -- (sts) -- cycle {};
		\fill[facet] (srtsrt) -- (rtsrt) -- (rtst) -- (rst) -- (srst) -- (srtst) -- cycle {};
		\fill[facet] (rtsr) -- (rts) -- (rtst) -- (rtsrt) -- cycle {};
		\fill[facet] (tsr) -- (ts) -- (t) -- (rt) -- (rts) -- (rtsr) -- cycle {};
		\node[vertex] at (r) {$r$};
		\node[vertex] at (rs) {$rs$};
		\node[vertex] at (e) {$e$};
		\node[vertex] at (rt) {$rt$};
		\node[vertex] at (rst) {$rst$};
		\node[vertex] at (t) {$t$};
		\node[vertex] at (rts) {$rts$};
		\node[vertex] at (rtst) {$rtst$};
		\node[vertex] at (srst) {$srst$};
		\node[vertex] at (rtsrt) {$rtsrt$};
		\node[vertex] at (ts) {$ts$};
		\node[vertex] at (srtst) {$srtst$};
		\node[vertex] at (rtsr) {$rtsr$};
		\node[vertex] at (stsr) {$stsr$};
		\node[vertex] at (sts) {$sts$};
		\node[vertex] at (tsr) {$tsr$};
		\node[vertex] at (srtsrt) {$srtsrt$};
		\node[vertex] at (srtsrt) {$srtsrt$};
		\node[color=green!80!black, right] at (W) {$W$};
	\end{tikzpicture}
}
	\caption{The facial weak order on the standard parabolic cosets of the Coxeter group of type~$A_3$.}
	\label{fig:A3Words}
\end{figure}
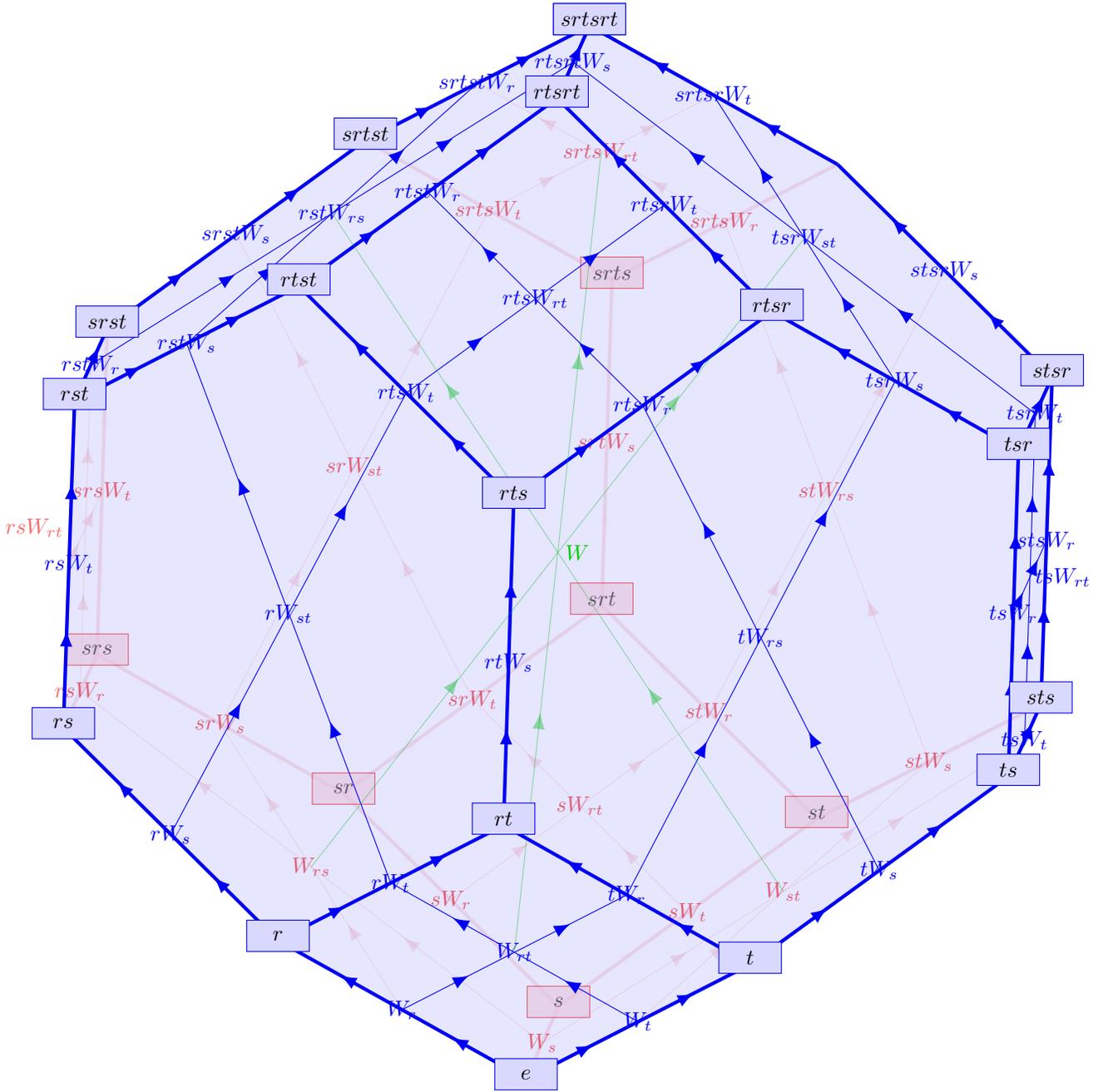

\begin{remark}
\begin{enumerate}
\item These cover relations translate to the following geometric conditions on faces of the permutahedron~$\Perm(W)$: a face~$F$ is covered by a face~$G$ if and only if either $F$ is a facet of~$G$ with the same weak order minimum, or $G$ is a facet of~$F$ with the same weak order maximum.
\item  Consider the natural inclusion  $x\mapsto xW_\varnothing$ from $W$ to $\CoxeterComplex{W}$. For~$x \precdot xs$ in weak order, we have~$xW_\varnothing \precdot xW_{\{s\}} \precdot xsW_\varnothing$ in facial weak order. By transitivity, all relations in the classical weak order are thus relations in the facial weak order. Although it is not obvious at first sight from Definition~\ref{def:facialWeakOrder}, we will see in Corollary~\ref{cor:WeakRestriction} that the restriction of the facial weak order  to the vertices of $\CoxeterComplex{W}$ precisely coincides with the weak order. 

\item It is known that for $I\subseteq S$ the set of minimal length coset representatives~$W^I$ has a maximal length element $\woo \wo{I}$. The element $\wo{I}\wo{I \ssm \{s\}}$ is therefore the maximal length element of the set $W^{I \ssm \{s\}}_I = W_I\cap W^{I \ssm \{s\}}$, which is the set of minimal coset representatives of the cosets $W_I/W_{I \ssm \{s\}}$, see \cite[Section~2.2]{GeckPfeiffer} for more details.
\end{enumerate}
\end{remark}

\begin{example}
\label{exm:weakOrderOrderedPartitions}
As already mentioned, the facial weak order was first considered by D.~Krob, M.~Latapy, J.-C.~Novelli, H.-D.~Phan and S.~Schwer~\cite{KrobLatapyNovelliPhanSchwer} in type~$A$. The standard parabolic cosets in type~$A_{n-1}$ correspond to ordered partitions of~$[n]$, see Example~\ref{exm:typeA}. The weak order on ordered partitions of~$[n]$ is the transitive closure of the cover relations
\begin{gather*}
\text{(1)}\qquad \lambda_1 | \cdots | \lambda_i | \lambda_{i+1} | \cdots | \lambda_k \; \precdot \; \lambda_1 | \cdots | \lambda_i\lambda_{i+1} | \cdots | \lambda_k \qquad\text{if } \lambda_i \ll \lambda_{i+1}, \\
\text{(2)}\qquad \lambda_1 | \cdots | \lambda_i\lambda_{i+1} | \cdots | \lambda_k \; \precdot \; \lambda_1 | \cdots | \lambda_i | \lambda_{i+1} | \cdots | \lambda_k \qquad\text{if } \lambda_{i+1} \ll \lambda_i,
\end{gather*}
where the notation~$X \ll Y$ is defined for~$X, Y \subseteq \N$ by
\[
X \ll Y \; \iff \; \max(X) < \min(Y) \; \iff \; x < y \text{ for all~$x \in X$ and~$y \in Y$.}
\]
\end{example}

This paper gives two convenient characterizations of the facial weak order (see Section~\ref{subsec:weakOrderCharacterizations}). The first one uses sets of roots (see Section~\ref{subsec:rootWeightSet}) to generalize the geometric characterization of the weak order with inversion sets. The second one uses weak order comparisons on the minimal and maximal representatives of the cosets. The advantage of these definitions is that they give immediate global comparisons, whereas the definition of~\cite{PalaciosRonco} uses cover relations. We use these new characterizations of the facial weak order to prove that this poset is in fact a lattice (see Section~\ref{subsec:weakOrderLattice}) and to study some of its order theoretic properties (see Section~\ref{subsec:propertiesFacialWeakOrder}).


\subsection{Root and weight inversion sets of standard parabolic cosets}
\label{subsec:rootWeightSet}

We now define a collection of roots and a collection of weights associated to each standard parabolic coset. The notion of root inversion sets of standard parabolic cosets generalizes the inversion sets of elements of~$W$ (see Proposition~\ref{prop:rootSetSingleElement}). We will use root inversion sets extensively for our study of the facial weak order.  In contrast, weight inversion sets are not essential for our study of the facial weak order but will be relevant when we study its lattice congruences. We define them here as they are polar to the root inversion sets and appear naturally in our geometric intuition of the $W$-Coxeter arrangement and of the $W$-permutahedron (see Proposition~\ref{prop:rootSetWeightSetConesPerm}).

\begin{definition}
\label{def:rootSetWeightSet}
The \defn{root inversion set}~$\rootSet(xW_I)$ and \defn{weight inversion set}~$\weightSet(xW_I)$ of a standard parabolic coset~$xW_I$~are respectively defined by
\[
\rootSet(xW_I) \eqdef x \big( \Phi^- \cup \Phi^+_I \big) \; \subseteq \; \Phi
\qquad\text{and}\qquad
\weightSet(xW_I) \eqdef x \big( \nabla_{S \ssm I} \big) \; \subseteq \; \Omega.
\]
\end{definition}

\begin{remark}
Root inversion sets are known as ``parabolic subsets of roots'' in the sense of~\cite[Section~1.7]{Bourbaki}. In particular for any~$x \in W$, the stabilizer of~$\rootSet(xW_I)$ for the action of $W$ on the subsets of~$\Phi$ is the parabolic subgroup~$x W_I x^{-1}$.
\end{remark}

\begin{example}
\label{exm:halfInversionTable}
Consider the facial weak order on the Coxeter group of type~$A_{n-1}$, see Examples~\ref{exm:typeA} and~\ref{exm:weakOrderOrderedPartitions}. Following~\cite{KrobLatapyNovelliPhanSchwer}, we define the \defn{inversion table} ${\inv(\lambda) \in \{-1,0,1\}^{\binom{n}{2}}}$ of an ordered partition~$\lambda$ of~$[n]$ by
\[
\inv(\lambda)_{i,j} = \begin{cases}
-1 & \text{if } \lambda^{-1}(i) < \lambda^{-1}(j), \\
0 & \text{if } \lambda^{-1}(i) = \lambda^{-1}(j), \\
1 & \text{if } \lambda^{-1}(i) > \lambda^{-1}(j). \\
\end{cases}
\]
The root inversion set of a parabolic coset~$xW_I$ of~$\fS_n$ is encoded by the inversion table of the corresponding ordered partition~$\lambda$. We have
\[
\inv(\lambda)_{i,j} = \begin{cases}
-1 & \text{if } e_i-e_j \in \rootSet(xW_I) \text{ but } e_j-e_i \notin \rootSet(xW_I), \\
0 & \text{if } e_i-e_j \in \rootSet(xW_I) \text{ and } e_j-e_i \in \rootSet(xW_I), \\
1 & \text{if } e_i-e_j \notin \rootSet(xW_I) \text{ but } e_j-e_i \in \rootSet(xW_I). \\
\end{cases}
\]
\end{example}

The following statement gives the precise connection to the geometry of the $W$-permutahedron and is illustrated on \fref{fig:rootSetWeightSet2} for the Coxeter group of type~$A_2$.

\begin{proposition}
\label{prop:rootSetWeightSetConesPerm}
Let~$xW_I$ be a standard parabolic coset of~$W$. Then
\begin{enumerate}[(i)]
\item
\label{item:innerPrimalCone}
$\cone(\rootSet(xW_I))$ is the inner primal cone of the face~$\face(xW_I)$ of~$\Perm(W)$,

\item
\label{item:outerNormalCone}
$\cone(\weightSet(xW_I))$ is the outer normal cone of the face~$\face(xW_I)$ of~$\Perm(W)$,

\item
\label{item:polarCones}
the cones generated by the root inversion set and by the weight inversion set of~$xW_I$ are polar to each other:
\[
\cone(\rootSet(xW_I))\polar = \cone(\weightSet(xW_I)).
\]
\end{enumerate}
\end{proposition}

\begin{proof}
On the one hand, the inner primal cone of~$\face(W_I)$ is generated by the vectors~${\Phi^- \cup \Phi^+_I = \rootSet(eW_I)}$. On the other hand, the outer normal cone of~$\face(W_I)$ is generated by the normal vectors of~$\face(W_I)$, \ie by~$\nabla_{S \ssm I} = \weightSet(eW_I)$. The first two points then follow by applying the orthogonal transformation~$x$ and the last point is an immediate consequence of the first two.
\end{proof}

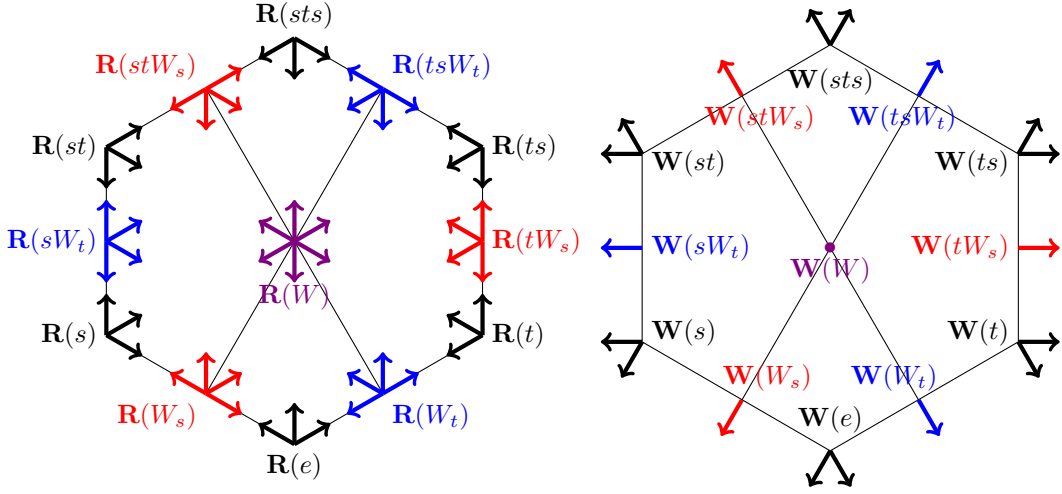
\begin{figure}
\DeclareDocumentCommand{\rs}{ O{1.1cm} O{->} m m O{0}} {
	\def \radius {#1}
	\def \inputPoints{#3}
	\def \excludeRoots{#4}
	\def \style {#2}
	\def \initialRotation {#5}

	\pgfmathtruncatemacro{\points}{\inputPoints * 2}
	\pgfmathsetmacro{\degrees}{360 / \points}
	
	\coordinate (0) at (0,0);
	
	\foreach \x in {1,...,\points}{%
		\pgfmathsetmacro{\location}{(\points+(\x-1))*\degrees + \initialRotation}
		
		\coordinate (\x) at (\location:\radius);
	}

	\ifthenelse{\equal{\excludeRoots}{}}{
		\foreach \x in {1,...,\points}{%
			\draw[\style, ultra thick] (0) -- (\x);
		}
	}{
		\foreach \x in {1,...,\points}{%
			\edef \showPoint {1};

			\foreach \y in \excludeRoots {
				\ifthenelse{\equal{\x}{\y}}{
					\xdef \showPoint {0};
				}{}
			}
			
			\ifthenelse{\equal{\showPoint}{1}}{
				\draw[->, ultra thick] (0) -- (\x);
			}{}
		}
	}  
}

\DeclareDocumentCommand{\ws}{ O{1.1cm} O{->} m m O{30}} {\rs[#1][#2]{#3}{#4}[#5]}

\centerline{
	\begin{tikzpicture}
		[scale=2.5,
		aface/.style={color=blue},
		bface/.style={color=red},
		face/.style={color=red!50!blue}
		]
		%
		\coordinate (e) at (0,0.42);
		\coordinate (s) at (-1,1);
		\coordinate (t) at (1,1);
		\coordinate (st) at (-1,2);
		\coordinate (ts) at (1,2);
		\coordinate (sts) at (0,2.58);
		%
		\coordinate (Ws) at (-0.47,0.69);
		\coordinate (Wt) at (0.47,0.69);
		\coordinate (tWs) at (1,1.5);
		\coordinate (sWt) at (-1,1.5);
		\coordinate (stWs) at (-0.47,2.31);
		\coordinate (tsWt) at (0.47,2.31);
		%
		\coordinate (W) at (0,1.5);
		%
		\draw (e) -- (s);
		\draw (e) -- (t);
		\draw (s) -- (st);
		\draw (t) -- (ts);
		\draw (st) -- (sts);
		\draw (ts) -- (sts);
		%
		\draw (Ws) -- (W);
		\draw (Wt) -- (W);
		\draw (W) -- (stWs);
		\draw (W) -- (tsWt);
		%
		\begin{scope}[shift={(e)}, scale=0.25]
			\node[below] {$\rootSet(e)$};
			\begin{scope}[scale=0.8]
				\rs{3}{1,2,3}[210]
			\end{scope}
		\end{scope}
		\begin{scope}[shift={(s)}, scale=0.25]
			\node[left] {$\rootSet(s)$};
			\begin{scope}[scale=0.8]
				\rs{3}{1,2,6}[210]
			\end{scope}
		\end{scope}
		\begin{scope}[shift={(t)}, scale=0.25]
			\node[right] {$\rootSet(t)$};
			\begin{scope}[scale=0.8]
				\rs{3}{2,3,4}[210]
			\end{scope}
		\end{scope}
		\begin{scope}[shift={(st)}, scale=0.25]
			\node[left] {$\rootSet(st)$};
			\begin{scope}[scale=0.8]
				\rs{3}{1,5,6}[210]
			\end{scope}
		\end{scope}
		\begin{scope}[shift={(ts)}, scale=0.25]
			\node[right] {$\rootSet(ts)$};
			\begin{scope}[scale=0.8]
				\rs{3}{3,4,5}[210]
			\end{scope}
		\end{scope}
		\begin{scope}[shift={(sts)}, scale=0.25]
			\node[above] {$\rootSet(sts)$};
			\begin{scope}[scale=0.8]
				\rs{3}{4,5,6}[210]
			\end{scope}
		\end{scope}
		%
		\begin{scope}[shift={(Ws)}, scale=0.25,bface]
			\node[bface, below left] {$\rootSet(W_{s})$};
			\begin{scope}[scale=0.8]
				\rs{3}{1,2}[210]
			\end{scope}
		\end{scope}
		\begin{scope}[shift={(Wt)}, scale=0.25, aface]
			\node[aface,below right] {$\rootSet(W_{t})$};
			\begin{scope}[scale=0.8]
				\rs{3}{2,3}[210]
			\end{scope}
		\end{scope}
		\begin{scope}[shift={(sWt)}, scale=0.25, aface]
			\node[aface, left] {$\rootSet(sW_{t})$};
			\begin{scope}[scale=0.8]
				\rs{3}{1,6}[210]
			\end{scope}
		\end{scope}
		\begin{scope}[shift={(tWs)}, scale=0.25,bface]
			\node[bface, right] {$\rootSet(tW_{s})$};
			\begin{scope}[scale=0.8]
				\rs{3}{3,4}[210]
			\end{scope}
		\end{scope}
		\begin{scope}[shift={(stWs)}, scale=0.25,bface]
			\node[bface,above left] {$\rootSet(stW_{s})$};
			\begin{scope}[scale=0.8]
				\rs{3}{5,6}[210]
			\end{scope}
		\end{scope}
		\begin{scope}[shift={(tsWt)}, scale=0.25, aface]
			\node[aface, above right] {$\rootSet(tsW_{t})$};
			\begin{scope}[scale=0.8]
				\rs{3}{4,5}[210]
			\end{scope}
		\end{scope}
		%
		\begin{scope}[shift={(W)}, scale=0.25,face]
			\node[below,face] at (0, -.7) {$\rootSet(W)$};
			\begin{scope}[scale=0.8]
				\rs{3}{}[210]
			\end{scope}
		\end{scope}
	\end{tikzpicture}
	\begin{tikzpicture}
		[scale=2.5,
		aface/.style={color=blue},
		bface/.style={color=red},
		face/.style={color=red!50!blue}
		]
		%
		\coordinate (e) at (0,0.42);
		\coordinate (s) at (-1,1);
		\coordinate (t) at (1,1);
		\coordinate (st) at (-1,2);
		\coordinate (ts) at (1,2);
		\coordinate (sts) at (0,2.58);
		%
		\coordinate (Ws) at (-0.47,0.69);
		\coordinate (Wt) at (0.47,0.69);
		\coordinate (tWs) at (1,1.5);
		\coordinate (sWt) at (-1,1.5);
		\coordinate (stWs) at (-0.47,2.31);
		\coordinate (tsWt) at (0.47,2.31);
		%
		\coordinate (W) at (0,1.5);
		%
		\draw (e) -- (s);
		\draw (e) -- (t);
		\draw (s) -- (st);
		\draw (t) -- (ts);
		\draw (st) -- (sts);
		\draw (ts) -- (sts);
		%
		\draw (Ws) -- (W);
		\draw (Wt) -- (W);
		\draw (W) -- (stWs);
		\draw (W) -- (tsWt);
		%
		\begin{scope}[shift={(e)}, scale=0.25]
			\node[above] at (0, .2) {$\weightSet(e)$};
			\begin{scope}[scale=0.8]
				\ws{3}{3,4,5,6}[240]
			\end{scope}
		\end{scope}
		\begin{scope}[shift={(s)}, scale=0.25]
			\node[right] at (0, .2) {$\weightSet(s)$};
			\begin{scope}[scale=0.8]
				\ws{3}{2,3,4,5}[240]
			\end{scope}
		\end{scope}
		\begin{scope}[shift={(t)}, scale=0.25]
			\node[left] at (0, .2) {$\weightSet(t)$};
			\begin{scope}[scale=0.8]
				\ws{3}{1,4,5,6}[240]
			\end{scope}
		\end{scope}
		\begin{scope}[shift={(st)}, scale=0.25]
			\node[right] at (0, -.2) {$\weightSet(st)$};
			\begin{scope}[scale=0.8]
				\ws{3}{1,2,3,4}[240]
			\end{scope}
		\end{scope}
		\begin{scope}[shift={(ts)}, scale=0.25]
			\node[left] at (0, -.2) {$\weightSet(ts)$};
			\begin{scope}[scale=0.8]
				\ws{3}{1,2,5,6}[240]
			\end{scope}
		\end{scope}
		\begin{scope}[shift={(sts)}, scale=0.25]
			\node[below] at (0, -.3) {$\weightSet(sts)$};
			\begin{scope}[scale=0.8]
				\ws{3}{1,2,3,6}[240]
			\end{scope}
		\end{scope}
		%
		\begin{scope}[shift={(Ws)}, scale=0.25,bface]
			\node[bface, above right] at (-.6, 0) {$\weightSet(W_{s})$};
			\begin{scope}[scale=0.8]
				\ws{3}{2,3,4,5,6}[240]
			\end{scope}
		\end{scope}
		\begin{scope}[shift={(Wt)}, scale=0.25, aface]
			\node[aface,above left] at (.6, 0) {$\weightSet(W_{t})$};
			\begin{scope}[scale=0.8]
				\ws{3}{1,3,4,5,6}[240]
			\end{scope}
		\end{scope}
		\begin{scope}[shift={(sWt)}, scale=0.25, aface]
			\node[aface, right] {$\weightSet(sW_{t})$};
			\begin{scope}[scale=0.8]
				\ws{3}{1,2,3,4,5}[240]
			\end{scope}
		\end{scope}
		\begin{scope}[shift={(tWs)}, scale=0.25,bface]
			\node[bface, left] {$\weightSet(tW_{s})$};
			\begin{scope}[scale=0.8]
				\ws{3}{1,2,4,5,6}[240]
			\end{scope}
		\end{scope}
		\begin{scope}[shift={(stWs)}, scale=0.25,bface]
			\node[bface,below right] at (-1, 0) {$\weightSet(stW_{s})$};
			\begin{scope}[scale=0.8]
				\ws{3}{1,2,3,4,6}[240]
			\end{scope}
		\end{scope}
		\begin{scope}[shift={(tsWt)}, scale=0.25, aface]
			\node[aface, below left] at (1, 0) {$\weightSet(tsW_{t})$};
			\begin{scope}[scale=0.8]
				\ws{3}{1,2,3,5,6}[240]
			\end{scope}
		\end{scope}
		%
		\begin{scope}[shift={(W)}, scale=0.25,face]
			\node[below,face] {$\weightSet(W)$};
			\begin{scope}[scale=0.8]
				\ws{3}{1,2,3,4,5,6}[240]
				\node at (0) {$\bullet$};
			\end{scope}
		\end{scope}
	\end{tikzpicture}
}
	\caption{The root inversion sets (left) and weight inversion sets (right) of the standard parabolic cosets in type~$A_2$. Note that positive roots point downwards.}
	\label{fig:rootSetWeightSet2}
\end{figure}

It is well-known that the map $\inversionSet$, sending  an element $w \in W$ to its inversion set $\inversionSet(w) = \Phi^+ \, \cap \, w(\Phi^-)$ is injective, see for instance~\cite[Section~2]{HohlwegLabbe}. The following corollary is the analogue for the maps~$\rootSet$ and~$\weightSet$. 

\begin{corollary}
\label{coro:rootSetWeightWetInjective}
The maps~$\rootSet$ and~$\weightSet$ are both injective.
\end{corollary}

\begin{proof}
A face of a polytope is characterized by its inner primal cone (resp. outer normal cone).
\end{proof}

In a finite Coxeter group, a subset~$R$ of~$\Phi^+$ is an inversion set if and only if it is separable from its complement by a linear hyperplane, or equivalently if and only if both~$R$ and its complement~$\Phi^+ \ssm R$ are convex (meaning that~$R = \Phi^+ \cap \cone(R)$). The following statement gives an analogue for root inversion sets.

\begin{corollary}
\label{coro:characterizationRootSets}
The following assertions are equivalent for a subset~$R$ of~$\Phi$:
\begin{enumerate}[(i)]
\item
\label{item:characterizationRootSetCoset}
$R = \rootSet(xW_I)$ for some coset~$xW_I \in \CoxeterComplex{W}$,

\item
\label{item:characterizationRootSetLinearFunction}
$R = \set{\alpha \in \Phi}{\psi(\alpha) \ge 0}$ for some linear function~$\psi : V \to \R$,

\item
\label{item:characterizationRootSetConvex}
$R = \Phi \cap \cone(R)$ and $R \cap \{\pm \alpha\} \ne \varnothing$ for all~$\alpha \in \Phi$.
\end{enumerate}
\end{corollary}

\begin{proof}
According to Proposition~\ref{prop:rootSetWeightSetConesPerm}, for any coset~$xW_I$, the set~$\rootSet(xW_I)$ is the set of roots in the inner normal cone of the face~$\face(xW_I)$ of~$\Perm(W)$. For any linear function~$\psi : V \to \R$, the set~$\set{\alpha \in \Phi}{\psi(\alpha) \ge 0}$ is the set of roots in the inner normal cone of the face of~$\Perm(W)$ defined by~$\psi$. Since any face is defined by at least one linear function and any linear function defines a face, we get~\eqref{item:characterizationRootSetCoset}$\iff$\eqref{item:characterizationRootSetLinearFunction}. The equivalence~\eqref{item:characterizationRootSetLinearFunction}$\iff$\eqref{item:characterizationRootSetConvex} is immediate.
\end{proof}

Our next three statements concern the root inversion set~$\rootSet(xW_\varnothing)$ for~$x \in W$. For brevity we write~$\rootSet(x)$ instead of~$\rootSet(xW_\varnothing)$. We first connect the root inversion set~$\rootSet(x)$ to the inversion set~$\inversionSet(x)$, to reduced words for~$x$, and to the root inversion sets~$\rootSet(x\woo)$ and~$\rootSet(\woo x)$.

\begin{proposition}
\label{prop:rootSetSingleElement}
For any~$x \in W$,  the root inversion set~$\rootSet(x)$ has the following properties.
\begin{enumerate}[(i)]

\item
\label{item:rootSetSingleElementFromInversionSet}
${\rootSet(x) = \inversionSet(x) \cup - \big( \Phi^+ \ssm \inversionSet(x) \big)}$ where~${\inversionSet(x) = \Phi^+ \cap x(\Phi^-)}$ is the (left) inversion set of~$x$. In other words, 
\[
\rootSet(x) \cap \Phi^+ = \inversionSet(x)
\qquad\text{and}\qquad
\rootSet(x) \cap \Phi^- = -\big( \Phi^+ \ssm \inversionSet(x) \big).
\]

\item
\label{item:rootSetSingleElementFromReducedExpression}
If~$x = s_1 s_2 \cdots s_k$ is reduced, then
\[
\rootSet(x) = \Phi^- \symdif \{\pm \alpha_{s_1}, \pm s_1(\alpha_{s_2}), \dots, \pm s_1 \cdots s_{k-1}(\alpha_{s_k})\}.
\]

\item
\label{item:rootSetSingleElementwo}
$\rootSet(x \woo) = -\rootSet(x)$ and $\rootSet(\woo x) = \woo \big( \rootSet(x) \big)$.
\end{enumerate}
\end{proposition}

\begin{proof}
For~\eqref{item:rootSetSingleElementFromInversionSet} we observe that $\rootSet(x) = x(\Phi^-) = \big( \Phi^+ \cap x(\Phi^-) \big) \cup \big( \Phi^- \cap x(\Phi^-) \big)$. By definition of the inversion set we get
\[
\rootSet(x) = \inversionSet(x) \cup - \big( \Phi^+ \cap x(\Phi^+) \big) = \inversionSet(x) \cup - \big( \Phi^+ \ssm \inversionSet(x) \big).
\]
\eqref{item:rootSetSingleElementFromReducedExpression} then follows from the fact that~$\inversionSet(x) = \{ \alpha_{s_1}, s_1(\alpha_{s_2}), \dots, s_1 \cdots s_{k-1}(\alpha_{s_k})\}$. Finally, \eqref{item:rootSetSingleElementwo} follows from the definition of $\rootSet$ and the fact that $\woo(\Phi^+) = \Phi^-$.
\end{proof}

The next statement gives a characterization of the  (classical)  weak order in terms of root inversion sets, which generalizes the characterization of the weak order in term of inversion sets. We will see later in Theorem~\ref{thm:facialWeakOrderCharacterizations} that the same characterization holds for the facial weak order.

\begin{corollary}
\label{coro:weakOrderFromRootSet}
For~$x,y \in W$, we have
\begin{align*}
x \le y
& \iff \rootSet(x) \ssm \rootSet(y) \subseteq \Phi^- \quad\text{and}\quad \rootSet(y) \ssm \rootSet(x) \subseteq \Phi^+, \\
& \iff \rootSet(x) \cap \Phi^+ \, \subseteq \, \rootSet(y) \cap \Phi^+ \quad\text{and}\quad \rootSet(x) \cap \Phi^- \, \supseteq \, \rootSet(y) \cap \Phi^-.
\end{align*}
\end{corollary}

\begin{proof}
We observe from Proposition~\ref{prop:rootSetSingleElement}\,\eqref{item:rootSetSingleElementFromInversionSet} that
\[
\rootSet(x) \ssm \rootSet(y) = \big( \inversionSet(x) \ssm \inversionSet(y) \big) \cup -\big( \inversionSet(y) \ssm \inversionSet(x) \big).
\]
The result thus follows immediately from the fact that~$x \le y \iff \inversionSet(x) \subseteq \inversionSet(y)$, see Section~\ref{subsec:lengthReducedWordsAndWeakOrder}.
\end{proof}

Finally, we observe that the root and weight inversion sets of a parabolic coset~$xW_I$ can be computed from that of its minimal and maximal length representatives~$x$ and~$x\wo{I}$.

\begin{proposition}
\label{prop:rootSetWeightSetFromBottomAndTop}
The root and weight inversion sets of~$xW_I$ can be computed from those of~$x$ and~$x\wo{I}$ by
\[
\rootSet(xW_I) = \rootSet(x) \cup \rootSet(x\wo{I})
\qquad\text{and}\qquad
\weightSet(xW_I) = \weightSet(x) \cap \weightSet(x\wo{I}).
\]
\end{proposition}

\begin{proof}
For the root inversion set, we just write
\begin{align*}
\rootSet(x) \cup \rootSet(x\wo{I}) & = x(\Phi^-) \cup x\wo{I}(\Phi^-) = x(\Phi^-) \cup x(\Phi^- \symdif \Phi_I) \\ & = x(\Phi^- \cup \Phi^+_I) = \rootSet(xW_I).
\end{align*}
The proof is similar for the weight inversion set (or can be derived from Proposition~\ref{prop:rootSetWeightSetConesPerm}).
\end{proof}

\begin{corollary}
\label{coro:rootSetWeightSetFromBottomAndTop}
For any coset~$xW_I$, we have
\[
\rootSet(xW_I) \cap \Phi^- = \rootSet(x) \cap \Phi^-
\qquad\text{and}\qquad
\rootSet(xW_I) \cap \Phi^+ = \rootSet(x\wo{I}) \cap \Phi^+
\]
\end{corollary}

\begin{proof}
Since~$x \le x\wo{I}$, Corollary~\ref{coro:weakOrderFromRootSet} ensures that~$\rootSet(x) \cap \Phi^+ \subseteq \rootSet(x\wo{I}) \cap \Phi^+$ and~$\rootSet(x) \cap \Phi^- \supseteq \rootSet(x\wo{I}) \cap \Phi^-$. Therefore, we obtain from Proposition~\ref{prop:rootSetWeightSetFromBottomAndTop} that
\[
\rootSet(xW_I) \cap \Phi^- = \big( \rootSet(x) \cup \rootSet(x\wo{I}) \big) \cap \Phi^- = \rootSet(x) \cap \Phi^-,
\]
and similarly
\[
\rootSet(xW_I) \cap \Phi^+ = \big( \rootSet(x) \cup \rootSet(x\wo{I}) \big) \cap \Phi^+ = \rootSet(x\wo{I}) \cap \Phi^+.
\qedhere
\]
\end{proof}


\subsection{Two alternative characterizations of the facial weak order}
\label{subsec:weakOrderCharacterizations}

Using the root inversion sets defined in the previous section, we now give two equivalent characterizations of the facial weak order defined by P.~Palacios and M.~Ronco in~\cite{PalaciosRonco} (see Definition~\ref{def:facialWeakOrder}).  In type~$A$, the equivalence~\mbox{\eqref{item:facialWeakOrderCoverCharacterizationRelations}$\iff$\eqref{item:facialWeakOrderCharacterizationRoots}} below is stated in~\cite[Theorem~5]{KrobLatapyNovelliPhanSchwer} in terms of half-inversion tables (see Examples~\ref{exm:weakOrderOrderedPartitions} and~\ref{exm:halfInversionTable}).

\begin{theorem}
\label{thm:facialWeakOrderCharacterizations}
The following conditions are equivalent for two standard parabolic cosets~$xW_I$ and~$yW_J$ in $\mathcal P_W$:
\begin{enumerate}[(i)]
\item
\label{item:facialWeakOrderCoverCharacterizationRelations}
$xW_I\leq yW_J$ in facial weak order,

\item
\label{item:facialWeakOrderCharacterizationRoots}
$\rootSet(xW_I) \ssm \rootSet(yW_J) \subseteq \Phi^-$ and~$\rootSet(yW_J) \ssm \rootSet(xW_I) \subseteq \Phi^+$,

\item
\label{item:facialWeakOrderCharacterizationCompareMinMax}
$x \le y$ and~$x\wo{I} \le y\wo{J}$ in weak order.
\end{enumerate}
\end{theorem}

\begin{proof}[Proof]
We will prove that~\eqref{item:facialWeakOrderCoverCharacterizationRelations}$\implies$\eqref{item:facialWeakOrderCharacterizationCompareMinMax}$\implies$\eqref{item:facialWeakOrderCharacterizationRoots}$\implies$\eqref{item:facialWeakOrderCoverCharacterizationRelations}, the last implication being the most technical.

\medskip

The implication~\eqref{item:facialWeakOrderCoverCharacterizationRelations}$\implies$\eqref{item:facialWeakOrderCharacterizationCompareMinMax} is immediate. The first cover relation keeps~$x$ and transforms~$x\wo{I}$ to~$x\wo{I \cup \{s\}}$, and the second cover relation transforms~$x$ to~$x\wo{I}\wo{I \ssm \{s\}}$ but keeps~$x\wo{I}$. Since~$x\wo{I} \le x\wo{I \cup \{s\}}$ and~${x \le x\wo{I}\wo{I \ssm \{s\}}}$, we obtain the result by transitivity.

\medskip

For the implication~\eqref{item:facialWeakOrderCharacterizationCompareMinMax}$\implies$\eqref{item:facialWeakOrderCharacterizationRoots}, Corollary~\ref{coro:weakOrderFromRootSet} ensures that~${\rootSet(x) \ssm \rootSet(y) \subseteq \Phi^-}$ and ${\rootSet(y) \ssm \rootSet(x) \subseteq \Phi^+}$ since~${x \le y}$, and similarly that~${\rootSet(x\wo{I}) \ssm \rootSet(y\wo{J}) \subseteq \Phi^-}$ and ${\rootSet(y\wo{J}) \ssm \rootSet(x\wo{I}) \subseteq \Phi^+}$ since~${x\wo{I} \le y\wo{J}}$. From Pro\-position~\ref{prop:rootSetWeightSetFromBottomAndTop}, we therefore obtain
\begin{align*}
\rootSet(xW_I) \ssm \rootSet(yW_J) & = \big( \rootSet(x) \cup \rootSet(x\wo{I}) \big) \ssm \big( \rootSet(y) \cup \rootSet(y\wo{J}) \big) \\
& \subseteq \big( \rootSet(x) \ssm \rootSet(y) \big) \cup \big( \rootSet(x\wo{I}) \ssm \rootSet(y\wo{J}) \big) \\
& \subseteq \Phi^-.
\end{align*}
We prove similarly that~$\rootSet(yW_J) \ssm \rootSet(xW_I) \subseteq \Phi^+$.

\medskip

We now focus on the implication~\eqref{item:facialWeakOrderCharacterizationRoots}$\implies$\eqref{item:facialWeakOrderCoverCharacterizationRelations}. We consider two standard parabolic cosets~$xW_I$ and~$yW_J$ which satisfy Condition~\eqref{item:facialWeakOrderCharacterizationRoots} and construct a path of cover relations as in Definition~\ref{def:facialWeakOrder} between them. We proceed by induction on the cardinality~$|\rootSet(xW_I) \symdif \rootSet(yW_J)|$. 

First, if $|\rootSet(xW_I) \symdif \rootSet(yW_J)| = 0$, then~$\rootSet(xW_I) = \rootSet(yW_J)$, which ensures that~$xW_I = yW_J$ by Corollary~\ref{coro:rootSetWeightWetInjective}. Assume now that $|\rootSet(xW_I) \symdif \rootSet(yW_J)| > 0$. So we either have~$\rootSet(xW_I) \ssm \rootSet(yW_J) \ne \varnothing$ or~$\rootSet(yW_J) \ssm \rootSet(xW_I) \ne \varnothing$. We consider only the case ${\rootSet(xW_I) \ssm \rootSet(yW_J) \ne \varnothing}$, the other case being symmetric.

To proceed by induction, our goal is to find a new coset~$zW_K$ so that
\begin{itemize}
\item $xW_I \precdot zW_K$ is one of the cover relations of Definition~\ref{def:facialWeakOrder},
\item $zW_K$ and~$yW_J$ still satisfy Condition~\eqref{item:facialWeakOrderCharacterizationRoots}, and
\item $\rootSet(zW_K) \symdif \rootSet(yW_J) \; \subsetneq \; \rootSet(xW_I) \symdif \rootSet(yW_J)$.
\end{itemize}
Indeed, by induction hypothesis, there will exist a path from~$zW_K$ to~$yW_J$ consisting of cover relations as in Definition~\ref{def:facialWeakOrder}. Adding the first step~$xW_I \precdot zW_K$, we then obtain a path from~$xW_I$ to~$yW_J$.

To construct this new coset~$zW_K$ and its root inversion set~$\rootSet(zW_K)$, we will add or delete (at least) one root from~$\rootSet(xW_I)$. We first claim that there exists~${s \in S}$ such that~$-x(\alpha_s) \notin \rootSet(yW_J)$. Otherwise, we would have~${x(-\Delta) \subseteq \rootSet(yW_J)}$. Since $\Phi^-=\cone(-\Delta)\cap \Phi$ and~$\rootSet(yW_J) = \cone(\rootSet(yW_J)) \cap \Phi$, this would imply that ${x(\Phi^-) \subseteq \rootSet(yW_J)}$. 
Moreover, ${x(\Phi_I^+) \subseteq \Phi^+}$ since~$x \in W^I$. Thus we would obtain
\begin{align*}
\rootSet(xW_I) \ssm \rootSet(yW_J)
& = \big( x(\Phi^-) \cup x(\Phi_I^+) \big) \ssm \rootSet(yW_J) \\
& \subseteq \big( x(\Phi^-) \ssm \rootSet(yW_J) \big) \cup x(\Phi_I^+) \\
& \subseteq \Phi^+.
\end{align*}
However, $\rootSet(xW_I) \ssm \rootSet(yW_J) \subseteq \Phi^-$ by Condition~\eqref{item:facialWeakOrderCharacterizationRoots}. Hence we would obtain that ${\rootSet(xW_I) \ssm \rootSet(yW_J) \subseteq \Phi^+ \cap \Phi^- = \varnothing}$, contradicting  our assumption.

\medskip

For the remaining of the proof we fix~$s \in S$ such that~$-x(\alpha_s) \notin \rootSet(yW_J)$ and we set $\beta \eqdef x(\alpha_s)$. By definition, we have~${-\beta \in \rootSet(xW_I) \ssm \rootSet(yW_J)}$. Moreover, since~${-\beta \notin \rootSet(yW_J)}$ and ${\rootSet(yW_J) \cup -\rootSet(yW_J) \supseteq y(\Phi^-) \cup -y(\Phi^-) = \Phi}$, we have~${\beta \in \rootSet(yW_J)}$. We now distinguish two cases according to whether or not ${\beta \in \rootSet(xW_I)}$, that is, on whether or not~$s \in I$. In both cases, we will need the following observation: since~${\rootSet(xW_I) \ssm \rootSet(yW_J) \subseteq \Phi^-}$, we have
\begin{equation}
\label{eqn: xphi+ subset rootset ywj}
x(\Phi_I^+) \subseteq \Phi^+ \cap \rootSet(xW_I) \subseteq \rootSet(yW_J). \tag{$\star$}
\end{equation}

\para{First case:~$s \notin I$}
Since~$-x(\alpha_s) = -\beta \in \rootSet(xW_I) \ssm \rootSet(yW_J) \subseteq \Phi^-$, we have that~${x(\alpha_s) \in \Phi^+}$ and thus~$x \in W^{I \cup \{s\}}$. We can therefore consider the standard parabolic coset~$zW_K \eqdef xW_{I \cup \{s\}}$ where~$z \eqdef x$ and~$K \eqdef I \cup \{s\}$. Its root inversion set is given by~${\rootSet(zW_K) = \rootSet(xW_I) \cup x(\Phi_K^+)}$. Note that $xW_I\precdot zW_K$ is a cover relation of type (1) in Definition~\ref{def:facialWeakOrder}. It thus remains to show that~$zW_K$ and~$yW_J$ still satisfy Condition~\eqref{item:facialWeakOrderCharacterizationRoots} and that~$\rootSet(zW_K) \symdif \rootSet(yW_J) \;\subsetneq\; \rootSet(xW_I) \symdif \rootSet(yW_J)$.

Since $\beta=x(\alpha_s) \in \rootSet(yW_J)$ and using Observation~\eqref{eqn: xphi+ subset rootset ywj} above, we thus have
\[
x(\Phi_K^+) = \cone \big( \{\beta\} \cup x(\Phi^+_I) \big) \cap \Phi \subseteq \rootSet(yW_J).
\]
Therefore we obtain
\[
\rootSet(zW_K) \ssm \rootSet(yW_J) = \rootSet(xW_I) \cup x(\Phi_K^+) \ssm \rootSet(yW_J) \subseteq \rootSet(xW_I) \ssm \rootSet(yW_J) \subseteq \Phi^-.
\]
Moreover, since ${\rootSet(xW_I) \subseteq \rootSet(zW_K)}$,
\[
\rootSet(yW_J) \ssm \rootSet(zW_K) \subseteq \rootSet(yW_J) \ssm \rootSet(xW_I) \subseteq \Phi^+.
\]
Therefore, we proved that the cosets~$zW_K$ and $yW_J$ still satisfy Condition~\eqref{item:facialWeakOrderCharacterizationRoots} and that $\rootSet(zW_K) \symdif \rootSet(yW_J) \subseteq \rootSet(xW_I) \symdif \rootSet(yW_J)$. The strict inclusion then follows since $-\beta$ belongs to~$\rootSet(xW_I) \symdif \rootSet(yW_J)$ but not to~$\rootSet(zW_K) \symdif \rootSet(yW_J)$.

\para{Second case:~$s \in I$}
Let~$s^\star \eqdef \wo{I} s \wo{I}$. Consider the standard parabolic coset~$zW_K$ where~$K \eqdef I \ssm \{s^\star\}$ and~$z \eqdef x \wo{I} \wo{K}$. Note that~$xW_I \precdot  zW_K$ is a cover relation of type (2) in Definition~\ref{def:facialWeakOrder}. It thus remains to show that~$zW_K$ and~$yW_J$ still satisfy Condition~\eqref{item:facialWeakOrderCharacterizationRoots} and that~$\rootSet(zW_K) \symdif \rootSet(yW_J) \;\subsetneq\; \rootSet(xW_I) \symdif \rootSet(yW_J)$.

We first prove that
\begin{equation}
\label{eq:proof1}
\rootSet(xW_I) = \rootSet(zW_K) \cup x(\Phi_I^-) \ssm x(\Phi_{I \ssm \{s\}}^-). \tag{$\clubsuit$}
\end{equation}
Observe that~$\wo{I}(\Phi^-) = \Phi^- \symdif \Phi_I$ and that~$\wo{I}(\Phi_K) = \wo{I}(\Phi_{I \ssm \{s^\star\}}) = \Phi_{I \ssm \{s\}}$. Therefore,
\begin{align*}
& \wo{I} \wo{K} (\Phi^-) = \wo{I} (\Phi^- \symdif \Phi_K) = \Phi^- \symdif (\Phi_I \ssm \Phi_{I \ssm \{s\}}) \\
\quad\text{and}\quad
& \wo{I} \wo{K} (\Phi_K^+) = \wo{I} (\Phi_K^-) = \Phi_{I \ssm \{s\}}^+.
\end{align*}
Therefore we obtain the desired equality:
\begin{align*}
\rootSet(xW_I) 
& = x(\Phi^- \cup \Phi_I^+) \\
& = x \big( \Phi^- \symdif (\Phi_I \ssm \Phi_{I \ssm \{s\}}) \big) \cup x(\Phi_{I \ssm \{s\}}^+) \cup x(\Phi_I^-) \ssm x(\Phi_{I \ssm \{s\}}^-)  \\
& = x \wo{I} \wo{K} (\Phi^-) \cup x \wo{I} \wo{K} (\Phi_K^+) \cup x(\Phi_I^-) \ssm x(\Phi_{I \ssm \{s\}}^-) \\
& = \rootSet(zW_K) \cup x(\Phi_I^-) \ssm x(\Phi_{I \ssm \{s\}}^-). 
\end{align*}

We now check that
\begin{equation}
\label{eq:proof2}
\big( x(\Phi_I^-) \ssm x(\Phi_{I \ssm \{s\}}^-) \big) \cap \rootSet(yW_J) = \varnothing. \tag{$\spadesuit$}
\end{equation}
Indeed, assume that this set contains an element~$\delta$. We have~$\delta = a(-\beta) + \gamma$, where $a>0$ and $\gamma \in x(\Phi_{I \ssm \{s\}}^-)$. Therefore, $-\beta = (\delta-\gamma)/a$. Since~${\delta \in \rootSet(yW_J)}$ and~${-\gamma \in x(\Phi_{I \ssm \{s\}}^+) \subseteq x(\Phi_I^+) \subseteq \rootSet(yW_J)}$ by Observation~\eqref{eqn: xphi+ subset rootset ywj} above, we would obtain that ${-\beta \in \rootSet(yW_J)}$, contradicting our definition of~$\beta$.

Now, combining Equations~\eqref{eq:proof1} and~\eqref{eq:proof2}, we obtain
\[
\rootSet(yW_J) \ssm \rootSet(zW_K) \subseteq \rootSet(yW_J) \ssm \rootSet(xW_I) \subseteq \Phi^+.
\]
Moreover, since~$\rootSet(zW_K) \subseteq \rootSet(xW_I)$,
\[
\rootSet(zW_K) \ssm \rootSet(yW_J) \subseteq \rootSet(xW_I) \ssm \rootSet(yW_J) \subseteq \Phi^-.
\]
Therefore, we proved that the cosets~$zW_K$ and $yW_J$ still satisfy Condition~\eqref{item:facialWeakOrderCharacterizationRoots} and that $\rootSet(zW_K) \symdif \rootSet(yW_J) \subseteq \rootSet(xW_I) \symdif \rootSet(yW_J)$. The strict inclusion then follows as~$-\beta$ belongs to~$\rootSet(xW_I) \symdif \rootSet(yW_J)$ but not to~$\rootSet(zW_K) \symdif \rootSet(yW_J)$. This concludes the proof.
\end{proof}

\begin{remark}
\label{rem:equivalentRootSetCharacterization}
Observe that our characterization of the facial weak order in terms of root inversion sets given in Theorem~\ref{thm:facialWeakOrderCharacterizations}\,\eqref{item:facialWeakOrderCharacterizationRoots} is equivalent to the following: ${xW_I \le yW_J}$ if and only if
\begin{enumerate}
\item[(ii')] $\rootSet(xW_I) \cap \Phi^+ \, \subseteq \, \rootSet(yW_J) \cap \Phi^+$ and~$\rootSet(xW_I) \cap \Phi^- \, \supseteq \, \rootSet(yW_J) \cap \Phi^-$.
\end{enumerate}
\end{remark}

\begin{example}
\label{exm:A3RootSets}
We have illustrated the facial weak order by the means of root inversion sets in \fref{fig:rootSet3}. In this figure each face is labelled by its root inversion set. To visualize the roots, we consider the affine plane~$P$ passing through the simple roots~$\{\alpha, \beta, \gamma\}$. A positive (resp.~negative) root~$\rho$ is then seen as a red upward (resp.~blue downward) triangle placed at the intersection of~$\R\rho$ with the plane~$P$. For instance,
\[
\rootSet(cbW_a) = \{\gamma, \beta + \gamma, \alpha + \beta + \gamma\} \cup \{-\alpha, -\beta, -\alpha-\beta-\gamma, -\alpha-\beta\}.
\] 
is labeled in \fref{fig:rootSet3} by	

\DeclareDocumentCommand{\rps}{m m m} {
	\begin{scope}[scale=0.08, x={(1cm,0)}, y={(0,1cm)}, z={(0,0)}]
		\def \ppoints {#1}
		\def \mpoints {#2}
		\def \bpoints {#3}
		
		
		\coordinate (100) at (-1, -1);
		\coordinate (111) at (0,-.5);
		\coordinate (001) at (1,-1);
		\coordinate (110) at (-0.5, 0);
		\coordinate (011) at (0.5, 0);
		\coordinate (010) at (0,1);
		
		\foreach \x in \ppoints{%
			\node[] at  (\x) {\color{red}$\bigtriangleup$};
		}
		
		\foreach \x in \mpoints{%
			\node[] at  (\x) {\color{blue}$\bigtriangledown$};
		}
		
		\foreach \x in \bpoints{%
			\node[] at  (\x) {$\color{blue}\bigtriangledown \!\!\!\!\! \color{red}\bigtriangleup$};
		}
	\end{scope}
}

\begin{center}
	\begin{tikzpicture}%
		[scale=8]
		\rps{001,011}{100,010,110}{111}
		\node[left=1mm] at (100) {\small$-\alpha$};
		\node[right=1mm] at (001) {\small$\gamma$};
		\node[left=1mm] at (110) {\small$-\alpha - \beta$};
		\node[right=1mm] at (011) {\small$\beta + \gamma$};
		\node[above=1mm] at (010) {\small$-\beta$};
	\end{tikzpicture}
\end{center}
	
Note that the star in the middle represents both $\alpha + \beta + \gamma$ and $-\alpha - \beta - \gamma$.

\begin{figure}
	\input{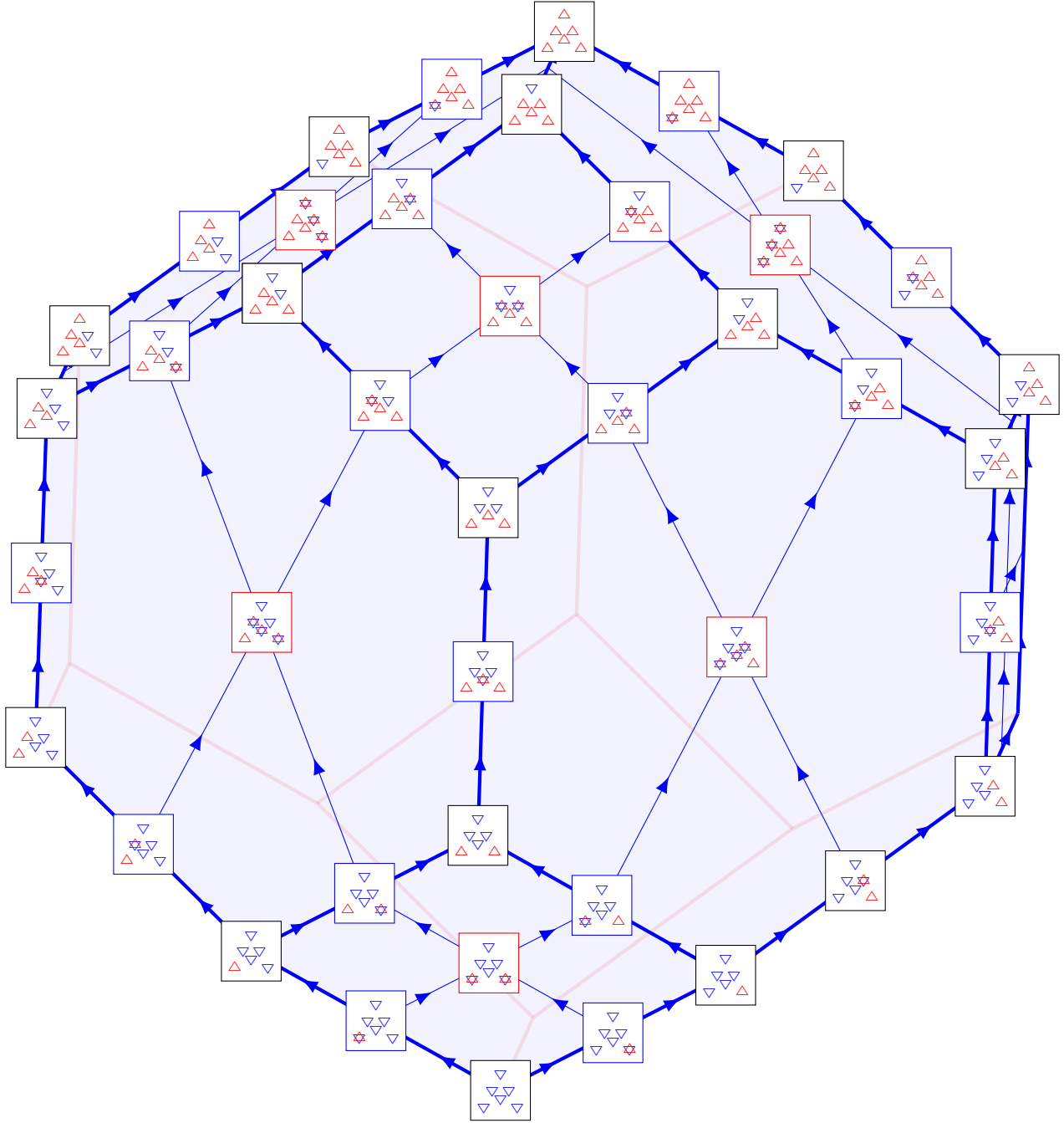}
	\caption{The facial weak order on the standard parabolic cosets of the Coxeter group of type~$A_3$. Each coset~$xW_I$ is replaced by its root inversion set~$\rootSet(xW_I)$, represented as follows: down blue triangles stand for negative roots while up red triangles stand for positive roots, and the position of each triangle is given by the barycentric coordinates of the corresponding root with respect to the three simple roots ($\alpha_1$ on the bottom left, $\alpha_2$ on the top, and $\alpha_3$ on the bottom right). See Example~\ref{exm:A3RootSets} for more details.}
	\label{fig:rootSet3}
\end{figure}
\end{example}

Using our characterization of Theorem~\ref{thm:facialWeakOrderCharacterizations}\,\eqref{item:facialWeakOrderCharacterizationRoots} together with Corollary~\ref{coro:weakOrderFromRootSet}, we obtain that the facial weak order and the weak order coincide on the elements of~$W$. Note that this is not at all obvious with the cover relations from Definition~\ref{def:facialWeakOrder}.

\begin{corollary}
\label{cor:WeakRestriction}
For any~$x,y \in W$, we have $x \le y$ in weak order if and only if~$xW_\varnothing \le yW_\varnothing$ in facial weak order.
\end{corollary}

The weak order anti-automorphisms ${x \mapsto x \woo}$ and~$x \mapsto \woo x$ and the automorphism~$x \mapsto \woo x \woo$ carry out on standard parabolic cosets. The following statement gives the precise definitions of the corresponding maps.

\begin{proposition}
\label{prop:antiautomorphisms}
The maps
\[
xW_I \longmapsto \woo x \wo{I} W_I
\qquad\text{and}\qquad
xW_I \longmapsto x \wo{I} \woo W_{\woo I \woo}
\]
are anti-automorphisms of the facial weak order. Consequently, the map
\[
xW_I \longmapsto \woo x \woo W_{\woo I \woo}
\]
is an automorphism of the facial weak order.
\end{proposition}

\begin{proof}
Using the characterization of the facial weak order given in Theorem~\ref{thm:facialWeakOrderCharacterizations}\,\eqref{item:facialWeakOrderCharacterizationRoots} we just need to observe that
\[
\rootSet(\woo x \wo{I} W_I) = \woo \big( \rootSet(xW_I) \big)
\qquad\text{and}\qquad
\rootSet(x \wo{I} \woo W_{\woo I \woo}) = -\rootSet(xW_I).
\]
This follows immediately from Propositions~\ref{prop:rootSetWeightSetFromBottomAndTop} and~\ref{prop:rootSetSingleElement}\,\eqref{item:rootSetSingleElementwo}.
\end{proof}


\subsection{The facial weak order is a lattice}
\label{subsec:weakOrderLattice}
In this section, we show that the  facial  weak order on standard parabolic cosets is actually a lattice. It generalizes the result for the symmetric group due to D.~Krob, M.~Latapy, J.-C.~Novelli, H.-D.~Phan and S.~Schwer~\cite{KrobLatapyNovelliPhanSchwer} to the facial weak order  on arbitrary finite Coxeter groups \ introduced by P.~Palacios and M.~Ronco~\cite{PalaciosRonco}.  It also gives a precise description of the meets and joins in this lattice.  The characterizations of the facial weak order given in Theorem~\ref{thm:facialWeakOrderCharacterizations} are key here.

\begin{theorem}
\label{thm:lattice}
The facial weak order~$(\CoxeterComplex{W}, \leq)$ is a lattice. The meet and join of two standard parabolic cosets~$xW_I$ and~$yW_J$ are given by:
\begin{gather*}
xW_I \meet yW_J = \zm W_\Km \quad\text{where}\quad \zm = x \meet y \quad\text{and}\quad \Km = \LdescentSet \big( \zm^{-1} (x\wo{I} \meet y\wo{J}) \big), \\
xW_I \join yW_J = \zj  W_\Kj \quad\text{where}\quad \zj = x\wo{I} \join y\wo{J} \quad\text{and}\quad \Kj = \LdescentSet \big( \zj^{-1} (x \join y) \big).
\end{gather*}
\end{theorem}

In other words, the meet~$xW_I \meet yW_J$ is the biggest parabolic coset in the interval~$[x \meet y, x\wo{I} \meet y\wo{J}]$ containing~$x \meet y$. Similarly, the join~$xW_I \join yW_J$ is the biggest parabolic coset in the interval~$[x \join y, x\wo{I} \join y\wo{J}]$ containing~$x\wo{I} \join y\wo{J}$.

Note that in the second point of Theorem~\ref{thm:lattice}, the minimal representative of the coset $\zj W_\Kj$ is in fact $\zj w_{\circ,\Kj}$, not~$\zj$. Unlike in the rest of the paper, we take the liberty to use another coset representative than the minimal one to underline the symmetry between meet and join in the facial weak order.

\begin{example}
\label{exm:computeMeet}
Before proving the above statement, we give two examples of computations of the meet and the join in the facial weak order.
\begin{enumerate}[(i)]

\item Consider first the Coxeter system
\[
\setangle{r,s,t}{r^2 = s^2 = t^2 = (rs)^3 = (st)^3 = (rt)^2 = 1}
\]
of type~$A_3$. \fref{fig:A3Words} shows the facial weak order of $\CoxeterComplex{W}$ and is a good way to follow along. To find the meet of~$tsrW_{st}$ and~$rtsW_\varnothing$, we compute:
\begin{gather*}
\zm = tsr \meet rts = t, \\
\Km = \LdescentSet \big(\zm^{-1} (tsr\wo{st} \meet rts\wo{\varnothing})\big) = \LdescentSet \big(t(tsrsts \meet rts)\big) = \LdescentSet(t (rts)) = \{r\}.
\end{gather*}
Thus we have that $tsrW_{st} \meet rtsW_\varnothing = \zm W_\Km = tW_r$.

\item For a slightly more complex example, consider the Coxeter system
\[
\setangle{r,s,t}{r^2 = s^2 = t^2 = (rs)^4 = (st)^3 = (rt)^2 = 1}
\]
of type $B_3$. To find the join of~$rstW_{rs}$ and~$tsrsW_\varnothing$, we compute:
\begin{gather*}
\zj = rst\wo{rs} \join tsrs\wo{\varnothing} = rstrsrs \join tsrs = rtsrtsrt \\
\Kj = \LdescentSet \big(\zj^{-1} (rst \join tsrs)\big) = \LdescentSet\big(trstrstr(rtsrtst)\big) = \LdescentSet(r) = \{r\}
\end{gather*}	
Thus we see that 
\[
rstW_{rs} \join tsrs\wo{\varnothing} = \zj \wo{\Kj} W_\Kj = rtsrtsrt(r)W_r = rtsrtstW_r.
\]
\end{enumerate}
\end{example}

\begin{proof}[Proof of Theorem~\ref{thm:lattice}]
Throughout the proof we use the characterization of the facial weak order given in  Theorem~\ref{thm:facialWeakOrderCharacterizations}\,\eqref{item:facialWeakOrderCharacterizationRoots}:
\[
xW_I \le yW_J \iff x \le y \text{ and } x \wo{I} \le y \wo{J}.
\]

We first prove the existence of the meet, then use Proposition~\ref{prop:antiautomorphisms} to deduce the existence and formula for the join.

\para{Existence of meet}
For any~$s \in \Km$, we have
\begin{align*}
\length(x\wo{I} \meet y\wo{J}) - \length(s \zm^{-1})
& \le \length \big( s \zm^{-1} (x\wo{I} \meet y\wo{J}) \big) \\
& = \length \big( \zm^{-1} (x\wo{I} \meet y\wo{J}) \big) - 1 \\
& = \length(x\wo{I} \meet y\wo{J}) - \length(\zm^{-1}) - 1.
\end{align*}
Indeed, the first inequality holds in general (for reduced or non-reduced words). The first equality follows from ${s \in \Km = \LdescentSet(\zm^{-1} (x\wo{I} \meet y\wo{J}))}$. The last equality holds since~$\zm = x \meet y \le x\wo{I} \meet y\wo{J}$. We deduce from this inequality that~${\length(\zm) < \length(\zm s)}$. Therefore, we have~$\zm \in W^{\Km}$.

Since $\Km = \LdescentSet(\zm^{-1} (x\wo{I} \meet y\wo{J}))$, we have~$\wo{\Km} \le \zm^{-1} (x\wo{I} \meet y\wo{J})$. Therefore $\zm \wo{\Km} \le x\wo{I}\meet y\wo{J}$, since $\zm \in W^{\Km}$. We thus have~$\zm = x \meet y \le x$ and~$\zm \wo{\Km} \le x\wo{I} \meet y\wo{J} \le x\wo{I}$, which implies~$\zm W_\Km \le xW_I$, by Theorem~\ref{thm:facialWeakOrderCharacterizations}\,\eqref{item:facialWeakOrderCharacterizationRoots}. By symmetry, $\zm W_\Km \le yW_J$.

It remains to show that~$\zm W_\Km$ is the greatest lower bound. Consider a standard parabolic coset~$zW_K$ such that~$zW_K \le xW_I$ and~$zW_K \le yW_J$. We want to show that~${zW_K \le \zm W_\Km}$, that is, $z \le \zm$ and~$z\wo{K} \le \zm\wo{\Km}$. The first inequality is immediate since~$z \le x$ and~$z \le y$ so that~$z \le x \meet y = \zm$. For the second one, we write the following reduced words: $x = zx'$, $y = zy'$, and~$\zm = z\zm'$ where~${\zm' = x' \meet y'}$. Since~$z\wo{K} \le x\wo{I}$ and~$z\wo{K} \le y\wo{J}$, we have
\[
{z\wo{K} \le x\wo{I} \meet y\wo{J} = zx'\wo{I} \meet zy'\wo{J} = z(x'\wo{I} \meet y'\wo{J})}.
\] 
Thus~$\wo{K} \le x'\wo{I} \meet y'\wo{J}$, since all words are reduced here. Therefore  
\[
K \subseteq \LdescentSet(x'\wo{I} \meet y'\wo{J}).
\]

We now claim that $\LdescentSet(x'\wo{I} \meet y'\wo{J}) \subseteq \LdescentSet(\zm' \wo{\Km})$. To see it, consider ${s \in \LdescentSet(x'\wo{I} \meet y'\wo{J})}$ and assume by contradiction that~$s \notin \LdescentSet(\zm' \wo{\Km})$. Then~$s$ does not belong to~$\LdescentSet(\zm')$, since the expression $\zm' \wo{\Km}$ is reduced. By \mbox{Deodhar's} Lemma (see Section~\ref{subsec:parabolicSubgroupsCosets}) we obtain that either $s \zm' \in W^\Km$ or  $s \zm' = \zm' t$ where 
\[
t \in \LdescentSet \big( \zm'^{-1}(x'\wo{I} \meet y'\wo{J}) \big) = \LdescentSet \big( \zm^{-1}(x\wo{I} \meet y\wo{J}) \big) = \Km.
\] 
In the first case we obtain
\[
1 + \length(\zm') + \length(\wo{\Km}) = \length(s \zm' \wo{\Km}) =   \length(\zm' \wo{\Km})-1=\length(\zm') + \length(\wo{\Km}) - 1
\]
a contradiction. In the second case, we get
\begin{align*}
1 + \length(\zm') + \length(\wo{\Km}) &= \length(s \zm' \wo{\Km}) = \length(\zm' t \wo{\Km}) \\
&= \length(\zm')+\length(t\wo{\Km}) = \length(\zm') + \length(\wo{\Km}) - 1,
\end{align*}
a contradiction again. This proves that~$\LdescentSet(x'\wo{I} \meet y'\wo{J}) \subseteq \LdescentSet(\zm' \wo{\Km})$.

To conclude the proof, we deduce from~$K \subseteq \LdescentSet(x'\wo{I} \meet y'\wo{J}) \subseteq \LdescentSet(\zm' \wo{\Km})$ that~${\wo{K} \le \zm' \wo{\Km}}$, and finally that~$z\wo{K} \le z\zm' \wo{\Km} = \zm \wo{\Km}$ since all expressions are reduced. Since~$z \le \zm$ and ${z\wo{K} \le \zm\wo{\Km}}$, we have~$zW_K \le \zm W_\Km$ so that~$\zm W_\Km$ is indeed the greatest lower bound.

\para{Existence of join}
The existence of the join follows from the existence of meet and the anti-automorphism~${\Psi: xW_I \longmapsto \woo x \wo{I} W_I}$ from Proposition~\ref{prop:antiautomorphisms}. Using the fact that~$\woo (\woo u \meet \woo v) = u \join v$, we get the  formula
\begin{align*}
xW_I \join yW_J
& = \Psi \big( \Psi(xW_I) \meet \Psi(yW_J) \big) \\
& = \Psi ( \woo x \wo{I} W_I \meet \woo y \wo{J} W_J ) \\
& = \Psi \big( (\woo x \wo{I} \meet \woo y \wo{J}) W_{\LdescentSet( (\woo x \wo{I} \meet \woo y \wo{J})^{-1} (\woo x \meet \woo y))} \big) \\
& = \Psi \big( (\woo x \wo{I} \meet \woo y \wo{J}) W_\Kj \big) \\
& = \woo (\woo x \wo{I} \meet \woo y \wo{J}) \wo{\Kj} W_\Kj \\
& = \zj \wo{\Kj} W_\Kj. \qedhere
\end{align*}
\end{proof}

We already observed in Corollary~\ref{cor:WeakRestriction} that the classical weak order is a subposet of the facial weak order. The formulas of Theorem~\ref{thm:lattice} ensure that it is also a sublattice.

\begin{corollary}
The classical weak order is a sublattice of the facial weak order.
\end{corollary}

\begin{proof}
If~$I = J = \varnothing$, then~$K = \varnothing$ in the formulas of Theorem~\ref{thm:lattice}.
\end{proof}

\begin{remark}
It is well-known that the map~$x \mapsto x\woo$ is an orthocomplementation of the weak order: it is involutive, order-reversing and satisfies $x\woo \meet x = e$ and $x\woo \join x = \woo$. In other words, it endows the weak order with a structure of ortholattice, see for instance \cite[Corollary~3.2.2]{BrentiBjorner}. This is no longer  the case for the facial weak order: the map~$xW_I \longmapsto \woo x \wo{I} W_I$ is indeed involutive and order-reversing, but is not an orthocomplementation: for a (counter-)example, consider~$x = e$ and~$I = S$.
\end{remark}


\subsection{Further properties of the facial weak order}
\label{subsec:propertiesFacialWeakOrder}

In this section, we study some properties of the facial weak order: we compute its partial M\"obius function, discuss formulas for the root inversion sets of meet and join, and describe its join-irreducible elements.


\subsubsection{M\"obius function}

Recall that the \defn{M\"obius function} of a poset~$P$ is the function~$\mu : P \times P \to \Z$ defined inductively by
\[
\mu(p,q) \eqdef
\begin{cases}
1 & \text{if } p = q, \\
\displaystyle{- \sum_{p \le r < q} \mu(p,r)} & \text{if } p < q, \\
0 & \text{otherwise.}
\end{cases}
\]
We refer the reader to~\cite{Stanley-enumerativeCombinatorics1} for more information on M\"obius functions. The following statement gives the values~$\mu(yW_J) \eqdef  \mu(eW_\emptyset, yW_J)$ of the M\"obius function on the facial weak order.

\begin{proposition}
The M\"obius function of the facial weak order is given by
\[
\mu(yW_J) \eqdef \mu(eW_\varnothing, yW_J) = 
\begin{cases}
(-1)^{|J|},\qquad & \text{if } y = e, \\
0, & \text{otherwise.}
\end{cases}
\]
\end{proposition}

\begin{proof}
We first show the equality when~$y = e$. From the characterization of Theorem~\ref{thm:facialWeakOrderCharacterizations}\,\eqref{item:facialWeakOrderCharacterizationCompareMinMax}, we know that~$xW_I \leq W_J$ if and only if~$x = e$ and~$I \subseteq J$. That is, the facial weak order below~$W_J$ is isomorphic to the boolean lattice on~$J$ (for example, the reader can observe a $3$-dimensional cube below~$W$ in \fref{fig:A3Words}). The result follows when~$y = e$ since the M\"obius function of the boolean lattice on~$J$ is given by~$\mu(I) = (-1)^{|I|}$ for~$I \subseteq J$ (inclusion-exclusion principle~\cite{Stanley-enumerativeCombinatorics1}).

We now prove by double induction on the length~$\ell(y)$ and the rank~$|J|$ that $\mu(yW_J) = 0$ for any coset~$yW_J$ with~$\length(y) \ge 1$ and~$J \subseteq S$. Indeed, consider~$yW_J$ with~$\length(y) \ge 1$ and assume that we have proved that~$\mu(xW_I) = 0$ for all~$xW_I$ with~$1 \le \length(x) < \length(y)$ or with~$\length(x) = \length(y)$ and~$|I| < |J|$. Since~$xW_I < yW_J$ if and only if~$x < y$, or~$x = y$ and~$I \subsetneq J$ we have
\[
\mu(yW_J) = - \sum_{xW_I < yW_J} \mu(xW_I) = - \sum_{W_I < yW_J} \mu(W_I).
\]
Therefore, since~$W_I < yW_J \iff \wo{I} \le y\wo{J} \iff I \subseteq \LdescentSet(y\wo{J})$, we have the non-empty boolean lattice on~$\LdescentSet(y\wo{J})$ and
\[
\mu(yW_J) = - \sum_{W_I < yW_J} \mu(W_I) = - \!\!\! \sum_{I \subseteq \LdescentSet(y\wo{J})} \!\!\! \mu(W_I) = - \!\!\! \sum_{I \subseteq \LdescentSet(y\wo{J})} \!\!\! (-1)^{|I|} = 0.
\qedhere
\]
\end{proof}


\subsubsection{Formulas for root inversion sets of meet and join}

For~$X \subseteq \Phi$, define the closure operators
\[
\plusTop{X} \eqdef \Phi^+ \cap \cone(X)
\qquad\text{and}\qquad
\minusTop{X} \eqdef \Phi^- \cap \cone(X),
\]
and their counterparts
\[
\plusBottom{X} \eqdef \Phi^+ \ssm \plusTop{\Phi^+ \ssm X}
\qquad\text{and}\qquad
\minusBottom{X} \eqdef \Phi^- \ssm \minusTop{\Phi^- \ssm X}.
\]
Note that~$\plusTop{X \cap \Phi^+} = \plusTop{X}$, $\plusBottom{X \cap \Phi^+} = \plusBottom{X}$, and~$-\plusTop{X} = \minusTop{{-}X}$. Similar formulas hold exchanging~$\ominus$'s with~$\oplus$'s.
Using these notations it is well known that the inversion sets of the meet and join in the (classical) weak order can be computed by
\begin{equation}
\label{eq:inversionSetMeetJoin}
\inversionSet(x \meet y) = \plusBottom{\inversionSet(x) \cap \inversionSet(y)}
\qquad\text{and}\qquad
\inversionSet(x \join y) = \plusTop{\inversionSet(x) \cup \inversionSet(y)}.
\tag{$\heartsuit$}
\end{equation}
For references on this property, see for example~\cite[Theorem~5.5]{BjornerEdelmanZiegler} and the discussion in~\cite{HohlwegLabbe} for its extension to infinite Coxeter groups.
Our next statement extends these formulas to compute the root inversion sets of the meet and join in the classical weak order.

\begin{corollary}
\label{coro:rootSetMeetJoin}
For~$x,y \in W$, the root inversion sets of the meet and join of~$x$ and~$y$ are given by
\begin{align*}
& \rootSet(x \meet y) = \minusTop{\rootSet(x) \cup \rootSet(y)} \cup \plusBottom{\rootSet(x) \cap \rootSet(y)}, \\
\text{and}\quad
& \rootSet(x \join y) = \minusBottom{\rootSet(x) \cap \rootSet(y)} \cup \plusTop{\rootSet(x) \cup \rootSet(y)}.
\end{align*}
\end{corollary}

\begin{proof}
This is immediate from Equation~\eqref{eq:inversionSetMeetJoin} and Proposition~\ref{prop:rootSetSingleElement}\,\eqref{item:rootSetSingleElementFromInversionSet}.
%
\end{proof}

We would now like to compute the root inversion sets of the meet~$xW_I \meet yW_J$ and join~$xW_I \join yW_J$ in the facial weak order in terms of the root inversion sets of~$xW_I$ and~$yW_J$. However, we only have a partial answer to this question.

\begin{proposition}
\label{prop:rootSetMeetJoinFacialWeakOrder}
For any cosets~$xW_I, yW_J \in \CoxeterComplex{W}$, we have
\begin{align*}
& \rootSet(xW_I \meet yW_J) \cap \Phi^- = \minusTop{\rootSet(xW_I) \cup \rootSet(yW_J)}, \\
\text{and}\quad
& \rootSet(xW_I \meet yW_J) \cap \Phi^+ \subseteq \plusBottom{\rootSet(xw_I) \cap \rootSet(yW_J)},
\end{align*}
while
\begin{align*}
& \rootSet(xW_I \join yW_J) \cap \Phi^- \subseteq \minusBottom{\rootSet(xW_I) \cap \rootSet(yW_J)}, \\
\text{and}\quad
& \rootSet(xW_I \join yW_J) \cap \Phi^+ =  \plusTop{\rootSet(xW_I) \cup \rootSet(yW_J)}.
\end{align*}
\end{proposition}

\begin{proof}
According to Theorem~\ref{thm:lattice}, we have $xW_I \meet yW_J = \zm W_{\Km}$ with~$\zm = x \meet y$ and~${\Km = \RdescentSet \big( \zm^{-1} (x\wo{I} \meet y\wo{J}) \big)}$. Then using Corollaries~\ref{coro:rootSetMeetJoin} and~\ref{coro:rootSetWeightSetFromBottomAndTop}, we have
\begin{align*}
\rootSet(xW_I \meet yW_J) \cap \Phi^- 
& = \rootSet(\zm W_\Km) \cap \Phi^-= \rootSet(\zm) \cap \Phi^- = \rootSet(x \meet y) \cap \Phi^- \\
& = \minusTop{\rootSet(x) \cup \rootSet(y)} = \minusTop{\rootSet(xW_I) \cup \rootSet(yW_J)}.
\end{align*}
Moreover, since~$\zm \wo{\Km} \le x\wo{I} \meet y\wo{J}$, we have
\begin{align*}
\rootSet(xW_I \meet yW_J) \cap \Phi^+
& = \rootSet(\zm W_\Km) \cap \Phi^+ = \rootSet(\zm\wo{\Km}) \cap \Phi^+ \subseteq \rootSet(x\wo{I} \meet y\wo{J}) \cap \Phi^+ \\
& = \plusBottom{\rootSet(x\wo{I}) \cap \rootSet(y\wo{J})} = \plusBottom{\rootSet(xW_I) \cap \rootSet(yW_J)}
\end{align*}
The proof is similar for the join, or can be obtained by the anti-automorphism of Proposition~\ref{prop:antiautomorphisms}.
\end{proof}

\begin{remark}
The inclusions~$\rootSet(xW_I \meet yW_J) \cap \Phi^+ \subseteq \plusBottom{\rootSet(xw_I) \cap \rootSet(yW_J)}$ and~$\rootSet(xW_I \join yW_J) \cap \Phi^- \subseteq \minusBottom{\rootSet(xW_I) \cap \rootSet(yW_J)}$ can be strict. For example, we have by Example~\ref{exm:computeMeet}
\[
\rootSet(tsrW_{st} \meet rtsW_\varnothing) \cap \Phi^+ = \rootSet(tW_r) \cap \Phi^+ = \{\alpha, \gamma\}
\]
which differs from
\[
\plusBottom{\rootSet(tsrW_{st}) \cap \rootSet(rtsW_\varnothing)} = \plusBottom{\{\alpha, \gamma, \alpha+\beta+\gamma, -\beta, -\alpha-\beta\}} = \{\alpha, \gamma, \alpha+\beta+\gamma\}.
\]
Following~\cite[Proposition~9]{KrobLatapyNovelliPhanSchwer}, the set~$\rootSet(xW_I \meet yW_J) \cap \Phi^+$ can be computed~as
\[
\rootSet(xW_I \meet yW_J) \cap \Phi^+ = \bigcap_{R} R \cap \Phi^+
\]
where the intersection runs over all~$R \subseteq \Phi$ satisfying the equivalent conditions of Corollary~\ref{coro:characterizationRootSets} such that
\[
R \cap \Phi^+ \subseteq \plusBottom{\rootSet(xw_I) \cap \rootSet(yW_J)}
\quad\text{while}\quad
R \cap \Phi^- \supseteq \minusTop{\rootSet(xW_I) \cup \rootSet(yW_J)}.
\]
However, this formula does not provide an efficient way to compute the root set of meets and joins in the facial weak order.
\end{remark}


\subsubsection{Join irreducible elements}
\label{subsubsec:joinIrreducible}

An element~$x$ of a finite lattice~$L$ is \defn{join-irreducible} if it cannot be written as~$x = \bigvee Y$ for some~$Y \subseteq L \ssm \{x\}$. Equivalently, $x$ is join-irreducible if it covers exactly one element~$x_\star$ of~$L$. For example, the join-irreducible elements of the classical weak order are the elements of~$W$ with a single descent. Meet-irreducible elements are defined similarly. We now characterize the join-irreducible elements of the facial weak order.

\begin{proposition}
\label{prop:joinIrreducibles}
A coset~$xW_I$ is join-irreducible in the facial weak order if and only if~$I = \varnothing$ and~$x$ is join-irreducible, or~$I = \{s\}$ and~$xs$ is join-irreducible.
\end{proposition}

\begin{proof}
Since~$xW_I$ covers~$xW_{I \ssm \{s\}}$ for any~$s \in S$, we have~$|I| \le 1$ for any join-irreducible coset~$xW_I$.

Suppose~$I = \varnothing$. The cosets covered by~$xW_\varnothing$ are precisely the cosets~$xsW_{\{s\}}$ with~$xs \precdot x$. Therefore, $xW_\varnothing$ is join-irreducible if and only if~$x$ is join-irreducible. Moreover, $(xW_\varnothing)_\star = \{x_\star,x\} = xsW_{s}$. 

Suppose~$I = \{s\}$. The cosets covered by~$xW_{\{s\}}$ are precisely~$xW_\varnothing$ and the cosets~$xs\wo{\{s,t\}}W_{\{s,t\}}$ for~$xst < xs$. Therefore, $xW_{\{s\}}$ is join-irreducible if and only if~$xs$ only covers~$x$, \ie if~$xs$ is join-irreducible. Moreover, ${(xW_{\{s\}})_\star = \{x\}}$.
\end{proof}

Using the anti-automorphism of Proposition~\ref{prop:antiautomorphisms}, we get the following statement.

\begin{corollary}
A coset~$xW_I$ is meet-irreducible in the facial weak order if and only if~$I = \varnothing$ and~$x$ is meet-irreducible, or~$I = \{s\}$ and~$x$ is meet-irreducible.
\end{corollary}


\subsection{Facial weak order on the Davis complex for infinite Coxeter groups}
\label{subsec:DavisComplex}

A natural question is to wonder if a facial weak order exists for the Coxeter complex of an infinite Coxeter system $(W,S)$. As the definition given by P.~Palacios and M.~Ronco makes extensive use of the longest element of each $W_I$ and of the longest element in the coset $W^I$, we do not know of any way to extend our results to the Coxeter complex.

However, partial results could be obtained with the \defn{Davis complex}, see for instance~\cite{Davis} or~\cite{AbramenkoBrown}, which is the abstract simplicial complex
\[
\DavisComplex{W} = \bigcup_{\substack{I \subseteq S \\ W_I \text{ finite}}} W/W_I.
\] 
\begin{definition} Call  \defn{facial weak order} $(\DavisComplex{W},\leq)$ on the Davis complex  the order defined by ${xW_I \leq yW_J}$ if and only if $x\leq y$ and $x\wo I\leq y\wo J$ in right weak order.
\end{definition}

All the results used to prove the existence and formula for the meet in Theorem~\ref{thm:lattice} in the case of a finite Coxeter system  only  use  the above definition of the facial weak order, as well as standard results valid for any Coxeter system; the finiteness of $W_\Km$ being guaranteed by the fact that a standard parabolic subgroup $W_K$ is finite if and only if there is $w \in W$ such that $\LdescentSet(w) = K$, see for instance~\cite[Proposition~2.3.1]{BrentiBjorner}. We therefore obtain the following result that generalizes A.~Bj\"orner's result for the weak order~\cite{Bjorner} to the facial weak order on the Davis complex. 

\begin{theorem}
\label{thm:latticeInfinite}
The facial weak order on the Davis complex is a meet-semilattice. The meet of two cosets~$xW_I$ and~$yW_J$ in~$\DavisComplex{W}$ is
\[
xW_I \meet yW_J = \zm W_\Km  \quad\text{where}\quad \zm = x \meet y
\quad\text{and}\quad
\Km = \LdescentSet \big( \zm^{-1} (x\wo{I} \meet y\wo{J}) \big).
\]
\end{theorem}


\section{Lattice congruences and quotients of the facial weak order}
\label{sec:latticeQuotients}

A \defn{lattice congruence} on a lattice~$(L,\le,\meet,\join)$ is an equivalence class~$\equiv$ which respects meets and joins, meaning that~$x \equiv x'$ and~$y \equiv y'$ implies $x \meet y \equiv x' \meet y'$ and~$x \join y \equiv x' \join y'$. In this section we start from any lattice congruence~$\equiv$ of the weak order and naturally define an equivalence relation~$\Equiv$ on the Coxeter complex~$\CoxeterComplex{W}$ by~$xW_I \Equiv yW_J \iff x \equiv y$ and~$x\wo{I} \equiv y\wo{J}$. The goal of this section is to show that~$\Equiv$ always defines a lattice congruence of the facial weak order. This will require some technical results on the weak order congruence~$\equiv$ (see Section~\ref{subsec:congruencesWeakOrder}) and on the projection maps of the congruence~$\Equiv$ (see Theorem~\ref{theo:latticeCongruenceFacialWeakOrder}).

On the geometric side, the congruence~$\Equiv$ of the facial weak order provides a complete description (see Theorem~\ref{theo:allFacesCongruenceFan}) of the simplicial fan~$\Fan_\equiv$ associated to the weak order congruence~$\equiv$ in N.~Reading's work~\cite{Reading-HopfAlgebras}: while the classes of~$\equiv$ correspond to maximal cones in~$\Fan_\equiv$, the classes of~$\Equiv$ correspond to all cones in~$\Fan_\equiv$ (maximal or not). We illustrate this construction in Section~\ref{subsec:facialCambrianLatticesFacialCube} with the facial boolean lattice (faces of a cube) and with the facial Cambrian lattices (faces of generalized associahedra) arising from the Cambrian lattices and fans of~\cite{Reading-CambrianLattices, ReadingSpeyer-CambrianFans}.


\subsection{Lattice congruences and projection maps}
\label{subsec:congruencesProjections}

We first recall the definition of lattice congruences and quotients and refer to~\cite{Reading-LatticeCongruences, Reading-CambrianLattices} for further details.

\begin{definition}
\label{def:latticeCongruences}
An \defn{order congruence} is an equivalence relation~$\equiv$ on a poset~$P$ such that:
\begin{enumerate}[(i)]
\item Every equivalence class under~$\equiv$ is an interval of~$P$.
\item The projection~$\projUp : P \to P$ (resp.~$\projDown : P \to P$), which maps an element of~$P$ to the maximal (resp.~minimal) element of its equivalence class, is order preserving.
\end{enumerate}

The \defn{quotient}~${P/\!\equiv}$ is a poset on the equivalence classes of~$\equiv$, where the order relation is defined by~$X \le Y$ in~${P/\!\equiv}$ if and only if there exist representatives~$x \in X$ and~$y \in Y$ such that~$x \le y$ in~$P$. The quotient~${P/\!\equiv}$ is isomorphic to the subposet of~$P$ induced by~$\projDown(P)$ (or equivalently by~$\projUp(P)$).

If, moreover,~$P$ is a finite lattice, then~$\equiv$ is a lattice congruence, meaning that it is compatible with meets and joins: for any~$x \equiv x'$ and~$y \equiv y'$, we have~$x \meet y \equiv x' \meet y'$ and~$x \join y \equiv x' \join y'$. The poset quotient~${P/\!\equiv}$ then inherits a lattice structure where the meet~$X \meet Y$ (resp.~the join~$X \join Y$) of two congruence classes~$X$ and~$Y$ is the congruence class of~$x \meet y$ (resp.~of~$x \join y$) for arbitrary representatives~$x \in X$ and~$y \in Y$.
\end{definition}

In our constructions we will use the projection maps~$\projUp$ and~$\projDown$ to define congruences. By definition note that~$\projDown(x) \le x \le \projUp(x)$, that~$\projUp \circ \projUp = \projUp \circ \projDown = \projUp$ while $\projDown \circ \projDown = \projDown \circ \projUp = \projDown$, and that $\projUp$ and~$\projDown$ are order preserving. The following lemma shows the reciprocal statement.

\begin{lemma}
\label{lem:conditionsProjectionMaps}
If two maps~$\projUp : P \to P$ and~$\projDown : P \to P$ satisfy
\begin{enumerate}[(i)]
\item
\label{item:conditionsProjectionMapsSandwich}
$\projDown(x) \le x \le \projUp(x)$ for any element~$x \in P$,

\item
\label{item:conditionsProjectionMapsLastWins}
$\projUp \circ \projUp = \projUp \circ \projDown = \projUp$ and $\projDown \circ \projDown = \projDown \circ \projUp = \projDown$,

\item
\label{item:conditionsProjectionMapsOrderPreserving}
$\projUp$ and~$\projDown$ are order preserving,
\end{enumerate}
then the fibers of~$\projUp$ and~$\projDown$ coincide and the relation~$\equiv$ on~$P$ defined by
\[
x \equiv y \iff \projUp(x) = \projUp(y) \iff \projDown(x) = \projDown(y)
\]
is an order congruence on~$P$ with projection maps~$\projUp$ and~$\projDown$.
\end{lemma}

\begin{proof}
First, Condition~\eqref{item:conditionsProjectionMapsLastWins} ensures that~${\projUp(x) = \projUp(y) \iff \projDown(x) = \projDown(y)}$ for any $x,y \in P$, so that the fibers of the maps~$\projUp$ and~$\projDown$ coincide. We now claim that if~$z \in \projDown(P)$, then the fiber~$\projDown^{-1}(z)$ is the interval~$[z, \projUp(z)]$. Indeed, if~${\projDown(x) = z}$, then~${\projUp(x) = \projUp(\projDown(x)) = \projUp(z)}$ by Condition~\eqref{item:conditionsProjectionMapsLastWins}, so that~$z \le x \le \projUp(z)$ by Condition~\eqref{item:conditionsProjectionMapsSandwich}. Reciprocally, for any~$z \le x \le \projUp(z)$, Conditions~\eqref{item:conditionsProjectionMapsLastWins} and~\eqref{item:conditionsProjectionMapsOrderPreserving} ensure that~${z = \projDown(z) \le \projDown(x) \le \projDown(\projUp(z)) = \projDown(z) = z}$, so that~${\projDown(x) = z}$. We conclude that the fibers of~$\projUp$ (or equivalently of~$\projDown$) are intervals of~$P$, and that~$\projUp$ (resp.~$\projDown$) indeed maps an element of~$P$ to the maximal (resp.~minimal) element of its fiber. Since~$\projUp$ and~$\projDown$ are order preserving, this shows that the fibers indeed define an order congruence.
\end{proof}


\subsection{Congruences of the weak order}
\label{subsec:congruencesWeakOrder}

Consider a lattice congruence~$\equiv$ of the weak order whose up and down projections are denoted by~$\projUp$ and~$\projDown$ respectively. We will need the following elementary properties of~$\,\equiv$. In this section, the notation $xW_I$ means that we are considering  $x$ in $W^I$.

\begin{lemma}
\label{lem:congruenceConjugate}
For any coset~$xW_I$ and any~$s \in I$, we have ${x \equiv xs \iff xs\wo{I} \equiv x\wo{I}}$.
\end{lemma}

\begin{proof}
Assume~$x \equiv xs$. As~$x \in W^I$ and~$s \in I$, we have~${xs \not\le xs\wo{I}}$. Therefore, $xs\wo{I} = x \join xs\wo{I} \equiv xs \join xs\wo{I} = x\wo{I}$. The reverse implication can be proved similarly or applying the anti-automorphism~$x \to x\woo$.
\end{proof}

\enlargethispage{.3cm}
We will need a refined version of the previous lemma for cosets of a rank~$2$ parabolic subgroup. Consider a coset~$xW_{\{s,t\}}$ with~$s,t \in S \ssm \RdescentSet(x)$. It consist of two chains
\[
x \le xs \le \dots \le xt\wo{\{s,t\}} \le x\wo{\{s,t\}} 
\quad \text{and} \quad
x \le xt \le \dots \le xs\wo{\{s,t\}} \le x\wo{\{s,t\}}
\]
from~$x$ to~$x\wo{\{s,t\}}$.
The following two lemmas are of same nature: they state that a single congruence between two elements of~$xW_{\{s,t\}}$ can force almost all elements in~$xW_{\{s,t\}}$ to be congruent. These lemmas are illustrated in \fref{fig:PolygonCongruence}.
\begin{figure}[b]
	\centerline{
	\begin{tikzpicture}
		[scale=1.8,
		vertex/.style={inner sep=0.5pt,circle,draw=black,fill=black,thick}
		]
		\begin{scope}[shift={(0,0)}]
		%
		\coordinate (e) at (0,0.7);
		\coordinate (s) at (-0.5,1);
		\coordinate (t) at (0.5,1);
		\coordinate (st) at (-0.5,1.5);
		\coordinate (ts) at (0.5,1.5);
		\coordinate (sts) at (0,1.8);
		\draw[red, thick, double distance=5pt, cap=round] (e) -- (s);
		%
		\draw (e) -- (s);
		\draw (e) -- (t);
		\draw[dashed] (s) -- (st);
		\draw[dashed] (t) -- (ts);
		\draw (st) -- (sts);
		\draw (ts) -- (sts);
		%
		\node[vertex] at (e) {};
		\node[vertex] at (s) {};
		\node[vertex] at (t) {};
		\node[vertex] at (st) {};
		\node[vertex] at (ts) {};
		\node[vertex] at (sts) {};
		%
		\node[below] at (e) {$x$};
		\node[left] at (s) {$xs$};
		\node[right] at (t) {$xt$};
		\node[left] at (st) {$xt\wo{\{s,t\}}$};
		\node[right] at (ts) {$xs\wo{\{s,t\}}$};
		\node[above] at (sts) {$x\wo{\{s,t\}}$};
		\end{scope}
		\node at (1.85,1.15) {$\overrightarrow{\color{red} \hspace{1cm} \; x \equiv xs \phantom{t} \hspace{1cm}}$};
		\begin{scope}[shift={(3.8,0)}]
		%
		\coordinate (e) at (0,0.7);
		\coordinate (s) at (-0.5,1);
		\coordinate (t) at (0.5,1);
		\coordinate (st) at (-0.5,1.5);
		\coordinate (ts) at (0.5,1.5);
		\coordinate (sts) at (0,1.8);
		\draw[red, thick, double distance=5pt, cap=round, rounded corners=.3mm] (e) -- (s) -- (st);
		\draw[red, thick, double distance=5pt, cap=round, rounded corners=.3mm] (t) -- (ts) -- (sts);
		%
		\draw (e) -- (s);
		\draw (e) -- (t);
		\draw[dashed] (s) -- (st);
		\draw[dashed] (t) -- (ts);
		\draw (st) -- (sts);
		\draw (ts) -- (sts);
		%
		\node[vertex] at (e) {};
		\node[vertex] at (s) {};
		\node[vertex] at (t) {};
		\node[vertex] at (st) {};
		\node[vertex] at (ts) {};
		\node[vertex] at (sts) {};
		%
		\node[below] at (e) {$x$};
		\node[left] at (s) {$xs$};
		\node[right] at (t) {$xt$};
		\node[left] at (st) {$xt\wo{\{s,t\}}$};
		\node[right] at (ts) {$xs\wo{\{s,t\}}$};
		\node[above] at (sts) {$x\wo{\{s,t\}}$};
		\end{scope}
		\begin{scope}[shift={(0,-2)}]
		%
		\coordinate (e) at (0,0.7);
		\coordinate (s) at (-0.5,1);
		\coordinate (t) at (0.5,1);
		\coordinate (st) at (-0.5,1.5);
		\coordinate (ts) at (0.5,1.5);
		\coordinate (sts) at (0,1.8);
		\draw[blue, thick, double distance=5pt, cap=round] (s) -- (t);
		%
		\draw (e) -- (s);
		\draw (e) -- (t);
		\draw[dashed] (s) -- (st);
		\draw[dashed] (t) -- (ts);
		\draw (st) -- (sts);
		\draw (ts) -- (sts);
		%
		\node[vertex] at (e) {};
		\node[vertex] at (s) {};
		\node[vertex] at (t) {};
		\node[vertex] at (st) {};
		\node[vertex] at (ts) {};
		\node[vertex] at (sts) {};
		%
		\node[below] at (e) {$x$};
		\node[left] at (s) {$xs$};
		\node[right] at (t) {$xt$};
		\node[left] at (st) {$xt\wo{\{s,t\}}$};
		\node[right] at (ts) {$xs\wo{\{s,t\}}$};
		\node[above] at (sts) {$x\wo{\{s,t\}}$};
		\end{scope}
		\node at (1.85,-0.85) {$\overrightarrow{\color{blue} \hspace{1cm} \; xs \equiv xt \; \hspace{1cm}}$};
		\begin{scope}[shift={(3.8,-2)}]
		%
		\coordinate (e) at (0,0.7);
		\coordinate (s) at (-0.5,1);
		\coordinate (t) at (0.5,1);
		\coordinate (st) at (-0.5,1.5);
		\coordinate (ts) at (0.5,1.5);
		\coordinate (sts) at (0,1.8);
		\draw[blue, thick, double distance=5pt, cap=round, rounded corners=.3mm] (e) -- (s) -- (st) -- (sts) -- (ts) -- (t) -- (e);
		%
		\draw (e) -- (s);
		\draw (e) -- (t);
		\draw[dashed] (s) -- (st);
		\draw[dashed] (t) -- (ts);
		\draw (st) -- (sts);
		\draw (ts) -- (sts);
		%
		\node[vertex] at (e) {};
		\node[vertex] at (s) {};
		\node[vertex] at (t) {};
		\node[vertex] at (st) {};
		\node[vertex] at (ts) {};
		\node[vertex] at (sts) {};
		%
		\node[below] at (e) {$x$};
		\node[left] at (s) {$xs$};
		\node[right] at (t) {$xt$};
		\node[left] at (st) {$xt\wo{\{s,t\}}$};
		\node[right] at (ts) {$xs\wo{\{s,t\}}$};
		\node[above] at (sts) {$x\wo{\{s,t\}}$};
		\end{scope}
	\end{tikzpicture}
}
	\caption{In a coset~$xW_{\{s,t\}}$, a single congruence may force many congruences. See Lemma~\ref{lem:forceCongruences1} for the congruence~${x \equiv xs}$ (top) and Lemma~\ref{lem:forceCongruences2} for the congruence~${xs \equiv xt}$ (bottom).}
	\label{fig:PolygonCongruence}
\end{figure}
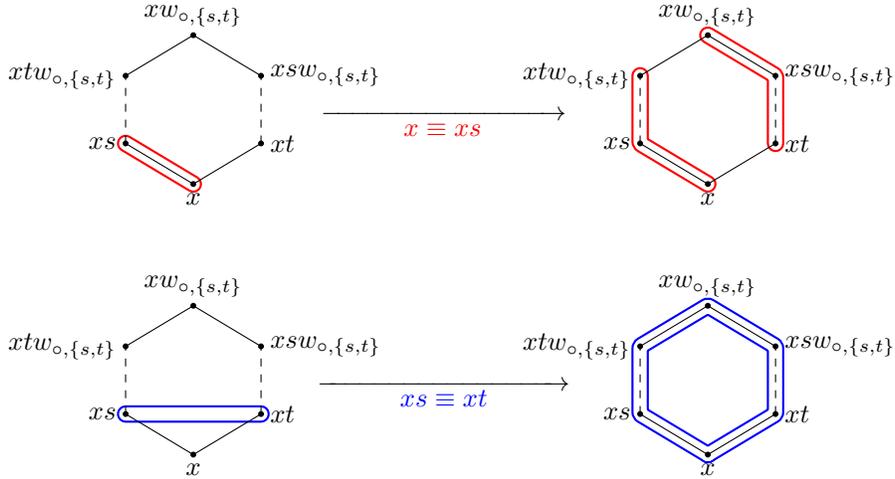

\begin{lemma}
\label{lem:forceCongruences1}
For any coset~$xW_{\{s,t\}}$, if~$x \equiv xs$ or $xs\wo{\{s,t\}} \equiv x\wo{\{s,t\}}$ then
\[
x \equiv xs \equiv xst \equiv \dots \equiv xt\wo{\{s,t\}}
\quad\text{and}\quad
xt \equiv xts \equiv \dots \equiv xs\wo{\{s,t\}} \equiv x\wo{\{s,t\}}.
\]
\end{lemma}

\begin{proof}
Assume~$x \equiv xs$. For any~$xt \le y \le x\wo{\{s,t\}}$ we have~$xs \join y = x\wo{\{s,t\}}$. Since~$x \equiv xs$ and~$\equiv$ is a lattice congruence, we get~${y = x \join y \equiv xs \join y = x\wo{\{s,t\}}}$. Now for any~$x \le z \le xt\wo{\{s,t\}}$, we have~$y \meet z = x$. Since~$y \equiv x\wo{\{s,t\}}$ and~$\equiv$ is a lattice congruence, we get~$z = x\wo{\{s,t\}} \meet z \equiv y \meet z = x$. The proof is similar if we assume instead~$xs\wo{\{s,t\}} \equiv x\wo{\{s,t\}}$.
\end{proof}

\begin{lemma}
\label{lem:forceCongruences2}
For any coset~$xW_{\{s,t\}}$, if~$xs \equiv xt$ then~${x \equiv y}$ for all~${y \in xW_{\{s,t\}}}$.
\end{lemma}

\begin{proof}
Since~$xs \equiv xt$, we have~$\projDown(xs) = \projDown(xt)$ and~$\projUp(xs) = \projUp(xt)$. Using that $\projDown(xs) \le xs \le \projUp(xs)$ and~$\projDown(xt) \le xt \le \projUp(xs)$, we obtain that
\[
\projDown(xs) \le xs \meet xt = x \le x \wo{\{s,t\}} = xs \join xt \le \projUp(xs).
\]
Since the congruence class of~$xs$ is the interval~$[\projDown(xs), \projUp(xs)]$, it certainly contains all the coset~$xW_{\{s,t\}}$. We conclude that~${x \equiv y}$ for all~${y \in xW_{\{s,t\}}}$.
\end{proof}

Throughout the end of this section, we write~$x \leqequiv y$ when~$x \le y$ and~$x \equiv y$. In other words, $x \leqequiv y \iff x \le y \le \projUp(x) \iff \projDown(y) \le x \le y$. Note that the relation~$\leqequiv$ is transitive (as the intersection of two transitive relations) and stable by meet and join (as~$\le$ is a lattice and~$\equiv$ a lattice congruence).

The goal of the next statements is to show that one can ``translate faces along congruence classes''. We make this statement precise in the next lemmas.

\begin{lemma}
\label{lem:transport}
For any~$x \in W$ and~$t \in S \ssm \RdescentSet(x)$ such that~$x \not\equiv xt$, there exists a unique~$\sUp(x,t) \in S \ssm \RdescentSet \big( \projUp(x) \big)$ such that~$xt \leqequiv \projUp(x)\sUp(x,t)$.
\end{lemma}

\begin{proof}
To prove the existence of~$\sUp(x,t)$, we work by induction on the length of a minimal path from~$x$ to~$\projUp(x)$ in weak order. If~$x = \projUp(x)$, then~${\sUp(x,t) = t}$ meets our criteria. We now assume that there exists ${s \in S \ssm \RdescentSet(x)}$ such that~$x \leqequiv xs \leqequiv \projUp(x)$. Let~${x' = xt\wo{\{s,t\}}}$ and ${t' = \wo{\{s,t\}} t \wo{\{s,t\}}}$. We get from Lemma~\ref{lem:forceCongruences1} that~$x \equiv x'$, thus~$\projUp(x)=\projUp(x')$ and~${xt \leqequiv x\wo{\{s,t\}} = x't'}$. Since~$x \not\equiv xt$, this also ensures that~${x' \not\equiv x't'}$. Thus the length of a minimal path from between~$x't'$ and~$\projUp(x')$ is strictly smaller than the length of a minimal path between~$x$~and~$\projUp(x)$. Therefore, by induction hypothesis, there exists~$\sUp(x',t') \in S$ such that~${x't' \leqequiv \projUp(x')\sUp(x',t')}$. We therefore obtain that $${xt \leqequiv x't' \leqequiv \projUp(x') \sUp(x',t') = \projUp(x) \sUp(x',t'),}$$ and conclude that~$\sUp(x,t) = \sUp(x',t')$ meets our criteria.

To prove uniqueness, assume that there exist~$r \ne s \in S \ssm \RdescentSet(y)$ which both satisfy~$xt \leqequiv \projUp(x)r$ and~$xt \leqequiv \projUp(x)s$. This implies that~$\projUp(x)r \equiv xt \equiv \projUp(x)s$, so that~$\projUp(x) \equiv \projUp(x)r \equiv \projUp(x)s$ by application of Lemma~\ref{lem:forceCongruences2}. We would therefore obtain that~$x \equiv \projUp(x) \equiv \projUp(x)r \equiv xt$, a contradiction.
\end{proof}

\begin{lemma}
\label{lem:defSUp}
For any coset~$xW_I$, the set~$\SUp(x,I) \eqdef \set{\sUp(x,t)}{t \in I, \; x \not\equiv xt}$ is the unique subset of~$S \ssm \RdescentSet \big( \projUp(x) \big)$ such that ${x\wo{I} \leqequiv \projUp(x)\wo{\SUp(x,I)}}$.
\end{lemma}

\begin{proof}
Split~$I$ into~${I_\equiv \sqcup I_{\not\equiv}}$ where~$I_\equiv \eqdef \set{t \in I}{x \equiv xt}$ and~$I_{\not\equiv} \eqdef \set{t \in I}{x \not\equiv xt}$. Since~$\leqequiv$ is stable by join, we get
\[
x\wo{I} = \bigg( \bigvee_{t \in I_\equiv} xt \bigg) \join \bigg( \bigvee_{t \in I_{\not\equiv}} xt \bigg) \leqequiv \projUp(x) \join \bigg( \bigvee_{t \in I_{\not\equiv}} \projUp(x)\sUp(x,t) \bigg) = \projUp(x)\wo{\SUp(x,I)}.
\]

To prove unicity, we observe that there already is a unique maximal subset~$\Sigma$ of~${S \ssm \RdescentSet \big( \projUp(x) \big)}$ such that~$x\wo{I} \leqequiv \projUp(x)\wo{\Sigma}$ since $\leqequiv$ is stable by join. Consider now any subset~$\Sigma'$ of~$\Sigma$ with this property. Since~$\projUp(x)\wo{\Sigma'} \equiv x\wo{I} \equiv \projUp(x)\wo{\Sigma}$, we obtain
\[
\projUp(x) \wo{\Sigma \ssm \Sigma'} = \projUp(x) \wo{\Sigma \ssm \Sigma'} \meet \projUp(x) \wo{\Sigma} \equiv \projUp(x) \wo{\Sigma \ssm \Sigma'} \meet \projUp(x) \wo{\Sigma'} = \projUp(x).
\]
Since~$\projUp(x)$ is maximal in its congruence class, this implies that~$\wo{\Sigma \ssm \Sigma'} = e$ so that~$\Sigma' = \Sigma$.
\end{proof}

\enlargethispage{-.6cm}
We will furthermore need the following properties of~$\sUp(x,t)$ and~$\SUp(x,I)$.

\begin{lemma}
\label{lem:propertiessUpSUp}
For any coset~$xW_I$ and any~$t \in I$, we have
\begin{itemize}
\item if~$x \equiv xt$ then~$xt\wo{I} \equiv \projUp(x)\wo{\SUp(x,I)}$,
\item if~$x \not\equiv xt$ then~$xt\wo{I} \equiv \projUp(x)\sUp(x,t)\wo{\SUp(x,I)}$.
\end{itemize}
In other words, either $xt\wo{I} \equiv x\wo{I} \equiv \projUp(x)\wo{\SUp(x,I)}$, or $xt\wo{I} \equiv \projUp(x)\sUp(x,t)\wo{\SUp(x,I)}$.
\end{lemma}

\begin{proof}
If~$x \equiv xt$, then Lemma~\ref{lem:congruenceConjugate} ensures that $xt\wo{I} \equiv x\wo{I} \equiv \projUp(x)\wo{\SUp(x,I)}$. Assume now that~$x \not\equiv xt$. We have~$\projUp(x)\sUp(x,t)\wo{\SUp(x,I)} \not\equiv \projUp(x)\wo{\SUp(x,I)}$. Indeed, Lemma~\ref{lem:congruenceConjugate} would imply that~$\projUp(x) \equiv \projUp(x)\sUp(x,t)$ contradicting the maximality of~$\projUp(x)$ in its congruence class. Using the same argument as in Lemma~\ref{lem:transport}, we obtain that there is~$t'$ such that~$xt'\wo{I} \equiv \projUp(x)\sUp(x,t)\wo{\SUp(x,I)}$. Observe that
\[
xt'\wo{I} \equiv \projUp(x)\sUp(x,t)\wo{\SUp(x,I)} = \projUp(x)\sUp(x,t)\wo{\SUp(x,I)} \join \projUp(x)\sUp(x,s) \equiv xt'\wo{I} \join xs
\]
for all~$s \in I \ssm \{t\}$ such that~$x \not\equiv xs$. Since~$xt'\wo{I} \not\equiv x\wo{I}$, we obtain that~${t' = t}$.
\end{proof}

Using similar arguments as in the previous lemmas, or applying the anti-automor\-phism~$x \mapsto x\woo$, we deduce the following statement, similar to Lemma~\ref{lem:defSUp}.

\begin{lemma}
\label{lem:defSDown}
For any coset~$xW_I$, there is a unique subset~$\SDown(x,I)$ of~$\RdescentSet \big( \projDown(x\wo{I}) \big)$ such that~$\projDown(x\wo{I}) \wo{\SDown(x,I)} \leqequiv x$.
\end{lemma}
\begin{remark}
Consider a coset~$xW_I$. If~$\projUp(x) = x$ then~$\sUp(x,t) = t$ for all~$t \in I$ and thus~$\SUp(x,I) = I$. Similarly, if~$\projDown(x\wo{I}) = x\wo{I}$ then~$\SDown(x,I) = I$.
\end{remark}


\subsection{Congruences of the facial weak order}
\label{subsec:congruencesFacialWeakOrder}

Based on the properties established in the previous section we now show that the lattice congruences of the weak order naturally extend to lattice congruences of the facial weak order. We start from a lattice congruence~$\equiv$ of the weak order whose up and down projections are denoted by~$\projUp$ and~$\projDown$ respectively. We then define two maps~$\ProjUp : \CoxeterComplex{W} \to \CoxeterComplex{W}$ and~${\ProjDown : \CoxeterComplex{W} \to \CoxeterComplex{W}}$~by
\[
\ProjUp(xW_I) = \projUp(x) W_{\SUp(x,I)}
\qquad\text{and}\qquad
\ProjDown(xW_I) = \projDown(x\wo{I}) W_{\SDown(x,I)}
\]
where~$\SUp(x,I)$ and~$\SDown(x,I)$ are the subsets of~$S$ defined by Lemmas~\ref{lem:defSUp} and~\ref{lem:defSDown}.
Note that we again take the liberty here to write~$\ProjDown(xW_I) = \projDown(x\wo{I}) W_{\SDown(x,I)}$ instead of~$\ProjDown(xW_I) = \projDown(x\wo{I}) \,\wo{\SUp(x,I)} W_{\SDown(x,I)}$ to make apparent the symmetry between~$\ProjUp$ and~$\ProjDown$.

It immediately follows from Lemmas~\ref{lem:defSUp} and~\ref{lem:defSDown} that $\ProjUp(xW_I)$ is the biggest parabolic coset in the interval~$[\projUp(x), \projUp(x\wo{I})]$ containing~$\projUp(x)$ and similarly~$\ProjDown(xW_I)$ is the biggest parabolic coset in the interval~$[\projDown(x), \projDown(x\wo{I})]$ containing~$\projDown(x\wo{I})$.

\begin{theorem}
\label{theo:latticeCongruenceFacialWeakOrder}
The maps~$\ProjUp$ and~$\ProjDown$ fulfill the following properties:
\begin{enumerate}[(i)]
\item
\label{item:latticeCongruenceFacialWeakOrderSandwich}
$\ProjDown(xW_I) \le xW_I \le \ProjUp(xW_I)$ for any coset~$xW_I$.

\item
\label{item:latticeCongruenceFacialWeakOrderLastWins}
$\ProjUp \circ \ProjUp = \ProjUp \circ \ProjDown = \ProjUp$ and $\ProjDown \circ \ProjDown = \ProjDown \circ \ProjUp = \ProjDown$.

\item
\label{item:latticeCongruenceFacialWeakOrderOrderPreserving}
$\ProjUp$ and~$\ProjDown$ are order preserving.
\end{enumerate}
Therefore, the fibers of the maps~$\ProjUp$ and~$\ProjDown$ coincide and define a lattice congruence~$\Equiv$ of the facial weak order.
\end{theorem}

\begin{proof}
Using the characterization of the facial weak order given in Theorem~\ref{thm:facialWeakOrderCharacterizations}\,\eqref{item:facialWeakOrderCharacterizationCompareMinMax}, we obtain that~$xW_I \le \ProjUp(xW_I)$ since~$x \le \projUp(x)$ and~$x\wo{I} \le \projUp(x)\wo{\SUp(x,I)}$. Similarly, $\ProjDown(xW_I) \le xW_I$ since~$\projDown(x\wo{I}) \wo{\SDown(x,I)} \le x$ and~$\projDown(x\wo{I}) \le x\wo{I}$. This shows~\eqref{item:latticeCongruenceFacialWeakOrderSandwich}.

\medskip
For~\eqref{item:latticeCongruenceFacialWeakOrderLastWins}, it follows from the definition that~$\ProjUp \big( \ProjUp(xW_I) \big) = \ProjUp \big( \projUp(x)W_{\SUp(x,I)} \big)$ is the biggest parabolic coset in the interval~$\big[ \projUp \big( \projUp(x) \big), \projUp \big( \projUp(x)\wo{\SUp(x,I)} \big) \big]$ containing~$\projUp \big( \projUp(x) \big)$. However, we have~${\projUp \big( \projUp(x) \big) = \projUp(x)}$ and~${\projUp \big( \projUp(x)\wo{\SUp(x,I)} \big) = \projUp(x\wo{I})}$ since $x\wo{I} \equiv \projUp(x)\wo{\SUp(x,I)}$. We conclude that~${\ProjUp \circ \ProjUp = \ProjUp}$. The proof is similar for the other equalities of~\eqref{item:latticeCongruenceFacialWeakOrderLastWins}.

\medskip
To prove~\eqref{item:latticeCongruenceFacialWeakOrderOrderPreserving}, it is enough to show that~$\ProjUp$ is order-preserving on covering relations of the facial weak order (it is then order preserving on any weak order relation by transitivity, and the result for~$\ProjDown$ can be argued similarly or using the anti-automorphisms of Proposition~\ref{prop:antiautomorphisms}). Therefore, we consider a cover relation~$xW_I \precdot yW_J$ in facial weak order and prove that~$\ProjUp(xW_I) \le \ProjUp(yW_J)$.

It is immediate if the cover relation~$xW_I \precdot yW_J$ is of type~(1), that is, if~$x = y$ and~$J = I \cup \{s\}$. Indeed, it follows from the characterization in terms of biggest parabolic subgroups and from the fact that~$\projUp(x) = \projUp(y)$ and~$\projUp(x\wo{I}) \le \projUp(y\wo{J})$.

Consider now a cover relation~$xW_I \precdot yW_J$ of type~(2), that is, with~${y = x\wo{I}\wo{J}}$ and~$J = I \ssm \{s\}$. Note that in this case $\projUp(x) \leq \projUp(y)$ and $\projUp(x\wo{I}) = \projUp(y\wo{J})$. We therefore need to show that ${\projUp(x)\wo{\SUp(x,I)} \le \projUp(y)\wo{\SUp(y,J)}}$.

For~$t \in S$, define~$t^\star \eqdef \wo{I}\wo{J}t\wo{J}\wo{I}$ so that the equality $x\wo{I} = y\wo{J}$ implies the equality ${yt\wo{J} = xt^\star\wo{I}}$.
Let
\begin{align*}
J_\equiv &\eqdef \set{t \in J}{yt\wo{J} \equiv y\wo{J}} = \set{t \in J}{xt^\star\wo{I} \equiv x\wo{I}}, \\
J_{\not\equiv} &\eqdef \set{t \in J}{yt\wo{J} \not\equiv y\wo{J}} = \set{t \in J}{xt^\star\wo{I} \not\equiv x\wo{I}},
\end{align*}
and consider
\[
K \eqdef \set{\wo{\SUp(x,I)}\sUp(x,t^\star)\wo{\SUp(x,I)}}{t \in J_{\not\equiv}}
\qquad\text{and}\qquad
z \eqdef \projUp(x)\wo{\SUp(x,I)}\wo{K}.
\] 
Lemma~\ref{lem:propertiessUpSUp} ensures that
\[
yt\wo{J} = xt^\star\wo{I} \equiv \begin{cases}
\projUp(x)\wo{\SUp(x,I)} &\text{if } t \in J_\equiv, \\
\projUp(x)\sUp(x,t^\star)\wo{\SUp(x,I)} & \text{if } t \in J_{\not\equiv}.
\end{cases}
\]
Therefore
\begin{align*}
y & = \bigwedge_{t \in J} yt\wo{J} = \bigwedge_{t \in J_\equiv} yt\wo{J} \wedge \bigwedge_{t \in J_{\not\equiv}} yt\wo{J} \\
& \equiv \projUp(x)\wo{\SUp(x,I)} \wedge  \bigwedge_{t \in J_{\not\equiv}} \projUp(x)\sUp(x,t^\star)\wo{\SUp(x,I)} \\
& = \projUp(x) \wo{\SUp(x,I)} \bigwedge_{t \in J_{\not\equiv}} \wo{\SUp(x,I)} \sUp(x,t^\star) \wo{\SUp(x,I)} \\
& = \projUp(x) \wo{\SUp(x,I)} \wo{K} = z.
\end{align*}
By Lemma~\ref{lem:defSUp} applied to the coset~$zW_K$, there exists~$\SUp(z,K)$ such that
\[
\projUp(x)\wo{\SUp(x,I)} = z\wo{K} \leqequiv \projUp(z)\wo{\SUp(z,K)} = \projUp(y)\wo{\SUp(z,K)}.
\]
Since~$\projUp(x)\wo{\SUp(x,I)} \equiv x\wo{I} = y\wo{J}$, it follows that~$\SUp(y,J) = \SUp(z,K)$ by unicity in Lemma~\ref{lem:defSUp} applied to the coset~$yW_J$. We  get that~${\projUp(x)\wo{\SUp(x,I)} \le \projUp(y)\wo{\SUp(y,J)}}$ and thus that~$\ProjUp(xW_I) \le \ProjUp(yW_J)$.

\medskip
We conclude by Lemma~\ref{lem:conditionsProjectionMaps} that the fibers of~$\ProjUp$ and~$\ProjDown$ indeed coincide and define a lattice congruence~$\Equiv$ of the facial weak order.
\end{proof}


\subsection{Properties of facial congruences}
\label{subsec:propertiesFacialCongruences}

In this section, we gather some properties of the facial congruence~$\Equiv$ defined in Theorem~\ref{theo:latticeCongruenceFacialWeakOrder}.


\subsubsection{Basic properties}

We first come back to the natural definition of~$\Equiv$ given in the introduction of Section~\ref{sec:latticeQuotients}.

\begin{proposition}
\label{prop:characterizationEquiv}
For any cosets~$xW_I, yW_J \in \CoxeterComplex{W}$,
\[
xW_I \Equiv yW_J \iff x \equiv y \text{ and } x\wo{I} \equiv y\wo{J}.
\]
\end{proposition}

\begin{proof}
If~$xW_I \Equiv yW_J$, then~$\ProjUp(xW_I) = \ProjUp(yW_J)$ so that~$\projUp(x) = \projUp(y)$ and~${x \equiv y}$. Moreover, $\ProjDown(xW_I) = \ProjDown(yW_J)$ so that~$\projDown(x\wo{I}) = \projDown(y\wo{J})$ and~$x\wo{I} \equiv y\wo{J}$. Therefore, the $\Equiv$-congruence class of~$xW_I$ determines the $\equiv$-congruence classes of~$x$ and of~$x\wo{I}$. Reciprocally, we already observed that~$\ProjUp(xW_I)$ is the biggest parabolic coset in the interval~$[\projUp(x), \projUp(x\wo{I})]$ containing~$\projUp(x)$. If~$x \equiv y$ and ${x\wo{I} \equiv y\wo{J}}$, we obtain that~$\ProjUp(xW_I) = \ProjUp(yW_J)$. Therefore, the $\Equiv$-congruence class of~$xW_I$ only depends on the $\equiv$-congruence classes of~$x$ and of~$x\wo{I}$.
\end{proof}


\begin{corollary}
\label{coro:restrictCongruence}
For any~$x,y \in W$, we have~$x \equiv y \iff xW_\varnothing \Equiv yW_\varnothing$. Therefore, each congruence class~$\gamma$ of~$\equiv$ is the intersection of~$W$ with a congruence class~$\Gamma$ of~$\Equiv$.
\end{corollary}


This corollary says that the congruence~$\Equiv$ of the facial weak order indeed extends the congruence~$\equiv$ of the weak order. Nevertheless, observe that not all congruences of the facial weak order arise as congruences of the weak order (consider for instance the congruence on~$\CoxeterComplex{A_2}$ that only contracts~$sW_t$ with~$stW_\varnothing$).

%


\subsubsection{Join-irreducible contractions}

Recall that an element~$x$ of a finite lattice~$L$ is \defn{join-irreducible} if it covers exactly one element~$x_\star$ (see Section~\ref{subsubsec:joinIrreducible}). The following statement can be found \eg in~\cite[Lemma~2.32]{FreeseJezekNation}. For a lattice congruence~$\equiv$ on~$L$ and~$y \in L$, let~$D_\equiv(y)$ denote the set of join-irreducible elements~$x \le y$ not contracted by~$\equiv$, that is such that~$x_\star \not\equiv x$. For~$y,z \in L$, we then have~${y \equiv z \iff D_\equiv(y) = D_\equiv(z)}$ and lattice quotient~$L/{\equiv}$ is isomorphic to the inclusion poset on~$\set{D_\equiv(y)}{y \in L}$. In other words, the lattice congruence~$\equiv$ is characterized by the join-irreducible elements of~$L$ that it contracts. Even if this characterization is not always convenient, it is relevant to describe the join-irreducibles of the facial weak order contracted by~$\Equiv$ in terms of those contracted~by~$\equiv$.

\begin{proposition}
The join-irreducible cosets of the facial weak order contracted by~$\Equiv$ are precisely:
\begin{itemize}
\item the cosets~$xW_\varnothing$ where~$x$ is a join-irreducible element of the weak order contracted by~$\equiv$,
\item the cosets~$xW_{\{s\}}$ where~$xs$ is a join-irreducible element of the weak order contracted by~$\equiv$.
\end{itemize}
\end{proposition}

\begin{proof}
The join-irreducible cosets of the facial weak order are described in Proposition~\ref{prop:joinIrreducibles}. Now~$xW_\varnothing$ is contracted by~$\Equiv$ when $xW_\varnothing \Equiv (xW_\varnothing)_\star = \{x_\star,x\}$, that is, when~$x \equiv x_\star$ by Proposition~\ref{prop:characterizationEquiv}. Similarly, $xW_{\{s\}}$ is contracted by~$\Equiv$ when~$xW_{\{s\}} \Equiv (xW_{\{s\}})_\star = \{x\}$, that is, when~$xs \equiv x$ by Proposition~\ref{prop:characterizationEquiv}.
\end{proof}


\subsubsection{Up and bottom cosets of facial congruence classes}

The next statements deal with maximal and minimal cosets in their facial congruence classes.

\begin{proposition}
\label{prop:sortable}
For any coset~$xW_I$, we have
\begin{enumerate}[(i)]
\item
\label{item:antisortable}
$\ProjUp(xW_I) = xW_I \iff \projUp(x) = x$,

\item
\label{item:sortable}
$\ProjDown(xW_I) = xW_I \iff \projDown(x\wo{I}) = x\wo{I}$.
\end{enumerate}
\end{proposition}

\begin{proof}
We only prove~\eqref{item:antisortable}, the proof of~\eqref{item:sortable} being symmetric. Recall the definition ${\ProjUp(xW_I) = \projUp(x)W_{\SUp(x,I)}}$. Therefore, $\ProjUp(xW_I) = xW_I$ clearly implies that~${\projUp(x) = x}$. Reciprocally, if~$\projUp(x) = x$, then~$\SUp(x,I) = I$ by the uniqueness of $\SUp(x,I)$ in Lemma~\ref{lem:defSUp}. Therefore~$\ProjUp(xW_I) = \projUp(x)W_{\SUp(x,I)} = xW_I$.
\end{proof}

Call an element $x$ in $W$ a \defn{$\equiv$-singleton} if it is alone in its $\equiv$-congruence class, \ie such that~$\projDown(x) = x = \projUp(x)$. Similarly, call a coset~$xW_I$ a \defn{facial $\Equiv$-singleton} if it is alone in its $\Equiv$-congruence class, \ie such that~$\ProjDown(xW_I) = xW_I = \ProjUp(xW_I)$.

\begin{proposition}
\label{prop:singletons}
\begin{enumerate}[(i)]
\item
\label{item:characterizationSingletons}
A coset~$xW_I$ is a facial $\Equiv$-singleton if and only if~${\projUp(x) = x}$ and~${\projDown(x\wo{I}) = x\wo{I}}$.

\item
\label{item:fromSingletonstoFacialSingletons}
If $x$ is a $\equiv$-singleton, then $xW_I$ is a facial $\Equiv$-singleton for any~${I \subset S \ssm \RdescentSet(x)}$. Moreover, $x\wo{J}W_J$ is a facial $\Equiv$-singleton for any~$J \subseteq \RdescentSet(x)$.
\end{enumerate}
\end{proposition}

\begin{proof}
\eqref{item:characterizationSingletons} is an immediate consequence of Proposition~\ref{prop:sortable}. To prove~\eqref{item:fromSingletonstoFacialSingletons}, we just need to show that if~$x$ is a $\equiv$-singleton, then~$\projDown(x\wo{I}) = x\wo{I}$ for any~${I \subseteq S \ssm \RdescentSet(x)}$. If not, there would exist~$t \in S$ such that~$xt\wo{I} \leqequiv x\wo{I}$. If~$t \in I$, then~$x \equiv xt$ by Lemma~\ref{lem:congruenceConjugate}. If~$t \notin I$, then~${xt \le x}$. Since~$\equiv$ is a lattice congruence and~${xt\wo{I} \equiv x\wo{I}}$,
\[
x = x \meet x\wo{I} \equiv x \meet xt\wo{I} = xt.
\]
In both cases, we contradict the assumption that~$x$ is a $\equiv$-singleton. We prove similarly that if~$x$ is a $\equiv$-singleton, then~$\projUp(x\wo{J}) = x\wo{J}$ for any~$J \subseteq \RdescentSet(x)$.
\end{proof}

\begin{remark}
Proposition~\ref{prop:singletons}\,\eqref{item:fromSingletonstoFacialSingletons} can be interpreted as follows: if the weak order minimum or maximum element of a coset~$xW_I$ is a $\equiv$-singleton, then the coset~$xW_I$ is a facial $\Equiv$-singleton. In fact, we conjecture that a coset is a facial $\Equiv$-singleton if it contains a $\equiv$-singleton.
\end{remark}


\subsection{Root and weight inversion sets for facial congruence classes}
\label{subsec:rootWeightSetCongruenceClasses}

As in Section~\ref{subsec:rootWeightSet}, we now introduce and study the root and weight inversion sets of the congruence classes of~$\Equiv$. Root inversion sets are then used to obtain equivalent characterizations of the quotient lattice of the facial weak order by~$\Equiv$. Weight inversion sets are used later in Section~\ref{subsec:congruenceFans} to describe all faces of N.~Reading's simplicial fan~$\Fan_\equiv$ associated to~$\equiv$.

\begin{definition}
The \defn{root inversion set}~$\rootSet(\Gamma)$ and the \defn{weight inversion set}~$\weightSet(\Gamma)$ of a congruence class~$\Gamma$ of~$\Equiv$ are defined by
\[
\rootSet(\Gamma) = \bigcap_{zW_K \in \Gamma} \rootSet(zW_K)
\qquad\text{and}\qquad
\weightSet(\Gamma) = \bigcup_{zW_K \in \Gamma} \weightSet(zW_K).
\]
\end{definition}

\begin{proposition}
\label{prop:propertiesRootSetCongruenceClasses}
Consider a congruence class~$\Gamma = [xW_I, yW_J]$ of~$\Equiv$.
\begin{enumerate}[(i)]
\item
\label{item:polarLatticeCongruence}
The cones generated by the root and weight inversion sets of~$\Gamma$ are polar to each other:
\[
\cone(\rootSet(\Gamma))\polar = \cone(\weightSet(\Gamma)).
\]

\item
\label{item:rootSetCongruenceClass}
The positive and negative parts of the root inversion set of~$\Gamma$ coincide with that of~$xW_I$ and~$yW_J$:
\[
\rootSet(\Gamma) \cap \Phi^+ = \rootSet(xW_I) \cap \Phi^+
\quad\text{and}\quad
\rootSet(\Gamma) \cap \Phi^- = \rootSet(yW_J) \cap \Phi^-.
\]

\item
\label{item:rootSetWeightSetFromBottomAndTopLatticeCongruence}
The root and weight inversion sets of~$\Gamma$ can be computed from that of~$xW_I$ and~$yW_J$ by
\[
\rootSet(\Gamma) = \rootSet(xW_I) \cap \rootSet(yW_J)
\quad\text{and}\quad
\weightSet(\Gamma) = \weightSet(xW_I) \cup \weightSet(yW_J).
\]
\end{enumerate}
\end{proposition}

\begin{proof}
Since the polar of a union is the intersection of the polars, \eqref{item:polarLatticeCongruence} is a direct consequence of Proposition~\ref{prop:rootSetWeightSetConesPerm}\,\eqref{item:polarCones}.

For~\eqref{item:rootSetCongruenceClass}, consider~$zW_K \in \Gamma$. Since~$xW_I \le zW_K \le yW_J$, we have by Remark~\ref{rem:equivalentRootSetCharacterization}
\[
\rootSet(xW_I) \cap \Phi^+ \subseteq \rootSet(zW_I) \cap \Phi^+
\quad\text{and}\quad
\rootSet(zW_K) \cap \Phi^- \supseteq \rootSet(yW_J) \cap \Phi^-.
\]
Therefore,
\begin{align*}
\rootSet(\Gamma) \cap \Phi^+ & = \bigcap_{zW_K \in \Gamma} \rootSet(zW_I) \cap \Phi^+ = \rootSet(xW_I) \cap \Phi^+, \\
\text{and}\qquad
\rootSet(\Gamma) \cap \Phi^- & = \bigcap_{zW_K \in \Gamma} \rootSet(zW_I) \cap \Phi^- = \rootSet(yW_J) \cap \Phi^-.
\end{align*}

Finally, for~\eqref{item:rootSetWeightSetFromBottomAndTopLatticeCongruence}, we have already~$\rootSet(\Gamma) \subseteq \rootSet(xW_I) \cap \rootSet(yW_J)$. For the other inclusion, we have
\begin{align*}
& \rootSet(xW_I) \cap \rootSet(yW_J) \cap \Phi^+ \subseteq \rootSet(xW_I) \cap \Phi^+ = \rootSet(\Gamma) \cap \Phi^+ \subseteq \rootSet(\Gamma), \\
\text{and}\qquad
& \rootSet(xW_I) \cap \rootSet(yW_J) \cap \Phi^- \subseteq \rootSet(yW_J) \cap \Phi^- = \rootSet(\Gamma) \cap \Phi^- \subseteq \rootSet(\Gamma).
\end{align*}
The equality on weights then follows by polarity.
\end{proof}

The following theorem is an analogue of Theorem~\ref{thm:facialWeakOrderCharacterizations}. It provides characterizations of the quotient lattice of the facial weak order by~$\Equiv$ in terms of root inversion sets of the congruence classes, and of comparisons of the minimal and maximal elements in the congruence classes.

\begin{theorem}
\label{thm:latticeQuotientCharacterizations}
The following assertions are equivalent for two congruence classes $\Gamma = [xW_I,yW_J]$ and~$\Gamma' = [x'W_{I'}, y'W_{J'}]$ of~$\Equiv$:
\begin{enumerate}[(i)]
\item
\label{item:smallerLatticeQuotient}
$\Gamma \le \Gamma'$ in the quotient of the facial weak order by~$\Equiv$,

\item
\label{item:latticeQuotientCharacterizationComparexWIx'WI'}
$xW_I \le x'W_{I'}$,

\item
\label{item:latticeQuotientCharacterizationCompareyWJy'WJ'}
$yW_J \le y'W_{J'}$,

\item
\label{item:latticeQuotientCharacterizationComparexWIy'WJ'}
$xW_I \le y'W_{J'}$,

\item
\label{item:latticeQuotientCharacterizationCompareMinMax}
$x \le y'$ and~$x\wo{I} \le y'\wo{J'}$,

\item
\label{item:latticeQuotientCharacterizationRoots}
$\rootSet(\Gamma) \ssm \rootSet(\Gamma') \subseteq \Phi^-$ and~$\rootSet(\Gamma') \ssm \rootSet(\Gamma) \subseteq \Phi^+$,

\item
\label{item:latticeQuotientCharacterizationRootsExtend}
$\rootSet(\Gamma) \cap \Phi^+ \, \subseteq \, \rootSet(\Gamma') \cap \Phi^+$ and~$\rootSet(\Gamma) \cap \Phi^- \, \supseteq \, \rootSet(\Gamma') \cap \Phi^-$.

\end{enumerate}
\end{theorem}

\begin{proof}
By definition, we have~$\Gamma \le \Gamma'$ in the quotient lattice if and only if there exists~${zW_K \in \Gamma}$ and~$z'W_{K'} \in \Gamma'$ such that~$zW_K \le z'W_{K'}$. Therefore, any of Conditions~\eqref{item:latticeQuotientCharacterizationComparexWIx'WI'}, \eqref{item:latticeQuotientCharacterizationCompareyWJy'WJ'}, and \eqref{item:latticeQuotientCharacterizationComparexWIy'WJ'} implies \eqref{item:smallerLatticeQuotient}. Reciprocally, since~$\ProjDown(zW_K) = xW_I$ and $\ProjDown(z'W_{K'}) = x'W_{I'}$, and~$\ProjDown$ is order preserving, we get that~\eqref{item:smallerLatticeQuotient} implies~\eqref{item:latticeQuotientCharacterizationComparexWIx'WI'}. Similarly, since~${\ProjUp(zW_K) = yW_J}$, ${\ProjUp(z'W_{K'}) = y'W_{J'}}$, and~$\ProjUp$ is order preserving, we get that~\eqref{item:smallerLatticeQuotient} implies~\eqref{item:latticeQuotientCharacterizationCompareyWJy'WJ'}. Since~$xW_I \le yW_J$ and~$x'W_{I'} \le y'W_{J'}$, either of~\eqref{item:latticeQuotientCharacterizationComparexWIx'WI'} and \eqref{item:latticeQuotientCharacterizationCompareyWJy'WJ'} implies~\eqref{item:latticeQuotientCharacterizationComparexWIy'WJ'}. Moreover~\eqref{item:latticeQuotientCharacterizationComparexWIy'WJ'}$\iff$\eqref{item:latticeQuotientCharacterizationCompareMinMax} by Theorem~\ref{thm:facialWeakOrderCharacterizations}. We thus already obtained that~\eqref{item:smallerLatticeQuotient}$\iff$\eqref{item:latticeQuotientCharacterizationComparexWIx'WI'}$\iff$\eqref{item:latticeQuotientCharacterizationCompareyWJy'WJ'}$\iff$\eqref{item:latticeQuotientCharacterizationComparexWIy'WJ'}$\iff$\eqref{item:latticeQuotientCharacterizationCompareMinMax}.

We now prove that~\eqref{item:smallerLatticeQuotient}$\iff$\eqref{item:latticeQuotientCharacterizationRootsExtend}. Assume first that~$\Gamma \le \Gamma'$. Since~\eqref{item:smallerLatticeQuotient} implies~\eqref{item:latticeQuotientCharacterizationComparexWIx'WI'} and~\eqref{item:latticeQuotientCharacterizationCompareyWJy'WJ'}, we have~${xW_I \le x'W_{I'}}$ and ${yW_J \le y'W_{J'}}$. By Remark~\ref{rem:equivalentRootSetCharacterization} and Proposition~\ref{prop:propertiesRootSetCongruenceClasses}\,\eqref{item:rootSetCongruenceClass}, we obtain
\begin{align*}
& \rootSet(\Gamma) \cap \Phi^+ = \rootSet(xW_I) \cap \Phi^+ \subseteq \rootSet(x'W_{I'}) \cap \Phi^+ = \rootSet(\Gamma') \cap \Phi^+, \\
\text{and}\qquad
& \rootSet(\Gamma) \cap \Phi^- = \rootSet(yW_J) \cap \Phi^- \supseteq \rootSet(y'W_{J'}) \cap \Phi^- = \rootSet(\Gamma') \cap \Phi^-.
\end{align*}
Reciprocally, assume that~\eqref{item:latticeQuotientCharacterizationRootsExtend} holds. By Proposition~\ref{prop:propertiesRootSetCongruenceClasses}\,\eqref{item:rootSetCongruenceClass}, we have
\[
\rootSet(xW_I) \cap \Phi^+ \, \subseteq \, \rootSet(x'W_{I'}) \cap \Phi^+
\quad\text{and}\quad
\rootSet(yW_J) \cap \Phi^- \, \supseteq \, \rootSet(y'W_{J'}) \cap \Phi^-.
\]
Since~$xW_I \le yW_J$ and~$x'W_{I'} \le y'W_{J'}$, we obtain by Remark~\ref{rem:equivalentRootSetCharacterization} that
\begin{align*}
& \rootSet(xW_I) \cap \Phi^+ \, \subseteq \, \rootSet(x'W_{I'}) \cap \Phi^+ \, \subseteq \, \rootSet(y'W_{J'}) \cap \Phi^+ \\
\text{and}\qquad
& \rootSet(xW_I) \cap \Phi^- \, \subseteq \, \rootSet(yW_J) \cap \Phi^- \, \supseteq \, \rootSet(y'W_{J'}) \cap \Phi^-.
\end{align*}
Again by Remark~\ref{rem:equivalentRootSetCharacterization}, we obtain that~$xW_I \le y'W_{J'}$, and thus that~$\Gamma \le \Gamma'$ since~\eqref{item:latticeQuotientCharacterizationComparexWIy'WJ'} implies~\eqref{item:smallerLatticeQuotient}. This proves that~\eqref{item:smallerLatticeQuotient}$\iff$\eqref{item:latticeQuotientCharacterizationRootsExtend}.

This concludes the proof as the equivalence \eqref{item:latticeQuotientCharacterizationRootsExtend}$\iff$\eqref{item:latticeQuotientCharacterizationRoots} is immediate.
\end{proof}


\subsection{Congruences and fans}
\label{subsec:congruenceFans}

Consider a lattice congruence~$\equiv$ of the weak order and the corresponding congruence~$\Equiv$ of the facial weak order. N.~Reading proved in~\cite{Reading-HopfAlgebras} that~$\equiv$ naturally defines a complete simplicial fan which coarsens the Coxeter fan. Namely, for each congruence class~$\gamma$ of~$\equiv$, consider the cone~$C_\gamma$ obtained by glueing the maximal chambers~$\cone(x(\nabla))$ of the Coxeter fan corresponding to the elements~$x$ in~$\gamma$. It turns out that each of these cones~$C_\gamma$ is convex and that the collection of cones~$\set{C_\gamma}{\gamma \in W/{\equiv}}$, together with all their faces, form a complete simplicial fan which we denote by~$\Fan_\equiv$.

We now use the congruence~$\Equiv$ of the facial weak order to describe all cones of~$\Fan_\equiv$ (not only the maximal ones). This shows that the lattice structure on the maximal faces of~$\Fan_\equiv$ extends to a lattice structure on all faces of the fan~$\Fan_\equiv$. Our description relies on the weight inversion sets defined in the previous section.

\begin{proposition}
\label{prop:maximalConesCongruence}
For a congruence class~$\gamma$ of~$\equiv$ and the corresponding congruence class~$\Gamma$ of~$\Equiv$ such that~$\gamma = W \cap \Gamma$ (see Corollary~\ref{coro:restrictCongruence}), we have
\[
C_\gamma = \bigcup_{x \in \gamma} \cone(x(\nabla)) = \cone(\weightSet(\Gamma)).
\]
\end{proposition}

\begin{proof}
We have
\[
C_\gamma = \bigcup_{x \in \gamma} \cone(x(\nabla)) = \bigcup_{x \in W \cap \Gamma} \cone(\weightSet(x)) = \bigcup_{xW_I \in \Gamma} \cone(\weightSet(xW_I)) = \cone(\weightSet(\Gamma)).
\]
\end{proof}

\begin{theorem}
\label{theo:allFacesCongruenceFan}
The collection of cones~$\set{\cone(\weightSet(\Gamma))}{\Gamma \in \CoxeterComplex{W}/{\Equiv}}$ forms the complete simplicial fan~$\Fan_\equiv$.
\end{theorem}

\begin{proof}
Denote by~$\cC$ the collection of cones~$\set{\cone(\weightSet(\Gamma))}{\Gamma \in \CoxeterComplex{W}/{\Equiv}}$. The relative interiors of the cones of~$\cC$ form a partition of the ambiant space~$V$, since~$\Equiv$ is a congruence of the Coxeter complex~$\CoxeterComplex{W}$. Similarly, the relative interiors of the cones of~$\Fan_\equiv$ form a partition of the ambiant space~$V$ since we already know that~$\Fan_\equiv$ is a complete simplicial fan~\cite{Reading-HopfAlgebras}. Therefore, we only have to prove that each cone of~$\Fan_\equiv$ is a cone of~$\cC$. First, Proposition~\ref{prop:maximalConesCongruence} ensures that the full-dimensional cones of~$\cC$ are precisely the full-dimensional cones of~$\Fan_\equiv$. Consider now another cone~$F$ of~$\Fan_\equiv$, and let~$C$ and~$C'$ be the minimal and maximal full-dimensional cones of~$\Fan_\equiv$ containing~$F$ (in the order given by~$\le/{\equiv}$). Since~$C$ and~$C'$ are full-dimensional cones of~$\Fan_\equiv$, there exist congruence classes~$\Gamma$ and~$\Gamma'$ of~$\Equiv$ such that~$C = \cone(\weightSet(\Gamma))$ and~$C' = \cone(\weightSet(\Gamma'))$. One easily checks that the Coxeter cones contained in the relative interior of~$F$ are precisely the cones~$\cone(\weightSet(xW_I))$ for the cosets~$xW_I$ such that~$x \in \Gamma$ while $x\wo{I} \in \Gamma'$. By Proposition~\ref{prop:characterizationEquiv}, these cosets form a congruence class~$\Omega$ of~$\Equiv$. It follows that~$F = \cone(\weightSet(\Omega)) \in \cC$, thus concluding the proof.
\end{proof}

\begin{corollary}
A coset~$xW_I$ is a facial $\Equiv$-singleton if an only if~$\cone(\weightSet(xW_I))$ is a cone of~$\Fan_\equiv$.
\end{corollary}


\subsection{Two examples: Facial boolean and Cambrian lattices}
\label{subsec:facialCambrianLatticesFacialCube}

To illustrate the results in this section, we revisit two relevant families of lattice congruences of the weak order, namely the descent congruence and the Cambrian congruences~\cite{Reading-CambrianLattices}.


\subsubsection{Facial boolean lattices}

\enlargethispage{-.5cm}
The \defn{descent congruence} is the congruence of the weak order defined by~$x \equivdes y$ if and only if~$\LdescentSet(x) = \LdescentSet(y)$. The corresponding up and down projections are given by~$\projDown(x) = \wo{\LdescentSet(x)}$ and~$\projUp(x) = \woo\wo{S \ssm \LdescentSet(x)}$. The quotient of the weak order by~$\equivdes$ is isomorphic to the boolean lattice on~$S$. The fan~$\Fan_\equivdes$ is given by the arrangement of the hyperplanes orthogonal to the simple roots of~$\Delta$. It is the normal fan of the parallelepiped~$\Para(W)$ generated by the simple roots~of~$\Delta$.

Denote by~$\Equivdes$ the facial weak order congruence induced by~$\equivdes$ as defined in Section~\ref{subsec:congruencesFacialWeakOrder}. According to Theorem~\ref{theo:allFacesCongruenceFan}, the $\Equivdes$ congruence classes correspond to all faces of the parallelepiped~$\Para(W)$.

\begin{figure}
\DeclareDocumentCommand{\rds}{ O{1.1cm} O{->} m m O{0}} {
	\def \radius {#1}
	\def \inputPoints{#3}
	\def \excludeRoots{#4}
	\def \style {#2}
	\def \initialRotation {#5}

	\pgfmathtruncatemacro{\points}{\inputPoints * 2}
	\pgfmathtruncatemacro{\inputPointsp}{\inputPoints + 1}
	\pgfmathsetmacro{\degrees}{360 / \points}
	
	\coordinate (0) at (0,0);
	
	\foreach \x in {1,...,\points}{%
		\pgfmathsetmacro{\location}{(\points+(\x-1))*\degrees + \initialRotation}
		\coordinate (\x) at (\location:\radius);
	}

	\ifthenelse{\equal{\excludeRoots}{}}{
		\foreach \x in {1, \inputPoints, \inputPointsp, \points}{%
			\draw[\style, ultra thick] (0) -- (\x);
		}
	}{
		\foreach \x in {1, \inputPoints, \inputPointsp, \points}{%
			\edef \showPoint {1};

			\foreach \y in \excludeRoots {
				\ifthenelse{\equal{\x}{\y}}{
					\xdef \showPoint {0};
				}{}
			}
			
			\ifthenelse{\equal{\showPoint}{1}}{
				\draw[->, ultra thick] (0) -- (\x);
			}{}
		}
	}  
}

\centerline{
	\begin{tikzpicture}
		[scale=2,
		aface/.style={color=blue},
		bface/.style={color=red},
		face/.style={color=red!50!blue}
		]
		\begin{scope}[shift={(-1.8,0)}]
    		\coordinate (e) at (0,0.42);
    		\coordinate (s) at (-1,1);
    		\coordinate (t) at (1,1);
    		\coordinate (st) at (-1,2);
    		\coordinate (ts) at (1,2);
    		\coordinate (sts) at (0,2.58);
    		%
    		\coordinate (Ws) at (-0.47,0.69);
    		\coordinate (Wt) at (0.47,0.69);
    		\coordinate (tWs) at (1,1.5);
    		\coordinate (sWt) at (-1,1.5);
    		\coordinate (stWs) at (-0.47,2.31);
    		\coordinate (tsWt) at (0.47,2.31);
    		%
    		\coordinate (W) at (0,1.5);
    		%
    		\draw[red, thick, double distance=5pt, cap=round] (e) -- (e);
    		\draw[red, thick, double distance=5pt, cap=round] (Ws) -- (Ws);
    		\draw[red, thick, double distance=5pt, cap=round] (Wt) -- (Wt);
    		\draw[red, thick, double distance=5pt, cap=round] (s) -- (st);
    		\draw[red, thick, double distance=5pt, cap=round] (t) -- (ts);
    		\draw[red, thick, double distance=5pt, cap=round] (stWs) -- (stWs);
    		\draw[red, thick, double distance=5pt, cap=round] (tsWt) -- (tsWt);		
    		\draw[red, thick, double distance=5pt, cap=round] (sts) -- (sts);
    		\draw[red, thick, double distance=5pt, cap=round] (W) -- (W);		
    		\draw (e) -- (s);
    		\draw (e) -- (t);
    		\draw (s) -- (st);
    		\draw (t) -- (ts);
    		\draw (st) -- (sts);
    		\draw (ts) -- (sts);
    		%
    		\draw (Ws) -- (W);
    		\draw (Wt) -- (W);
    		\draw (W) -- (stWs);
    		\draw (W) -- (tsWt);
    		%
    		\node[below] at (e) {$e$};
    		\node[below=1mm, left] at (s) {$s$};
    		\node[below=1mm, right] at (t) {$t$};
    		\node[above=1mm, left] at (st) {$st$};
    		\node[above=1mm, right] at (ts) {$ts$};
    		\node[above] at (sts) {$sts$};
    		\node[below=1.5mm, left] at (Ws) {$W_{s}$};
    		\node[below=1.5mm, right] at (Wt) {$W_{t}$};
    		\node[right] at (tWs) {$tW_{s}$};
    		\node[left] at (sWt) {$sW_{t}$};
    		\node[above=1.5mm, left] at (stWs) {$stW_{s}$};
    		\node[above=1.5mm, right] at (tsWt) {$tsW_{t}$};
    		\node[right=1mm] at (W) {$W$};
		\end{scope}
		\begin{scope}[shift={(1.8,0)}]	
    		\coordinate (st) at (0,0.42);
    		\coordinate (St) at (-1.9,1.5);
    		\coordinate (sT) at (1.9,1.5);
    		\coordinate (ST) at (0,2.58);
    		%
    		\coordinate (sSt) at (-.95,.96);
    		\coordinate (stT) at (.95,.96);
    		\coordinate (StT) at (-.95,2.04);
    		\coordinate (sST) at (.95,2.04);
    		%
    		\coordinate (sStT) at (0,1.5);
    		%
    		\draw (st) -- (St);
    		\draw (st) -- (sT);
    		\draw (St) -- (ST);
    		\draw (sT) -- (ST);
    		%
    		\draw (sSt) -- (sST);
    		\draw (stT) -- (StT);
    		%
    		\node[below] at (st) {$\{e\}$};
    		\node[right=2mm] at (St) {$\{s, sW_{t}, st\}$};
    		\node[left=2mm] at (sT) {$\{t, tW_{s}, ts\}$};
    		\node[above] at (ST) {$\{sts\}$};
    		\node[below=1.5mm, left] at (sSt) {$\{W_{s}\}$};
    		\node[below=1.5mm, right] at (stT) {$\{W_{t}\}$};
    		\node[above=1.5mm, left] at (StT) {$\{stW_{s}\}$};
    		\node[above=1.5mm, right] at (sST) {$\{tsW_{t}\}$};
    		\node[above=1.5mm] at (sStT) {$\{W\}$};
		\end{scope}
	\end{tikzpicture}
}
	\caption{The descent congruence classes of the standard parabolic cosets in type~$A_2$ (left) and the resulting quotient (right).}
	\label{fig:A2Cube}
\end{figure}
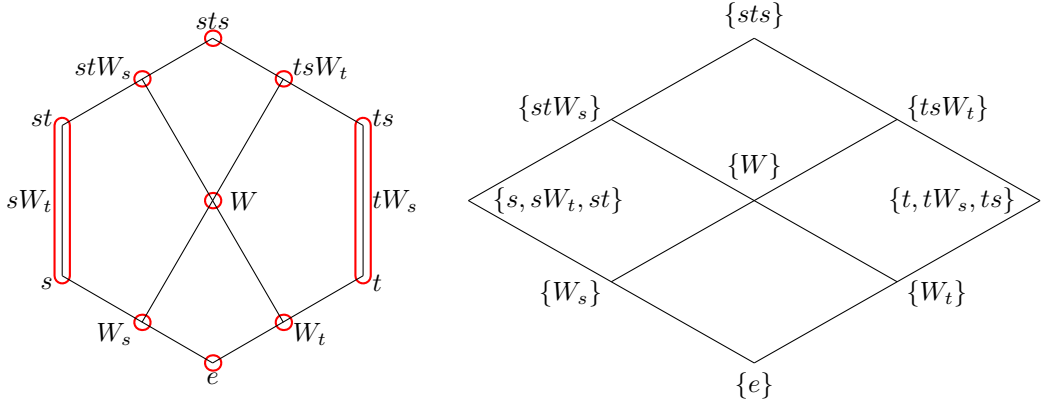

In the next few statements, we provide a direct criterion to test whether two cosets are $\Equivdes$-congruent. For this, we need to extend to all cosets the notion of descent sets.

\begin{definition}
Let the (left) \defn{root descent set} of a coset~$xW_I$ be the set of roots
\[
\rootDescentSet(xW_I) \eqdef \rootSet(xW_I) \cap \pm\Delta \; \subseteq \; \Phi.
\]
\end{definition}

\fref{fig:rootDescentSet2} illustrates the root descent sets in type~$A_2$ (left) and~$A_3$ (right). For the latter, we have just discarded the interior triangles in each root inversion set in \fref{fig:rootSet3}.

\begin{figure}
	\input{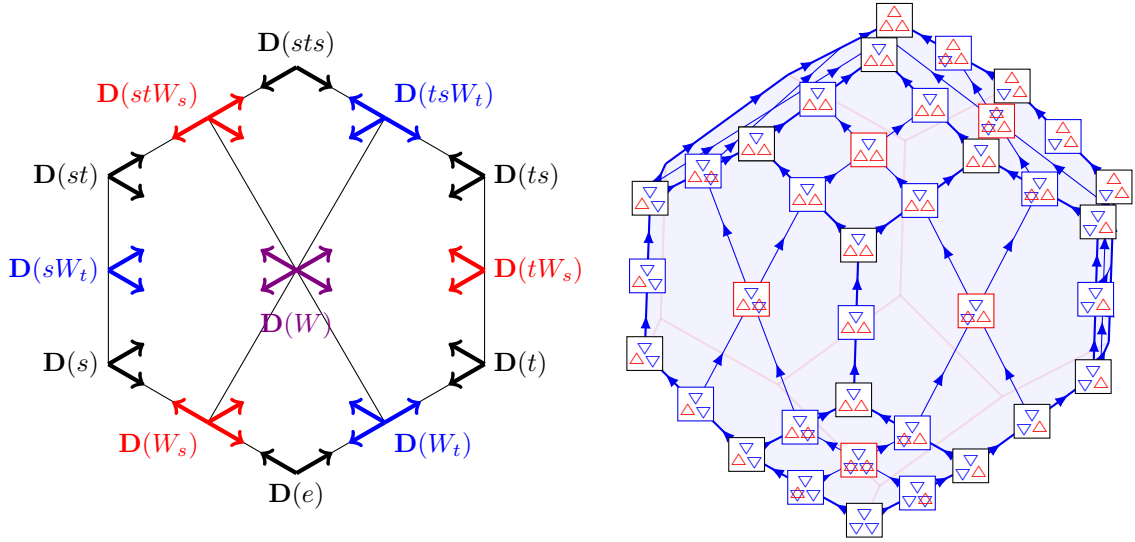}
	\caption{The root descent sets of the standard parabolic cosets in type~$A_2$ (left) and $A_3$ (right).}
	\label{fig:rootDescentSet2}
\end{figure}

Notice that the simple roots in the inversion set~$\inversionSet(x)$ precisely correspond to the descent set~$\LdescentSet(x)$:
\[
\Delta \cap \inversionSet(x) = \set{\alpha_s}{s \in \LdescentSet(x)} = \set{\alpha_s}{s \in S, \, \length(sx) < \length(x)}.
\]
Similar to Proposition~\ref{prop:rootSetSingleElement}, the next statement concerns the root descent set~$\rootDescentSet(xW_\varnothing)$ for~$x \in W$. For brevity we write~$\rootDescentSet(x)$ instead of~$\rootDescentSet(xW_\varnothing)$.

\begin{proposition}
\label{prop:rootDescentSetSingleElement}
For any~$x \in W$, the root descent set~$\rootDescentSet(x)$ has the following properties.
\begin{enumerate}[(i)]
\item
\label{item:rootDescentSetSingleElementFromInversionSet}
$\rootDescentSet(x) = \big( \Delta \cap \inversionSet(x) \big) \cup - (\Delta \ssm \inversionSet(x))$. In other words,
\[
\rootDescentSet(x) \cap \Phi^+ = \big( \Delta \cap \inversionSet(x) \big)
\qquad\text{and}\qquad
\rootDescentSet(x) \cap \Phi^- = - \big( \Delta \ssm \inversionSet(x) \big).
\]

\item
\label{item:rootDescentSetSingleElementwo}
$\rootDescentSet(x \woo) = -\rootDescentSet(x)$ and $\rootDescentSet(\woo x) = \woo \big( \rootDescentSet(x) \big)$.
\end{enumerate}
\end{proposition}

\begin{proof}
The results follow immediately from Proposition~\ref{prop:rootSetSingleElement} by intersecting with~$\pm\Delta$ appropriately.
\end{proof}

As in Proposition~\ref{prop:rootSetWeightSetFromBottomAndTop} and Corollary~\ref{coro:rootSetWeightSetFromBottomAndTop}, the root descent set of a coset~$xW_I$ can be computed from that of its minimal and maximal length representatives~$x$ and~$x\wo{I}$.

\begin{proposition} 
\label{prop:rootDescentSetFromBottomAndTop}
The root and weight inversion sets of~$xW_I$ can be computed from those of~$x$ and~$x\wo{I}$ by
$\rootDescentSet(xW_I) = \rootDescentSet(x) \cup \rootDescentSet(x\wo{I})$. In other words, 
\[
\rootDescentSet(xW_I) \cap \Phi^- = \rootDescentSet(x) \cap \Phi^-
\qquad\text{and}\qquad
\rootDescentSet(xW_I) \cap \Phi^+ = \rootDescentSet(x\wo{I}) \cap \Phi^+.
\]
\end{proposition}

\begin{proof}
Follow immediately from Proposition~\ref{prop:rootSetWeightSetFromBottomAndTop} and Corollary~\ref{coro:rootSetWeightSetFromBottomAndTop} by intersecting with~$\pm\Delta$ appropriately.
\end{proof}

From the previous propositions, we obtain that the $\Equivdes$-equivalence class of~$xW_I$ is determined by the root descent set~$\rootDescentSet(xW_I)$.

\begin{proposition}
For any cosets~$xW_I, yW_J$, we have~$xW_I \Equivdes yW_J$ if and only if~$\rootDescentSet(xW_I) = \rootDescentSet(yW_J)$.
\end{proposition}

\begin{proof}
As observed in Proposition~\ref{prop:characterizationEquiv}, the $\Equivdes$-congruence class of~$xW_I$ only depends on the $\equivdes$-congruence class of~$x$ and~$x\wo{I}$, and thus on the descent sets~$\LdescentSet(x)$ and~$\LdescentSet(x\wo{I})$. By Propositions~\ref{prop:rootDescentSetSingleElement}\,\eqref{item:rootDescentSetSingleElementFromInversionSet} and~\ref{prop:rootDescentSetFromBottomAndTop}, the root descent set~$\rootDescentSet(xW_I)$ and the descent sets~$\LdescentSet(x)$ and~$\LdescentSet(x\wo{I})$ determine each other. We conclude that the $\Equivdes$-equivalence class of~$xW_I$ is determined by the root descent set~$\rootDescentSet(xW_I)$.
\end{proof}

Finally, we observe that the facial $\Equivdes$-singletons correspond to the bottom and top faces of the $W$-permutahedron.

\begin{proposition}
A coset~$xW_I$ is a facial $\Equivdes$-singleton if and only if~$x = e$ or~$x\wo{I} = \woo$.
\end{proposition}

\begin{proof}
As already mentioned, the up and down projection maps of the descent congruence are given by~$\projDown(x) = \wo{\LdescentSet(x)}$ and~$\projUp(x) = \woo\wo{S \ssm \LdescentSet(x)}$. From Proposition~\ref{prop:singletons}, we therefore obtain that a coset~$xW_I$ is a singleton if and only if~$\woo\wo{S \ssm \LdescentSet(x)} = x$ and~$\wo{\LdescentSet(x\wo{I})} = x\wo{I}$. The result follows.
\end{proof}


\begin{example}
In type~$A$, the \defn{descent vector} of an ordered partition~$\lambda$ of~$[n]$ is the vector~$\des(\lambda) \in \{-1,0,1\}^{n-1}$ given by
\[
\des(\lambda)_i =
\begin{cases}
-1 & \text{if } \lambda^{-1}(i) < \lambda^{-1}(i+1), \\
0 & \text{if } \lambda^{-1}(i) = \lambda^{-1}(i+1), \\
1 & \text{if } \lambda^{-1}(i) > \lambda^{-1}(i+1). \\
\end{cases}
\]
These descent vectors where used by J.-C.~Novelli and J.-Y.~Thibon in~\cite{NovelliThibon-trialgebras} to see that the facial weak order on the cube is a lattice. See also~\cite{ChatelPilaud}.
\end{example}


\subsubsection{Facial Cambrian lattices}
Fix a \defn{Coxeter element}~$c$, \ie the product of all simple reflections in~$S$ in an arbitrary order. A simple reflection~$s \in S$ is \defn{initial} in~$c$ if~$\length(sc) < \length(c)$. For~$s$ initial in~$c$, note that $scs$ is another Coxeter element for~$W$ while $sc$ is a Coxeter element for~$W_{S \ssm \{s\}}$.

In~\cite{Reading-CambrianLattices, Reading-sortableElements}, N.~Reading defines the $c$-Cambrian lattice as a lattice quotient of the weak order (by the $c$-Cambrian congruence) or as a sublattice of the weak order (induced by $c$-sortable elements). There are several ways to present his constructions, we choose to start from the projection maps of the $c$-Cambrian congruence (as we did in the previous sections). These maps are defined by an induction both on the length of the elements and on the rank of the underlying Coxeter group. Namely, define the projection~$\projDown^c : W \to W$ inductively by~$\projDown^c(e) = e$ and for any~$s$ initial in~$c$,
\[
\projDown^{c}(w) =
\begin{cases}
s \cdot \projDown^{scs}(sw) & \text{if } \length(sw) < \length(w) \\
\projDown^{sc}(w_{\left<s\right>}) & \text{if } \length(sw) > \length(w),
\end{cases}
\]
where~$w = w_{\left<s\right>} \cdot {^{\left<s\right>}w}$ is the unique factorization of~$w$ such that~$w_{\left<s\right>} \in W_{S \ssm \{s\}}$ and ${\length(t{^{\left<s\right>}w}) > \length(^{\left<s\right>}w)}$ for all~$t \in S \ssm \{s\}$. The projection~$\projUp_c : W \to W$ can then be defined similarly, or by
\[
\projUp_{c}(w) = \big( \projDown^{(c^{-1})}(w\woo) \big) \woo.
\]
N.~Reading proves in~\cite{Reading-sortableElements} that these projection maps~$\projUp_c$ and~$\projDown^c$ satisfy the properties of Lemma~\ref{lem:conditionsProjectionMaps} and therefore define a congruence~$\equivc$ of the weak order called \defn{$c$-Cambrian congruence}. The quotient of the weak order by the $c$-Cambrian congruence is called the \defn{$c$-Cambrian lattice}. It was also defined as the smallest congruence contracting certain edges, see~\cite{Reading-CambrianLattices}.

Cambrian congruences are relevant in the context of finite type cluster algebras, generalized associahedra, and $W$-Catalan combinatorics. Without details, let us point out the following facts:
\begin{enumerate}[(i)]
\item The fan~$\Fan_{\equivc}$ associated to the $c$-Cambrian congruence~$\equivc$ is the \defn{Cambrian fan} studied by N.~Reading and D.~Speyer~\cite{ReadingSpeyer}. Is was proved to be the normal fan of a polytope by C.~Hohlweg, C.~Lange and H.~Thomas~\cite{HohlwegLangeThomas}. See also~\cite{Stella, PilaudStump-brickPolytope} for further geometric properties. The resulting polytopes are called \defn{generalized associahedra}.

\item These polytopes realize the \defn{$c$-cluster complexes} of type~$W$. When~$W$ is crystallographic, these complexes were defined from the theory of finite type \defn{cluster algebras} of S.~Fomin and A.~Zelevinsky~\cite{FominZelevinsky-ClusterAlgebrasI, FominZelevinsky-ClusterAlgebrasII}.

\item The minimal elements in the $c$-Cambrian congruence classes are precisely the \defn{$c$-sortable elements}, defined as the elements~$w \in W$ such that there exists nested subsets~$K_1 \supseteq K_2 \supseteq \dots \supseteq K_r$ of~$S$ such that~$w = c_{K_1} c_{K_2} \dots c_{K_r}$ where~$c_K$ is the product of the elements in~$K$ in the order given by~$c$. The maximal elements of the $c$-Cambrian congruence classes are the \defn{$c$-antisortable elements}, defined as the elements~$w \in W$ such that~$w\woo$ is $c^{-1}$-sortable. N.~Reading proved in~\cite{Reading-sortableElements} that the Cambrian lattice is in fact isomorphic to the sublattice of the weak order induced by $c$-sortable elements (or by $c$-antisortable elements). The $c$-sortable elements are connected to various $W$-Catalan families: $c$-clusters, vertices of the $c$-associahedron, $W$-non-crossing partitions. See~\cite{Reading-coxeterSortable} for precise definitions.
\end{enumerate}

The results presented in this paper translate to the following statement.

\begin{theorem}\label{thm:FacialCambrian}
For any Coxeter element~$c$ of~$W$, the \defn{facial $c$-Cambrian congruence}~$\Equivc$ on the Coxeter complex~$\CoxeterComplex{W}$, defined by
\[
xW_I \Equivc yW_J \iff x \equivc y \text{ and } x\wo{I} \equivc y\wo{J},
\]
has the following properties:
\begin{enumerate}[(i)]
\item
\label{item:restrictFacialCambrianCongruence}
The $c$-Cambrian congruence~$\equivc$ is the restriction of the facial $c$-Cambrian congruence~$\Equivc$ to~$W$.

\item
\label{item:geometryFacialCambrianCongruence}
The quotient of the facial weak order by the facial $c$-Cambrian congruence~$\Equivc$ defines a lattice structure on the cones of the $c$-Cambrian fan of~\cite{ReadingSpeyer}, or equivalently on the faces of the $c$-associahedron of~\cite{HohlwegLangeThomas}.

\item
\label{item:sortableSingletonsFacialCambrianCongruence}
A coset~$xW_I$ is minimal (resp.~maximal) in its facial $c$-congruence class if and only if~$x\wo{I}$ is $c$-sortable (resp.~$x$ is $c$-antisortable). In particular, a Coxeter cone~$\cone(\weightSet(xW_I))$ is a cone of the $c$-Cambrian fan if and only if~$x$ is $c$-antisortable and $x\wo{I}$ is $c$-sortable.
\end{enumerate}
\end{theorem}

\begin{proof}
\eqref{item:restrictFacialCambrianCongruence} is an application of Corollary~\ref{coro:restrictCongruence}. \eqref{item:sortableSingletonsFacialCambrianCongruence} follows from Theorem~\ref{theo:allFacesCongruenceFan} and the fact that the $c$-Cambrian fan of~\cite{ReadingSpeyer} is the normal fan of the $c$-associahedron of~\cite{HohlwegLangeThomas}. Finally, \eqref{item:sortableSingletonsFacialCambrianCongruence} is a direct translation of Propositions~\ref{prop:sortable} and~\ref{prop:singletons}.
\end{proof}

\begin{example}
Examples of facial Cambrian congruences in type~$A_2$, $A_3$, and~$B_3$ are represented in Figures~\ref{fig:A2Cambrian}, \ref{fig:A3Cambrian}, and~\ref{fig:B3Cambrian} respectively.

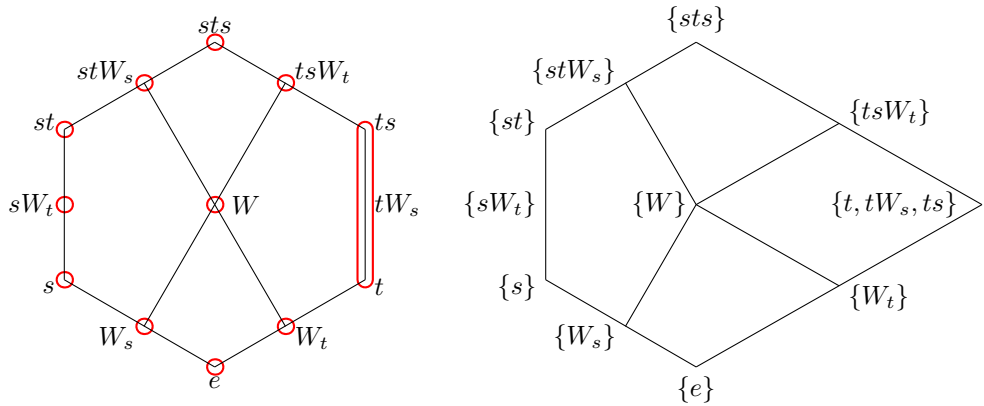
\begin{figure}
\DeclareDocumentCommand{\rds}{ O{1.1cm} O{->} m m O{0}} {
	\def \radius {#1}
	\def \inputPoints{#3}
	\def \excludeRoots{#4}
	\def \style {#2}
	\def \initialRotation {#5}

	\pgfmathtruncatemacro{\points}{\inputPoints * 2}
	\pgfmathtruncatemacro{\inputPointsp}{\inputPoints + 1}
	\pgfmathsetmacro{\degrees}{360 / \points}
	
	\coordinate (0) at (0,0);
	
	\foreach \x in {1,...,\points}{%
		\pgfmathsetmacro{\location}{(\points+(\x-1))*\degrees + \initialRotation}
		\coordinate (\x) at (\location:\radius);
	}

	\ifthenelse{\equal{\excludeRoots}{}}{
		\foreach \x in {1, \inputPoints, \inputPointsp, \points}{%
			\draw[\style, ultra thick] (0) -- (\x);
		}
	}{
		\foreach \x in {1, \inputPoints, \inputPointsp, \points}{%
			\edef \showPoint {1};

			\foreach \y in \excludeRoots {
				\ifthenelse{\equal{\x}{\y}}{
					\xdef \showPoint {0};
				}{}
			}
			
			\ifthenelse{\equal{\showPoint}{1}}{
				\draw[->, ultra thick] (0) -- (\x);
			}{}
		}
	}  
}

\centerline{
	\begin{tikzpicture}
		[scale=2,
		aface/.style={color=blue},
		bface/.style={color=red},
		face/.style={color=red!50!blue}
		]
		\begin{scope}[shift={(-1.8,0)}]
    		\coordinate (e) at (0,0.42);
    		\coordinate (s) at (-1,1);
    		\coordinate (t) at (1,1);
    		\coordinate (st) at (-1,2);
    		\coordinate (ts) at (1,2);
    		\coordinate (sts) at (0,2.58);
    		%
    		\coordinate (Ws) at (-0.47,0.69);
    		\coordinate (Wt) at (0.47,0.69);
    		\coordinate (tWs) at (1,1.5);
    		\coordinate (sWt) at (-1,1.5);
    		\coordinate (stWs) at (-0.47,2.31);
    		\coordinate (tsWt) at (0.47,2.31);
    		%
    		\coordinate (W) at (0,1.5);
    		%
    		\draw[red, thick, double distance=5pt, cap=round] (e) -- (e);
    		\draw[red, thick, double distance=5pt, cap=round] (Ws) -- (Ws);
    		\draw[red, thick, double distance=5pt, cap=round] (Wt) -- (Wt);
    		\draw[red, thick, double distance=5pt, cap=round] (s) -- (s);
    		\draw[red, thick, double distance=5pt, cap=round] (sWt) -- (sWt);
    		\draw[red, thick, double distance=5pt, cap=round] (st) -- (st);
    		\draw[red, thick, double distance=5pt, cap=round] (t) -- (ts);
    		\draw[red, thick, double distance=5pt, cap=round] (stWs) -- (stWs);
    		\draw[red, thick, double distance=5pt, cap=round] (tsWt) -- (tsWt);		
    		\draw[red, thick, double distance=5pt, cap=round] (sts) -- (sts);
    		\draw[red, thick, double distance=5pt, cap=round] (W) -- (W);		
    		\draw (e) -- (s);
    		\draw (e) -- (t);
    		\draw (s) -- (st);
    		\draw (t) -- (ts);
    		\draw (st) -- (sts);
    		\draw (ts) -- (sts);
    		%
    		\draw (Ws) -- (W);
    		\draw (Wt) -- (W);
    		\draw (W) -- (stWs);
    		\draw (W) -- (tsWt);
    		%
    		\node[below] at (e) {$e$};
    		\node[below=1mm, left] at (s) {$s$};
    		\node[below=1mm, right] at (t) {$t$};
    		\node[above=1mm, left] at (st) {$st$};
    		\node[above=1mm, right] at (ts) {$ts$};
    		\node[above] at (sts) {$sts$};
    		\node[below=1.5mm, left] at (Ws) {$W_{s}$};
    		\node[below=1.5mm, right] at (Wt) {$W_{t}$};
    		\node[right] at (tWs) {$tW_{s}$};
    		\node[left] at (sWt) {$sW_{t}$};
    		\node[above=1.5mm, left] at (stWs) {$stW_{s}$};
    		\node[above=1.5mm, right] at (tsWt) {$tsW_{t}$};
    		\node[right=1mm] at (W) {$W$};
		\end{scope}
		\begin{scope}[shift={(1.4,0)}]
    		\coordinate (e) at (0,0.42);
    		\coordinate (s) at (-1,1);
    		\coordinate (st) at (-1,2);
    		\coordinate (sts) at (0,2.58);
    		%
    		\coordinate (Ws) at (-0.47,0.69);
    		\coordinate (Wt) at (.95,.96);
    		\coordinate (tWs) at (1.9,1.5);
    		\coordinate (sWt) at (-1,1.5);
    		\coordinate (stWs) at (-0.47,2.31);
    		\coordinate (tsWt) at (.95,2.04);
    		%
    		\coordinate (W) at (0,1.5);
    		%
    		\draw (e) -- (s);
    		\draw (e) -- (tWs);
    		\draw (s) -- (st);
    		\draw (st) -- (sts);
    		\draw (tWs) -- (sts);
    		%
    		\draw (Ws) -- (W);
    		\draw (Wt) -- (W);
    		\draw (W) -- (stWs);
    		\draw (W) -- (tsWt);
    		%
    		\node[below] at (e) {$\{e\}$};
    		\node[below=1mm, left] at (s) {$\{s\}$};
    		\node[above=1mm, left] at (st) {$\{st\}$};
    		\node[above] at (sts) {$\{sts\}$};
    		\node[below=1.5mm, left] at (Ws) {$\{W_{s}\}$};
    		\node[below=1.5mm, right] at (Wt) {$\{W_{t}\}$};
    		\node[left=2mm] at (tWs) {$\{t, tW_{s}, ts\}$};
    		\node[left] at (sWt) {$\{sW_{t}\}$};
    		\node[above=1.5mm, left] at (stWs) {$\{stW_{s}\}$};
    		\node[above=1.5mm, right] at (tsWt) {$\{tsW_{t}\}$};
    		\node[left] at (W) {$\{W\}$};
		\end{scope}
	\end{tikzpicture}
}
	\caption{The $st$-Cambrian congruence classes of the standard parabolic cosets in type~$A_2$ (left) and the resulting quotient (right).}
	\label{fig:A2Cambrian}
\end{figure}

\hvFloat[floatPos=p, capWidth=h, capPos=r, capAngle=90, objectAngle=90, capVPos=c, objectPos=c]{figure}
{	\begin{tikzpicture}%
	[x={(0.767968cm, 0.559570cm)},
	y={(-0.407418cm, 0.802202cm)},
	z={(0.494203cm, -0.208215cm)},
	scale=3.30000,
    edge/.style={color=blue!95!black}, 
    bedge/.style={color=red!95!black, opacity=.3}, 
	cedge/.style={color=green!80!black, opacity=.4, thin}, 
    edgeEq/.style={color=green!55!black, line width=1mm, cap=round}, 
    bedgeEq/.style={color=green!55!black, line width=1mm, cap=round, opacity=.5}, 
    facet/.style={fill=blue,fill opacity=0.100000}, 
    face/.style={color=blue!85!blue}, 
    bface/.style={color=red!85!red, opacity=.3}, 
    vertex/.style={inner sep=1pt, circle, anchor=base, color=blue!75!black}, 
    bvertex/.style={inner sep=1pt, circle, anchor=base, color=red!75!black, opacity=.3}] 
	%
	\coordinate (r) at (-2.12, -0.408, 0.577);
	\coordinate (rs) at (-2.12, 0.408, -0.577);
	\coordinate (e) at (-1.41, -1.63, 0.577);
	\coordinate (srs) at (-1.41, 0.0000200, -1.73);
	\coordinate (rt) at (-1.41, 0.0000300, 1.73);
	\coordinate (rst) at (-1.41, 1.63, -0.577);
	\coordinate (s) at (-0.707, -2.04, -0.577);
	\coordinate (sr) at (-0.707, -1.22, -1.73);
	\coordinate (t) at (-0.707, -1.22, 1.73);
	\coordinate (rts) at (-0.707, 1.22, 1.73);
	\coordinate (rtst) at (-0.707, 2.04, 0.577);
	\coordinate (srst) at (-0.707, 1.22, -1.73);
	\coordinate (srt) at (0.707, -1.22, -1.73);
	\coordinate (rtsrt) at (0.707, 2.04, 0.577);
	\coordinate (st) at (0.707, -2.04, -0.577);
	\coordinate (ts) at (0.707, -1.22, 1.73);
	\coordinate (srtst) at (0.707, 1.22, -1.73);
	\coordinate (rtsr) at (0.707, 1.22, 1.73);
	\coordinate (stsr) at (2.12, -0.408, 0.577);
	\coordinate (sts) at (1.41, -1.63, 0.577);
	\coordinate (srts) at (1.41, 0.0000200, -1.73);
	\coordinate (tsr) at (1.41, 0.0000300, 1.73);
	\coordinate (srtsrt) at (1.41, 1.63, -0.577);
	\coordinate (srtsr) at (2.12, 0.408, -0.577);
	\coordinate (W) at (0, 0, 0);
	%
	\coordinate (rt-A) at (-1.41, 3.27, 4.039);
	\coordinate (r-A) at (-3.54,2.036,0.577);
	\coordinate (t-A) at (0.699, -0.4, 4.039);
	\coordinate (srs-A) at (-2.121, 1.22, -1.73);
	\coordinate (sts-A) at (2.12,-1.224,1.731);
	%
	\draw[bedge] (rs) -- (srs) node (rsWr) [midway]{}; 
	\draw[bedge] (e) -- (s) node (Ws) [midway]{}; 
	\draw[bedge] (sr) -- (srs) node (srWs) [midway]{}; 
	\draw[bedgeEq] (srs) -- (srst) node (srsWt) [midway]{}; 
	\draw[bedge] (s) -- (sr) node (sWr) [midway]{}; 
	\draw[bedge] (s) -- (st) node (sWt) [midway]{}; 
	\draw[bedge] (sr) -- (srt) node (srWt) [midway]{}; 
	\draw[bedge] (st) -- (srt) node (stWr) [midway]{}; 
	\draw[bedge] (srt) -- (srts) node (srtWs) [midway]{}; 
	\draw[bedge] (st) -- (sts) node (stWs) [midway]{}; 
	\draw[bedge] (srts) -- (srtst) node (srtsWt) [midway]{}; 
	\draw[bedge] (srts) -- (srtsr) node (srtsWr) [midway]{}; 
	\draw[bedge] (e) -- (r) node (Wr) [midway] {};
	\draw[bedge] (r) -- (rs) node (rWs) [midway] {};
	\draw[bedge] (rs) -- (rst) node (rsWt) [midway] {};
	\draw[bedge] (rst) -- (srst) node (rstWr) [midway] {};
	\draw[bedge] (srst) -- (srtst) node (srstWs) [midway] {};
	\draw[bedge] (srtst) -- (srtsrt) node (srtstWr) [midway] {};
	\draw[bedge] (e) -- (t) node (Wt) [midway] {};
	\draw[bedge] (t) -- (ts) node (tWs) [midway] {};
	\draw[bedge] (ts) -- (sts) node (tsWt) [midway] {};
	\draw[bedge] (sts) -- (stsr) node (stsWr) [midway] {};
	\draw[bedge] (stsr) -- (srtsr) node (stsrWs) [midway] {};
	\draw[bedge] (srtsr) -- (srtsrt) node (srtsrWt) [midway] {};
	\draw[bedge] (Wr.center) -- (srWs.center);
	\draw[bedge] (Ws.center) -- (rsWr.center) node (Wrs) [midway]{}; 
	\draw[bedge] (Ws.center) -- (tsWt.center);
	\draw[bedge] (Wt.center) -- (stWs.center) node (Wst) [midway]{}; 
	\draw[bedge] (sWr.center) -- (stWr.center);
	\draw[bedge] (sWt.center) -- (srWt.center) node (sWrt) [midway]{}; 
	\draw[bedge] (stWr.center) -- (stsrWs.center);
	\draw[bedge] (stWs.center) -- (srtsWr.center) node (stWrs) [midway]{}; 
	\draw[bedge] (srWs.center) -- (srtsWt.center);
	\draw[bedge] (srWt.center) -- (srstWs.center) node (srWst) [midway]{}; 
	\draw[bedgeEq] (rsWr.center) -- (rstWr.center);
	\draw[bedge] (rsWt.center) -- (srsWt.center) node (rsWrt) [midway]{}; 
	\draw[bedge] (srtsWt.center) -- (srtsrWt.center);
	\draw[bedge] (srtsWr.center) -- (srtstWr.center) node (srtsWrt) [midway]{}; 
	\node[bvertex] at (srt){}; 
	\node[bvertex] at (st){}; 
	\node[bvertex] at (srts){}; 
	\node[bvertex] at (srs){}; 
	\node[bvertex, above] at (s) {$s$}; 
	\node[bvertex] at (sr){}; 
	%
	\draw[edge] (e) -- (r) node (Wr) [midway]{}; 
	\draw[edgeEq] (r) -- (rs) node (rWs) [midway]{}; 
	\draw[edge] (r) -- (rt) node (rWt) [midway]{}; 
	\draw[edgeEq] (rs) -- (rst) node (rsWt) [midway]{}; 
	\draw[edge] (e) -- (t) node (Wt) [midway]{}; 
	\draw[edge] (t) -- (rt) node (tWr) [midway]{}; 
	\draw[edgeEq] (rt) -- (rts) node (rtWs) [midway]{}; 
	\draw[edge] (rst) -- (rtst) node (rstWs) [midway]{}; 
	\draw[edge] (rst) -- (srst) node (rstWr) [midway]{}; 
	\draw[edgeEq] (t) -- (ts) node (tWs) [midway]{}; 
	\draw[edgeEq] (rts) -- (rtst) node (rtsWt) [midway]{}; 
	\draw[edgeEq] (rts) -- (rtsr) node (rtsWr) [midway]{}; 
	\draw[edgeEq] (rtst) -- (rtsrt) node (rtstWr) [midway]{}; 
	\draw[edge] (srst) -- (srtst) node (srstWs) [midway]{}; 
	\draw[edgeEq] (rtsr) -- (rtsrt) node (rtsrWt) [midway]{}; 
	\draw[edge] (rtsrt) -- (srtsrt) node (rtsrtWs) [midway]{}; 
	\draw[edge] (ts) -- (sts) node (tsWt) [midway]{}; 
	\draw[edgeEq] (ts) -- (tsr) node (tsWr) [midway]{}; 
	\draw[edge] (srtst) -- (srtsrt) node (srtstWr) [midway]{}; 
	\draw[edge] (tsr) -- (rtsr) node (tsrWs) [midway]{}; 
	\draw[edgeEq] (sts) -- (stsr) node (stsWr) [midway]{}; 
	\draw[edge] (tsr) -- (stsr) node (tsrWt) [midway]{}; 
	\draw[edge] (stsr) -- (srtsr) node (stsrWs) [midway]{}; 
	\draw[edge] (srtsr) -- (srtsrt) node (srtsrWt) [midway]{}; 
	\draw[edge] (Wr.center) -- (tWr.center);
	\draw[edge] (Wt.center) -- (rWt.center) node (Wrt) [midway]{}; 
	\draw[edgeEq] (rWt.center) -- (rstWs.center);
	\draw[edge] (rWs.center) -- (rtsWt.center) node (rWst) [midway]{}; 
	\draw[edgeEq] (tWr.center) -- (tsrWs.center);
	\draw[edge] (tWs.center) -- (rtsWr.center) node (tWrs) [midway]{}; 
	\draw[edgeEq] (rtsWt.center) -- (rtsrWt.center);
	\draw[edgeEq] (rtsWr.center) -- (rtstWr.center) node (rtsWrt) [midway]{}; 
	\draw[edge] (rstWr.center) -- (rtsrtWs.center);
	\draw[edge] (rstWs.center) -- (srtstWr.center) node (rstWrs) [midway]{}; 
	\draw[edge] (tsrWt.center) -- (rtsrtWs.center);
	\draw[edge] (tsrWs.center) -- (srtsrWt.center) node (tsrWst) [midway]{}; 
	\draw[edgeEq] (tsWt.center) -- (tsrWt.center);
	\draw[edge] (tsWr.center) -- (stsWr.center) node (tsWrt) [midway]{}; 
	\draw[cedge] (Wrs.center) -- (tsrWst.center);
	\draw[cedge] (Wrt.center) -- (srtsWrt.center);
	\draw[cedge] (Wst.center) -- (rstWrs.center);
	\fill[facet] (srtsr) -- (stsr) -- (tsr) -- (rtsr) -- (rtsrt) -- (srtsrt) -- cycle {};
	\fill[facet] (t) -- (e) -- (r) -- (rt) -- cycle {};
	\fill[facet] (rtst) -- (rst) -- (rs) -- (r) -- (rt) -- (rts) -- cycle {};
	\fill[facet] (stsr) -- (tsr) -- (ts) -- (sts) -- cycle {};
	\fill[facet] (srtsrt) -- (rtsrt) -- (rtst) -- (rst) -- (srst) -- (srtst) -- cycle {};
	\fill[facet] (rtsr) -- (rts) -- (rtst) -- (rtsrt) -- cycle {};
	\fill[facet] (tsr) -- (ts) -- (t) -- (rt) -- (rts) -- (rtsr) -- cycle {};
	\node[vertex, below=.5mm] at (r) {$r$}; 
	\node[vertex] at (rs){}; 
	\node[vertex, below=.5mm] at (e) {$e$}; 
	\node[vertex] at (rt){}; 
	\node[vertex] at (rst){}; 
	\node[vertex, below=.5mm] at (t) {$t$}; 
	\node[vertex] at (rts){}; 
	\node[vertex] at (rtst){}; 
	\node[vertex] at (srst){}; 
	\node[vertex] at (rtsrt){}; 
	\node[vertex] at (ts){}; 
	\node[vertex] at (srtst){}; 
	\node[vertex] at (rtsr){}; 
	\node[vertex] at (stsr){}; 
	\node[vertex] at (sts){}; 
	\node[vertex] at (tsr){}; 
	\node[vertex] at (srtsrt){}; 
	\node[vertex] at (srtsrt){}; 
	\node[color=green!80!black, right] at (W){}; 
	%
    \draw[thick] (rtsrt) -- (rt-A);
	\draw[thick] (r) -- (r-A);
	\draw[thick] (t) -- (t-A);
	\draw[thick] (r-A) -- (rt-A);
	\draw[thick] (t-A) -- (rt-A);
	\draw[thick] (srs-A) -- (srst);
	\draw[thick] (r-A) -- (srs-A);
	\draw[thick, opacity=.3] (srs-A) -- (srs);
	\draw[thick] (sts-A) -- (stsr);
	\draw[thick, opacity=.3] (sts-A) -- (sts);
	\draw[thick] (sts-A) -- (t-A);

\end{tikzpicture}}
{The $srt$-Cambrian congruence on the standard parabolic cosets in type~$A_3$ and the resulting quotient. Bold green edges are contracted edges. The quotient is given the geometry of the type~$A_3$ $srt$-associahedron of~\cite{HohlwegLangeThomas}.}
{fig:A3Cambrian}

\hvFloat[floatPos=p, capWidth=h, capPos=r, capAngle=90, objectAngle=90, capVPos=c, objectPos=c]{figure}
{\input{figures/B3Cambrian}}
{The $srt$-Cambrian congruence on the standard parabolic cosets in type~$B_3$ and the resulting quotient. Bold green edges are contracted edges. The quotient is given the geometry of the type~$B_3$ $srt$-associahedron of~\cite{HohlwegLangeThomas}.}
{fig:B3Cambrian}
\end{example}

\begin{example}
In type~$A$, the Tamari congruence classes correspond to binary trees, while the facial Tamari congruence classes correspond to Schr\"oder trees. The quotient of the facial weak order by the facial Tamari congruence was already described in~\cite{PalaciosRonco, NovelliThibon-trialgebras}. In~\cite{ChatelPilaud}, G.~Chatel and V.~Pilaud describe the Cambrian counterparts of binary trees and Schr\"oder trees, and use them to introduce the facial type~$A$ Cambrian lattices.
\end{example}

\begin{remark}
If~$\equiv$ is an order congruence on a poset~$(P,\le)$ with up and down projections~$\projUp$ and~$\projDown$, the suborder of~$\le$ induced by~$\projDown(P)$ is isomorphic to the quotient order~$P/{\equiv}$ (see Definition~\ref{def:latticeCongruences}). When~$P$ is a lattice, $P/{\equiv}$ is also a lattice, so that~$(\projDown(P), \le)$ is a lattice. Although~$(\projDown(P),\le)$ is always a meet subsemilattice of~$P$, it is not necessarily a sublattice of~$P$. In~\cite{Reading-sortableElements}, N.~Reading proved moreover that the weak order induced on~$\projDown(W)$ is actually a sublattice of the weak order on~$W$. In contrast, the facial weak order induced on~$\ProjDown(\CoxeterComplex{W})$ is not a sublattice of the facial weak order on~$\CoxeterComplex{W}$. An example already appears in~$A_3$ for~$c = srt$. Consider~${xW_I = tsrW_{st}}$ and~${yW_J = stsrW_s}$, so that~${xW_I \meet yW_J = \zm W_{\Km} = tsrW_t}$. We observe that~$x\wo{I} = srt | srt = \woo$ and~$y\wo{J} = srt | sr$ are $srt$-sortable while $\zm\wo{\Km} = st | sr$ is not.
\end{remark}


\bibliographystyle{alpha}
\bibliography{facialWeakOrder}
\label{sec:biblio}

\end{document}